\documentclass[12pt]{amsart}
\usepackage{hyperref} 

\hypersetup{
    colorlinks=true,   
    linkcolor=black,    
    urlcolor=black      
}

\usepackage[margin=1in]{geometry} 
\usepackage{caption}
\usepackage{microtype}
\usepackage{subcaption}
\usepackage{tikz} 
\usepackage{fourier}
\usepackage{pst-poly}
\usepackage{multido}
\usepackage{fourier}
\usetikzlibrary{shapes.geometric}
\usepackage{pgfplots}
\pgfplotsset{compat=1.15}
\usepackage{mathrsfs}
\usetikzlibrary{arrows}
\usepackage[utf8]{inputenc}
\usepackage{graphicx}
\graphicspath{{Figures/}}
\usepackage{amssymb}
\usepackage{bbm}
\usepackage{bm}
\usepackage{mathtools}
\usepackage{amsthm}
\usepackage{amsbsy}
\usepackage{amstext}
\usepackage{amscd}
\usepackage{nccmath}
\usepackage{enumerate}
\usepackage{float}
\usepackage{amsmath}
\usepackage{comment}
\allowdisplaybreaks
\newcommand{\ab}{\allowbreak}

\title[{Third order cumulants of products}] {Third order cumulants of products}

\author[arizmendi]{Octavio Arizmendi} \address{Department
  of Probability and Statistics, Centro de Investigaci\'on en Matem\'aticas, Calle Jalisco S/N. Guanajuato, Mexico}

\email{octavius@cimat.mx} 

\author[munoz]{Daniel Munoz George}
\address{Department
  of Mathematics and Statistics, The University of Hong Kong, Pok Fu Lam, Hong Kong}

\email{dmunozgeorge@gmail.com, 18dmg1@queensu.ca} 

\author[Sigarreta]{Sayle Sigarreta} \address{Department
  of Mathematics, Benem\'erita universidad aut\'onoma de puebla,
  Ciudad Universitaria, 72592 Puebla, Mexico}

\email{sayle.sigarretari@alumno.buap.mx}

\thanks{AMS Subject Classification: 60B20, 46L54, 15B52}

\newtheorem{theorem}{Theorem}[section]
\newtheorem{proposition}[theorem]{Proposition}
\newtheorem{corollary}[theorem]{Corollary}
\newtheorem{lemma}[theorem]{Lemma}

\theoremstyle{definition}
\newtheorem{definition}[theorem]{Definition}
\newtheorem{notation}[theorem]{Notation}
\newtheorem{example}[theorem]{Example}
\newtheorem{remark}[theorem]{Remark}
\newtheorem{remark and notation}[theorem]{Remark and notation}

\newcommand{\V}{\mathcal{V}}
\newcommand{\U}{\mathcal{U}}
\newcommand{\PS}{\mathcal{PS}}
\newcommand{\NC}{\mathcal{NC}}
\newcommand{\Pn}{\mathcal{P}}
\newcommand{\PSone}{\mathcal{PS}_{NC}^{(1)}(m_1,m_2,m_3)}
\newcommand{\PStwo}{\mathcal{PS}_{NC}^{(2)}(m_1,m_2,m_3)}
\newcommand{\PSthree}{\mathcal{PS}_{NC}^{(3)}(m_1,m_2,m_3)}

\begin{document}\thispagestyle{empty}

\begin{abstract}
We provide a formula for the third order free cumulants of products as entries. We apply this formula to find the third order free cumulants of various Random Matrix Ensambles including product of Ginibre Matrices and Wishart matrices, both  in the Gaussian case.
\end{abstract}

\maketitle

\section{Introduction}
The present paper is set within the framework of higher-order freeness, a generalization of Voiculescu's Free Probability Theory \cite{VDN} that arises from the study of fluctuations in large-dimensional random matrices. Introduced in \cite{CMSS}, higher-order freeness extends the concept of second-order freeness by analyzing the asymptotic behavior of classical cumulants of traces of random matrix models as the matrix dimension tends to infinity. A key feature of this approach is that, under certain independence conditions, given the distribution of variables \(X_1, X_2, \dots, X_n\), one can compute the distribution and fluctuations of polynomials in these variables. This is particularly relevant in fundamental cases such as sums or products of random variables.  

The theory originated with the work of Mingo and Nica \cite{MN}, who described second-order fluctuations of random matrices, such as Wishart and GUE matrices, in terms of annular partitions. Later, in a foundational paper, Mingo and Speicher \cite{MS} introduced second-order freeness, providing an algebraic framework analogous to that of non-commutative probability spaces. To extend this theory to higher-order fluctuations, Collins, Mingo, Speicher, and Śniady developed the notion of higher-order freeness in \cite{CMSS}, introducing higher-order cumulants as fundamental tools to characterize these fluctuations. They also established, in \cite{CMSS}, functional relationships between the generating functions of second-order cumulants and moments, as well as the precise connection of the theory with unitarily invariant matrices.  

Recent developments in Topological Recursion have sparked renewed interest in higher-order freeness \cite{BCGLS, BCG21, BDKS23, DF, Hock23, Lionni22}. In particular, Borot et al. \cite{BCGLS} derived functional equations linking higher-order cumulants and higher-order moments. Additional derivations can be found in the works of Hock \cite{Hock23} and Lionni \cite{Lionni22}.  

In the context of free independence, free cumulants are multilinear objects that provide a straightforward characterization of this notion, often referred to as the vanishing of mixed cumulants \cite[Lecture 11]{NS}. This result states that whenever variables are freely independent (i.e. first-order free), their mixed cumulants vanish.

A fundamental question in this setting is how to express the cumulants of products in terms of the cumulants of its individual variables. This provides a crucial tool for computing the moments of new examples from known ones. The first result in this direction was obtained for first-order free cumulants by Krawczyk and Speicher \cite{KS}. In the classical probability setting, a similar result was previously derived by Leonov and Shiryaev \cite{LS}, who described the behavior of classical cumulants under products of independent random variables.  

Building on these ideas, Mingo, Speicher, and Tan \cite{MST} established the second-order analog, showing how to compute second-order cumulants of products in terms of those of the individual variables. In this work, we extend these results to third-order freeness, providing an analog of the results of Krawczyk and Speicher \cite{KS} and Mingo, Speicher, and Tan \cite{MST} for third-order cumulants.

Let us present our main results more precisely. For the sake of clarity, we will remind their corresponding first and second order version. The first one, in \cite[Theorem 2.2]{KS}, shows that if $(\mathcal{A},\varphi)$ is a non-commutative probability space and $a_1,\dots,a_n \in\mathcal{A}$ is a collection of variables, then

\begin{equation}\label{Equation: First order cumulants of products}
\kappa_p(A_1,\dots,A_p)=\sum_{\pi \in \NC(n)}\kappa_n(a_1,\dots,a_n),
\end{equation}
where the summation is over $\pi$ such that $\pi \vee \gamma =1_n$, 
$$A_i= \prod_{j=n_1+\cdots + n_{i-1}+1}^{n_1+\cdots +n_i} a_j,$$ 
and $\gamma=(1,\dots,n_1)(n_1+1,\dots,n_1+n_2)\cdots (n_1+\cdots +n_{p-1}+1,\dots, n_1+\cdots +n_p)$ with $n=\sum_{i=1}^p n_i$. In \cite{MST} is it shown that the condition $\pi \vee \gamma =1_n$ is equivalent to the condition $\pi^{-1}\gamma_{n}$ separates the points of $\{n_1,n_1+n_2,\dots,n_1+\cdots+n_p\}$, with $\gamma_n=(1,\dots,n)$. Here by separates it means that no cycle of $\pi^{-1}\gamma_{n}$ contains two or more elements of $\{n_1,n_1+n_2,\dots,n_1+\cdots+n_p\}$. In \cite[Theorem 3]{MST} they provide the analogous result to Equation \ref{Equation: First order cumulants of products} for second order cumulants,

\begin{equation}\label{Equation: Second order cumulants of products}
\kappa_{r,s}(A_1,\dots,A_r,A_{r+1},\dots,A_{r+s})=\sum_{(\V,\pi) \in PS_{NC}(p,q)}\kappa_{\V,\pi}(a_1,\dots,a_{p+q}),
\end{equation}
where the summation is over $(\V,\pi)$ such that $\pi^{-1}\gamma_{p,q}$ separates the points of $\{n_1,\dots,n_1+\cdots+n_{r+s}\}$. In this case, we are given a second order probability space $(\mathcal{A},\varphi_1,\varphi_2)$ and $\kappa_{r,s}$ are their corresponding second order cumulants. The permutation $\gamma_{p,q}=(1,\dots,p)(p+1,\dots,p+q)\in S_{p+q}$ with $p=n_1+\cdots +n_r$ and $q=n_{r+1}+\cdots +n_{r+s}$. \\

The main theorem of this paper is the analogous result for third order cumulants. Let $(\mathcal{A},\varphi,\varphi_2,\varphi_3)$ be a third order non-commutative probability space and elements $a_1,\dots a_{n_1+\dots +n_{r+s+t}} \in \mathcal{A}$ where $n_1,\dots, n_{r+s+t}$ are given positive integers. For $1\leq i\leq r+s+t$, let
$$A_i= \prod_{j=n_1+\cdots + n_{i-1}+1}^{n_1+\cdots +n_i} a_j.$$ Let us remind that $\varphi_{(\V,\pi)}$ are determined by the knowledge of all

 $\varphi_{(1,\gamma_{m_1,\dots,m_r})}[a_1,\dots,a_m] \vcentcolon=$
   $$\varphi_r(a_1\dots a_{m_1};a_{m_1+1}\cdots a_{m_1+m_2};\cdots ; a_{m_1+\cdots,m_{r-1}+1}\cdots a_{m}),$$and the free cumulants $\kappa_{(\V,\pi)}$ are also determined by the values of\\

$\kappa_{(1,\gamma_{m_1,\dots,m_r})}[a_1,\dots,a_m] \vcentcolon =$ 
$$\kappa_{m_1,\dots,m_r}(a_1,\dots, a_{m_1};a_{m_1+1},\dots, a_{m_1+m_2};\dots ; a_{m_1+\cdots,m_{r-1}+1},\dots, a_{m}),$$
 since both are multiplicative functions.

\begin{theorem}[Third order cumulants with products as arguments]\label{Main theorem cumulants with products as entries}
\begin{equation}\label{me}
\kappa_{r,s,t}(A_1,\dots,A_{r+s+t}) = \sum_{(\V,\pi)\in \PS_{NC}(p,q,l)}\kappa_{(\V,\pi)}(a_1,\dots,a_{p+q+l})
\end{equation}
where the summation is over those $(\V,\pi)\in \PS_{NC}(p,q,l)$ such that $\pi^{-1}\gamma_{p,q,l}$ separates the points of $N \vcentcolon= \{n_1,n_1+n_2,\dots,n_1+\cdots +n_{r+s+t}\}$ and $p=n_1+\cdots+n_r$, $q=n_{r+1}+\cdots +n_{r+s}$ and $l=n_{r+s+1}+\cdots+n_{r+s+t}$.
\end{theorem}

Using our main result, Theorem \ref{Main theorem cumulants with products as entries}, we compute the third order fluctuation cumulants for various new examples.  

First, we consider \( s^2 \), where \( s \) is a third order semicircular operator. This corresponds to the square of a Gaussian Unitary Ensemble.  We find the rather surprising fact that third order cumulants are all $0$.

Next, we examine the product \( c a c^* \), where \( a \) is an operator that is third order free with \( c \), and \( c \) is a third order circular operator. This case corresponds to Wishart matrices with Gaussian entries and a given covariance matrix \cite{W,LW}.  

We then compute the third order cumulants of \( aa^* \), where \( a \) is a third order \( R \)-diagonal operator. Additionally, we prove that \( R \)-diagonality is preserved under multiplication by a free element.  

Finally, we calculate the third order fluctuation  cumulants and fluctuation moments for the product of $k$ third order free circular operators, \( c_1 c_2 \dots c_k \). This example is particularly relevant in random matrix theory, as it corresponds to the product of independent Ginibre matrices. The second order case was recently derived by Dartois and Forrester \cite{DF} for \( k=2 \) and by Arizmendi and Mingo \cite{AM} for the general case.  

This paper is organized as follows: In Section \ref{Section: Preliminaries of noncrossing objects}, we introduce the necessary preliminaries on non-crossing partitions, non-crossing permutations, and non-crossing partitioned permutations.  Section \ref{Section: Notation} establishes the notation that will be used throughout the paper.  In Section \ref{Section: Partial order on Sn}, we define a partial order on \( S_n \) and present additional relations that will play a key role in our proofs.  Section \ref{Section: Topology of noncrossing permutations} explores properties of non-crossing permutations, extending known results for non-crossing partitions on \( [n] \) and non-crossing permutations on an \( (m,n) \)-annulus to non-crossing permutations on an \( (m_1,\dots,m_r) \)-annulus. Section \ref{Section: Preliminary results}, provide preliminary results that will be essential for proving our main theorem.  Section \ref{Section: Main theorem proof} contains the proof of our main result. In Section \ref{Section: Application of the theorem}, we apply the main theorem to compute the third-order cumulants of various examples. Lastly, in  Appendix \ref{append}, the main combinatorial  lemmas used in the proof of the main theorem are posed and proved, for a fluent reading.

\section{Preliminaries on noncrossing partitions, permutations and partitioned permutations}\label{Section: Preliminaries of noncrossing objects}

The main objects of this paper will be the set of non-crossing partitions, permutations and partitioned permutations, each of which, for the sake of clarity, will be exhaustively explained in this section.

\subsection{The set of partitions and non-crossing permutations}

\begin{notation}
Let $M \subset \mathbb{Z}$, a partition of $M$ is a collection of sets $B_i$ whose disjoint union is $M$, we call to the sets $B_i$ \textit{the blocks of the partition}. We will denote to the set of partitions on 
$M$ by $\Pn(M)$. When $M = [n] \vcentcolon = \{1,\dots,n\}$ we recover the set of partitions on $n$ elements which we simply denote $\Pn(n)$. We can put a partial order on $\Pn(M)$, $\leq$, given by $\pi \leq \sigma$ if every block of $\pi$ is contained in a block of $\sigma$.  The suppremum of two partitions $\pi$ and $\sigma$ will be denoted by $\pi \vee \sigma$. The largest element of $\Pn(M)$, denoted $1_M$, is the partition consisting of a single block.
\end{notation}

\begin{notation}
Let $M \subset \mathbb{Z}$, a permutation on $M$ is a bijective function from $M$ to $M$, the set of permutations on $M$ will be denoted by $S_M$. For a permutation $\pi \in S_M$ we let $\#(\pi)$ be the number of cycles in the cycle decomposition of $\pi$. We put the metric on $S_M$ given by $|\pi|=|M|-\#(\pi)$. Sometimes we may regard a permutation $\pi$ as a partition $0_{\pi}$ by considering the cycles of $\pi$ as the blocks of $0_{\pi}$.  
\end{notation}

\begin{definition}{Non-crossing permutation.}\label{Notation: Definition of non-crossing permutation}
Let $M\subset \mathbb{Z}$ and let $\gamma\in S_M$ be a fixed permutation. We say that the permutation $\pi \in S_M$ is non-crossing with respect to $\gamma$ if the following satisfy
\begin{enumerate}
    \item $\pi \vee \gamma \vcentcolon = 0_{\pi} \vee 0_{\gamma} = 1_M$, and,
    \item $\#(\pi)+\#(\pi^{-1}\gamma)+\#(\gamma)=|M|+2$.
\end{enumerate}
The set of non-crossing permutations with respect to $\gamma$ will be denoted $S_{NC}(\gamma)$. When $M=[n]$ and $\gamma=(1,\dots,n_1)(n_1+1,\dots,n_1+n_2)\cdots (n_1+\cdots +n_{r-1}+1,\dots,n)$ for some $n_1,\dots,n_r \geq 1$ with $n=\sum n_i$ we rather use the notation $S_{NC}(n_1,\dots,n_r)$ to refer $S_{NC}(\gamma)$.
\end{definition}

\begin{remark}\label{Remark: The case when Gamma has one cycle}
When $\#(\gamma)=1$ the first condition in Definition \ref{Notation: Definition of non-crossing permutation} is automatically satisfied and then we are reduced to verify,
$$\#(\pi)+\#(\pi^{-1}\gamma)=|M|+1.$$
In this case we adopt the notation $\NC(M)$ instead of $S_{NC}(|M|)$. If $M=[n]$ we simply write $\NC(n)$.
\end{remark}

As pointed out in Remark \ref{Remark: The case when Gamma has one cycle}, when $M=[n]$ and $\gamma=(1,\dots,n)$ we are back in the set of non-crossing partitions, usually denoted $\NC(n)$ see \cite[Lecure 9]{NS}. Let us elaborate a bit more on this direction. For a partition $\pi \in \Pn(n)$ we say that $\pi$ has a crossing if there are $i<j<k<l$ such that $i,k$ are in the same block of $\pi$ and $j,l$ are in the same block of $\pi$ distinct from the one containing $i$ and $k$. A partition is said to be non-crossing if it has no crossings (see Figure \ref{f1}). Each partition can be seen as a permutation by considering the block of the partition as the cycle of the permutation, however this depends on the choice of the cyclic order. Bianne \cite{B}, showed that only the permutation whose cyclic order is increasing turns out to be non-crossing with respect to $\gamma$, in the sense of Definition \ref{Notation: Definition of non-crossing permutation}. This is why, when $\gamma=(1,\dots,n)\in S_n$ we do not distinguish in between non-crossing permutations with respect to $\gamma$, and non-crossing partitions on $[n]$, and both sets are usually denoted $\NC(n)$.

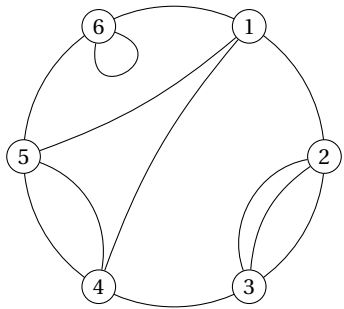
\begin{figure}[ht]
\begin{center}
\begin{tikzpicture}
\begin{scriptsize}
  \begin{scope}[shift={(0,0)}]
    \draw[] (0,0) circle (2);
    \foreach \i in {1,...,6} {
      \node[fill=white, draw, circle, inner sep=2pt] (n\i) at ({360/6 * (2-\i )}:2) {\i};
    }

     \path[-] (n1) edge[bend left=-10] (n4);
  \path[-] (n5) edge[bend left=40] (n4);
  \path[-] (n5) edge[bend left=-10] (n1);
  \path[-] (n3) edge[bend left=50] (n2);
  \path[-] (n3) edge[bend left=25] (n2);

   \path (n6) edge [out=260,in=340,looseness=8] node[above]{}(n6);
  \end{scope}
\end{scriptsize}

\end{tikzpicture}
\caption{An example of a non-crossing partition in $\NC(6)$.}
\label{f1}
\end{center}
\end{figure}

\begin{definition}\label{dm1}
Let $M \subset \mathbb{Z}$ and $M^\prime \subset M$. Let $\sigma \in S_M$. We let $\sigma|_{M^\prime}$ be the permutation in $S_{M^\prime}$ defined as,
$$\sigma|_{M^\prime}(m)=\sigma^{m^{min}}(m),$$
where $m^{min} = min\{k\geq 1: \sigma^k(m)\in M^\prime\}.$ It is well known that $\sigma^{|M|!}=id$ and therefore such a minimum exist. Moreover if $m_1,m_2\in M^\prime$ are such that $\sigma|_{M^\prime}(m_1)=\sigma|_{M^\prime}(m_2)$ then $\sigma^{m_1^{min}}(m_1)=\sigma^{m_2^{min}}(m_2)$. If $m_1^{min}=m_2^{min}$ then $m_1=m_2$, if $m_1^{min}<m_2^{min}$ then we apply the inverse function, $\sigma^{-1}$, $m_1^{min}$ times to get $m_1=\sigma^{m_2^{min}-m_1^{min}}(m_2)$, however this means that $\sigma^{m_2^{min}-m_1^{min}}(m_2) \in M^\prime$ with $0<m_2^{min}-m_1^{min}<m_2^{min}$ which is not possible, so it must be $m_1^{min}=m_2^{min}$ and $m_1=m_2$. The latter proves that $\sigma|_{M^\prime}$ is well defined.
\end{definition}

\begin{example}
    In order to illustrate how the above definition works,  
 let us take $\sigma= (1,4,5)(2,3)(6) \in S_M$ with $M=\{1,2,3,4,5,6\}$ which, in fact, is derived from the partition depicted in Figure \ref{f1}. Thus, if $M^\prime=\{1,3,6\}$ then $\sigma|_{M^\prime}=(1)(3)(6)=0_{M^\prime}$ and if $M^\prime=\{1,2,4\}$ then $\sigma|_{M^\prime}=(1,4)(2)$. Therefore, to construct $\sigma|_{M^\prime}$ it is only necessary to remove, maintaining the order, from each cycle of $\sigma$ the elements that are not in $M^\prime$.

\end{example}

\begin{remark}
    Under the notation of Definition \ref{dm1}, $\gamma_{m_1,\dots,m_l}\pi^{-1}$  separates the points of $M \subset [m_1+\dots+m_l]$ iff $\gamma_{m_1,\dots,m_l}\pi^{-1}|_{M}=id_{M}$, where $id_{M}$ is the identity permutation. Also, at this point, it's worth noting that, the condition $\pi^{-1}\gamma_{r_1,\dots,r_m}$ separates the points of $N \vcentcolon= \{n_1,n_1+n_2,\dots,n_1+\cdots +n_{r_1+\dots+r_m}\}$ is equivalent to the condition that $\gamma_{r_1,\dots,r_m}\pi^{-1}$ separates the points of
$O\vcentcolon=\{1, n_1+1, n_1+n_2+1, \ldots,$
$n_1+\cdots+n_{r_1+\dots+r_{m-1}}+1\}$ since $\pi^{-1}\gamma_{r_1,\dots,r_m}(i)=j$ iff $\gamma_{r_1,\dots,r_m}\pi^{-1}(\gamma_{r_1,\dots,r_m}(i))=\gamma_{r_1,\dots,r_m}(j)$. By the way, the last relation also implies that, $\#\pi^{-1}\gamma_{r_1,\dots,r_m}=\#\gamma_{r_1,\dots,r_m}\pi^{-1}$.
\end{remark}

\subsection{Non-crossing partitioned permutations.}\hfill \\
Recall that the higher order free cumulants are defined by the equations \cite[Definition 7.4]{CMSS},

$$\varphi_{\V,\pi}[a_1,\dots,a_m] = \sum_{(\U,\pi)\in \PS_{NC}(\V,\pi)} \kappa_{(\U,\pi)}[a_1,\dots,a_m],$$
where $\PS_{NC}(\V,\pi)$ denotes the set of $(\V,\pi)$ non-crossing partitioned permutations defined as in \cite{CMSS}. Throughout this work we will be interested in the case $(\V,\pi)=(1,\gamma_{m_1,m_2,m_3})$. Where $\gamma_{m_1,\dots,m_r}$ denotes the permutation of $S_{m}$ with cycle decomposition,
$$(1,\dots,m_1)(m_1+1,\dots m_1+m_2)\cdots (m_1+\cdots +m_{r-1}+1,\dots ,m),$$
and $m=\sum m_i$.

We have the following result \cite[Lemma 3.8]{MM} to classify $\PS_{NC}(1_m,\allowbreak \gamma_{m_1,m_2,m_3})$ which we simply denote $\PS_{NC}(m_1,m_2,m_3)$.

\begin{lemma}\label{Solutions to third order partitioned permutations}
Let $m = m_1 + m_2 + m_3$, $\gamma = \gamma_{m_1, m_2, m_3}$
and $(\mathcal{V}, \pi)$ be a partitioned permutation such that $\mathcal{V}
\vee \gamma = 1$ and
\[
|(\mathcal{V}, \pi)| + |(0_{\pi^{-1}\gamma}, \pi^{-1} \gamma)| =
|(1, \gamma)|.
\]
Then, either
\begin{enumerate}

\item
$\mathcal{V} = 0_\pi$ and $\pi \in
  S_{NC}(m_1, m_2, m_3)$;

\item
$\pi =\pi_1 \times \pi_2 \in S_{NC}(m_{i_1}, m_{i_2}) \times
  \NC(m_{i_3})$ for some permutation $(i_1, i_2, i_3)$ of $\{1,
  2, 3\}$, $\#(\mathcal{V}) = \#(\pi) - 1$, and $\mathcal{V}$ joins a cycle of
  $\pi_1$ with a cycle of $\pi_2$;

\item
$\pi = \pi_1 \times \pi_2 \times \pi_3 \in \NC(m_1) \times
  \NC(m_2) \times \NC(m_3)$, $\#(\mathcal{V}) = \#(\pi) - 2$ and,
  $\mathcal{V}$ joins a cycle of $\pi_{i_1}$ with a cycle of
  $\pi_{i_2}$ in one block and joins a cycle of $\pi_{i_2}$
  with a cycle of $\pi_{i_3}$ into another block of $\mathcal{V}$,
  with $(i_1, i_2, i_3)$ some permutation of $\{1, 2, 3\}$;

\item
$\pi = \pi_1 \times \pi_2 \times \pi_3 \in \NC(m_1) \times
  \NC(m_2) \times \NC(m_3)$, $\#(\mathcal{V}) = \#(\pi) - 2$ and,
  $\mathcal{V}$ joins a cycle of $\pi_{1}$, a cycle of $\pi_{2}$ and
  a cycle of $\pi_3$ into a single block of $\mathcal{V}$.
\end{enumerate}
In all above cases all blocks of $\mathcal{V}$ are cycles of $\pi$, except possibly by the ones which are unions of cycles of $\pi$, we refer to these cycles as the marked blocks of $\pi$.
\end{lemma}

Lemma \ref{Solutions to third order partitioned permutations} describes all possible non-crossing partitioned permutations in three circles. For the sake of clarity let us label these sets. The first is $S_{NC}(m_1, m_2,m_3)$ under the abuse of notation that $\pi$ is identified with the pair $(0_{\pi},\pi)$. Next, there is the set
\begin{multline*}
\{ (\V, \pi) \mid \pi=\pi_1\times \pi_2 \in S_{NC}(m_{i_1}, m_{i_2}) \times
\NC(m_{i_3}), |\V| = |\pi| + 1\\ \textrm{\ and\ } \V \textrm{ joins a cycle of }\pi_1 \textrm{ with a cycle of } \pi_2 \\
    \textrm{with } (i_1,i_2,i_3) \textrm{ a permutation of } \{1,2,3\}\},
\end{multline*}
which we denote by $\PSone$. Next, we have
\begin{multline*}
\{ (\V, \pi) \mid \pi \in \NC(m_1) \times \NC(m_2) \times
\NC(m_3), |\V| = |\pi| + 2, \V \vee \gamma = 1
\\ \textrm{\ and 2 blocks of\ }\V \textrm{\ each contain
  two cycles of\ }\pi\},
\end{multline*}
which we denote by $\PStwo$. Finally we
have
\begin{multline*}
\{ (\V, \pi) \mid \pi \in \NC(m_1) \times \NC(m_2) \times
\NC(m_3), |\V| = |\pi| + 2, \V \vee \gamma = 1
\\ \textrm{\ and 1 block of\ }\V \textrm{\ contain three
  cycles of\ }\pi\},
\end{multline*}
which we denote by $\PSthree$

\begin{example}
    Let $\pi=(1,2,12,9,8)(3,4)(5,10,11)(6)(7)(13,15)(14)$ and $V=\{\{1,2,12,9,8\},\{3,4\},\{5,\\10,11,13,15\},\{6\},\{7\},\{14\}\}$. Then, for $(r,s,t)=(8,4,3)$, $(\V,\pi)\in PS_{NC}^{(1)}(8,4,3)$  since $\pi =\pi_1 \times \pi_2 \in S_{NC}(8, 4) \times
  \NC(3)$, $\#(\mathcal{V}) = \#(\pi) - 1$, and $\mathcal{V}$ joins a cycle of
  $\pi_1$ with a cycle of $\pi_2$, this is
indicated by the red dotted line in the diagram shown in Figure \ref{f2}. 
\end{example}

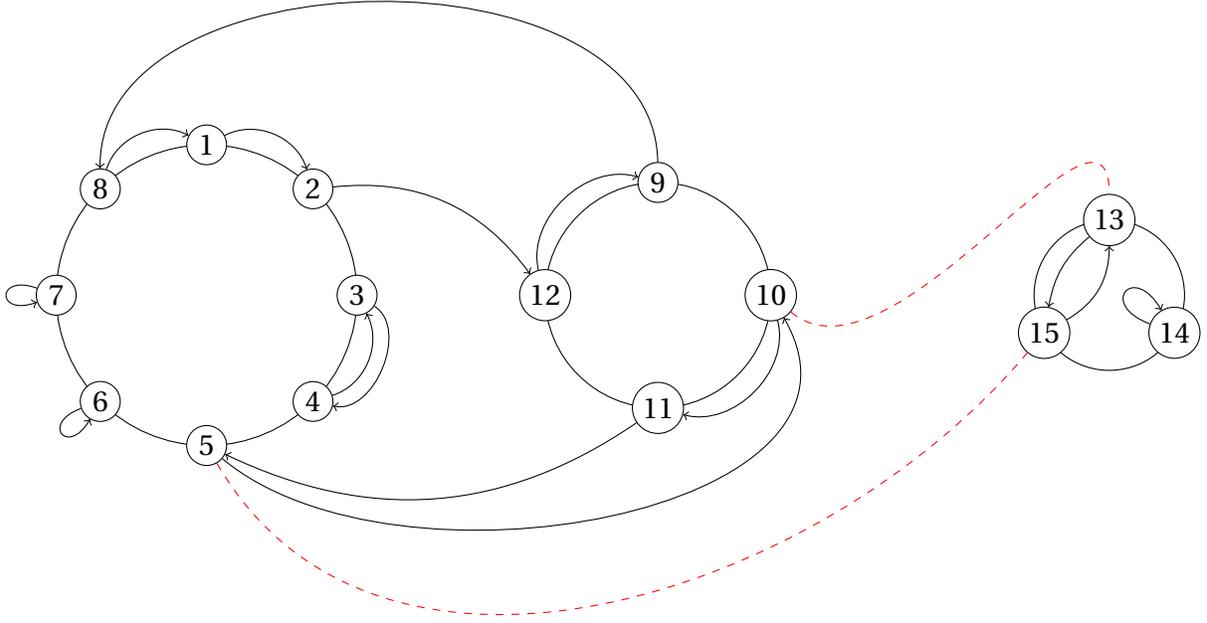
\begin{figure}[ht]
\begin{center}

\begin{tikzpicture}

  \begin{scope}[shift={(0,0)}]
    \draw[] (0,0) circle (2);
    \foreach \i in {1,...,8} {
      \node[fill=white, draw, circle, inner sep=2pt] (n\i) at ({360/8 * (3-\i )}:2) {\i};
    }
  \end{scope}

  \begin{scope}[shift={(6,0)}]
    \draw[] (0,0) circle (1.5);
    \foreach \i in {9,10,11,12} {
      \node[fill=white, draw, circle, inner sep=2pt] (n\i) at ({360/4 * (10-\i )}:1.5) {\i};
    }
  \end{scope}

  \begin{scope}[shift={(12,0)}]
    \draw[] (0,0) circle (1);
    \foreach \i/\label in {90/13, 210/15, 330/14} {
      \node[fill=white, draw, circle, inner sep=2pt] (n\i) at (\i:1) {\label};
    }
  \end{scope}
  \path[->] (n1) edge[bend left=50] (n2);
  \path[->] (n8) edge[bend left=50] (n1);
  \path[->] (n9) edge[bend left=-90] (n8);
  \path[->] (n2) edge[bend left=30] (n12);
   \path[->] (n12) edge[bend left=60] (n9);
   \path[->] (n3) edge[bend left=80] (n4);
     \path[->] (n4) edge[bend left=-50] (n3);
      \path[->] (n11) edge[bend left=30] (n5);
      \path[->] (n10) edge[bend left=60] (n11);
     \draw[->] (n5) to [out=320, in=300] (n10);
       \path[->] (n90) edge[bend left=-20] (n210);
       \path[->] (n210) edge[bend left=-30] (n90);
       \draw[color=red, dashed] (n10) to  [out=320, in=90] (n90);
        \draw[color=red, dashed] (n5) to  [out=300, in=230] (n210);

\path[->]  (n7) edge [out=160,in=200,looseness=8] node[above]{}(n7);

\path[->]  (n6) edge [out=200,in=240,looseness=8] node[above]{}(n6);
\path[->]  (n330) edge [out=160,in=120,looseness=8] node[above]{}(n330);


\end{tikzpicture}

\caption{$(\V,\pi)\in PS_{NC}^{(1)}(8,4,3)$.}
\label{f2}
\end{center}
\end{figure}

\section{Notation}\label{Section: Notation}

Let us set up the notation that we will use through  the paper. Most of the time we will adhere to notation used in \cite{MST} to make things more consistent. We let $r,s,t \in \mathbb{N}$ and let $n_1,\dots,n_{r+s+t} \in \mathbb{N}$ be a collection of integers. Let $N \vcentcolon =\{n_1,n_1+n_2,\dots,n_1+n_2+ \dots +n_{r+s+t}\}$ and let $p=n_1+\cdots +n_r$, $q=n_{r+1}+\cdots +n_{r+s}$ and $l=n_{r+s+1}+\cdots +n_{r+s+t}$. 

We let $\gamma \in S_n$ to be the permutation with cycle decomposition,
$$(1,\dots,n_1)\cdots (n_1+\cdots +n_{r+s+t-1}+1,\dots,n_{1}+\cdots+n_{r+s+t}),$$
where $n=p+q+l=\sum n_i$. We also denote by $T_i$ to the $i^{th}$ cycle of $\gamma$, i.e,
$$T_i=(n_1+\cdots +n_{i-1}+1,\dots,n_1+\cdots +n_{i}).$$

For a permutation $\pi \in S_{r+s+t}$ we let $\pi_{\vec{n}}$ be the permutation in $S_{n}$ given by,
$$\pi_{\vec{n}}(i)= \left\{ \begin{array}{lcc} \gamma(i) & if & i\notin N \\ \\ n_1+\cdots+n_{\pi(j)-1}+1 & if & i\in N,i=n_1+\cdots+n_j \text{ for some }j\in [r+s+t] \end{array} \right.$$

We let $\gamma_{r,s,t} \in S_{r+s+t}$ and $\gamma_{p,q,l} \in S_n$ be the permutations given by,
$$\gamma_{r,s,t}=(1,\dots,r)(r+1,\dots,r+s)(r+s+1,\dots,r+s+t),$$
$$\gamma_{p,q,l}=(1,\dots,p)(p+1,\dots,p+q)(p+q+1,\dots,p+q+l).$$

\begin{remark}\label{Remark: How Pin works}
Observe that $\pi_{\vec{n}}$ is a permutation whose cycle decomposition is given by the cycle decomposition of $\gamma$ after joining $T_i$ to $T_{\pi(i)}$. In other words, $\pi_{\vec{n}}$ takes the last element of $T_i$ to the first element of $T_{\pi(i)}$. This can be rephrased as saying that $\pi_{\vec{n}}^{-1}$ sends the first element of $T_i$ to the last element of  $T_{\pi^{-1}(i)}$, i.e. $\pi_{\vec{n}}^{-1}(\gamma(n_1+\cdots +n_i))=n_1+\cdots+n_{\pi^{-1}(i)}$.
\end{remark}

\begin{remark and notation}\label{Remark: Decomposition of Pin}
Observe that the cycle decomposition of $\pi_{\vec{n}}$ is determined by the cycle decomposition of $\pi$. If $\pi\in S_{r+s+t}$ has cycle decomposition $C_1\cdots C_v$, then $\pi_{\vec{n}}$ has cycle decomposition $\tilde{C_1}\cdots \tilde{C_v}$ where $\tilde{C_i}$ corresponds to the cycle of $\pi_{\vec{n}}$ obtained from merging the cycles $T_{j}$ for $j \in C_i$, that is, if the cycle of $\pi$ is $(i_1,\dots,i_w)$ the corresponding cycle of $\pi_{\vec{n}}$ is $T_{i_1}\cup T_{i_2}\cup\cdots \cup T_{i_w}$ where the union of two cycles simply means merging the cycles. 
It is clear that each cycle of $\gamma$ is contained in a cycle of $\pi_{\vec{n}}$. For a cycle, $\tilde{C_i}$, of $\pi_{\vec{n}}$ we denote by $\gamma_i$ to the restriction of $\gamma$ to $\tilde{C_i}$.
\end{remark and notation}

For a partitioned permutation, 
$$(\V,\pi)\in PS_{NC}^{(1)}(r,s,t)\cup PS_{NC}^{(2)}(r,s,t)\cup PS_{NC}^{(3)}(r,s,t),$$ 
we let $(\V_{\vec{n}},\pi_{\vec{n}})$ be the partitioned permutation in $S_n$ given as follows,
\begin{enumerate}
    \item $\pi_{\vec{n}}$ is defined as before.
    \item If $C_1,\dots, C_w$ are cycles of $\pi$ such that $0_{C_1\cup \cdots\cup C_w}$ is a block of $\V$ then we let $0_{\tilde{C_1}\cup \cdots \cup \tilde{C_w}}$ to be a block of $\V_{\vec{n}}$, 
\end{enumerate}
where $\tilde{C_i}$ is the cycle of $\pi_{\vec{n}}$ as in Remark \ref{Remark: Decomposition of Pin}.

\begin{example}
    Let $r=3$, $s=t=2$, $n_1=n_6=2$, $n_2=n_7=1$, $n_3=4$, $n_4=5$, $n_5=3$, $\pi=(1,2,3)(4,5)(6)(7)$ and $\V=\{\{1,2,3,6\},\{4,5,7\}\}$. Then $\pi_{\vec{18}}=(1,2,3,4,5,6,7)(8,9,10,11,12,13,14,15)\\
    (16,17)(18)$ and $\V_{\vec{18}}=\{\{1,2,3,4,5,6,7,16,17\},\{8,9,10,11,12,13,14,15,18\}\}$ (see Figure \ref{f3}). In this example, $(\V,\pi) \in PS_{NC}^{(2)}(3,2,2)$ while $(V_{\vec{18}},\pi_{\vec{18}}) \in PS_{NC}^{(2)}(7,8,3)$.
\end{example}

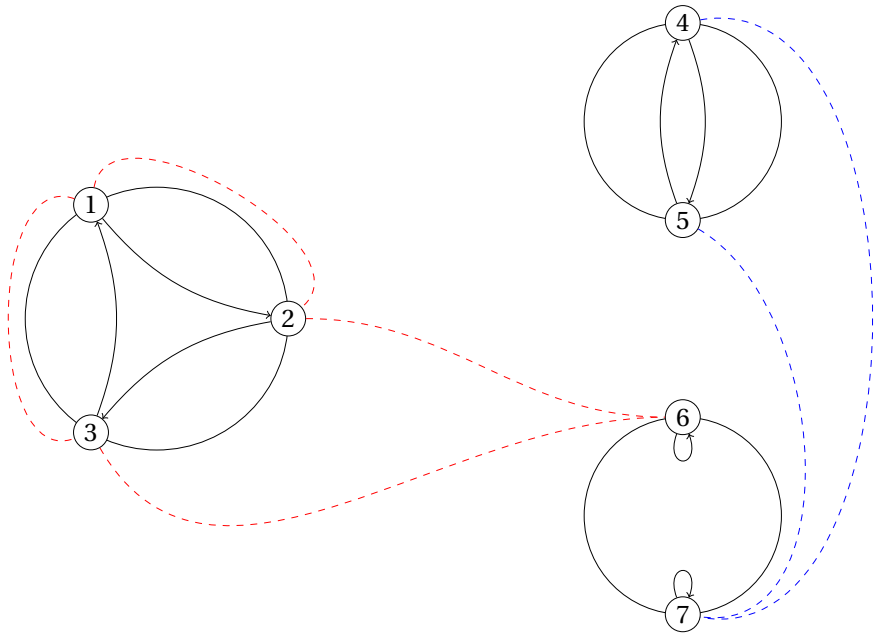
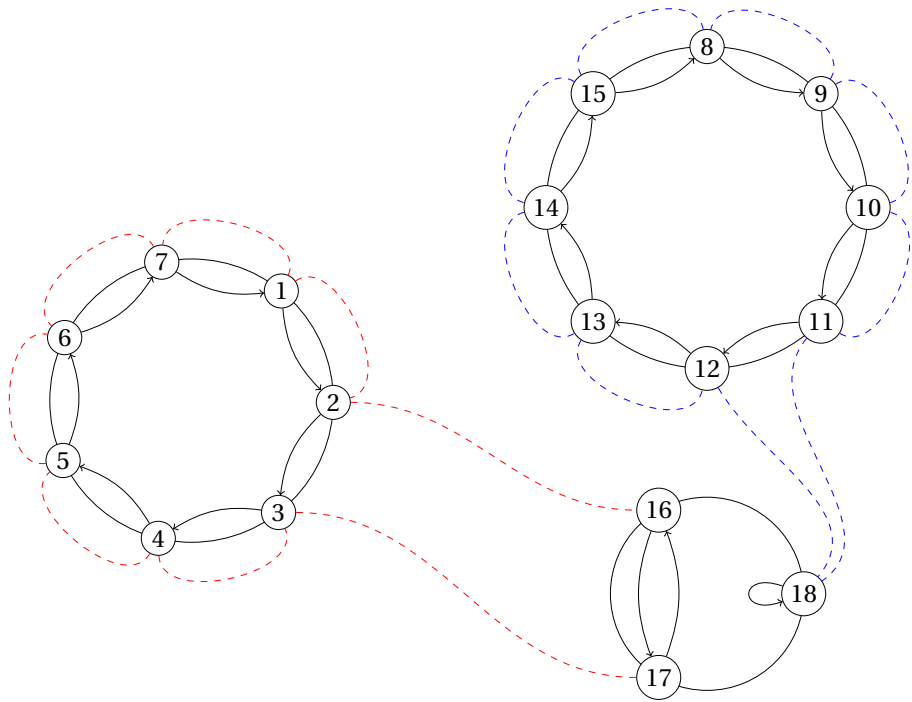
\begin{figure}[ht]
    \centering

    \begin{subfigure}{0.77\textwidth}
        \centering
        \resizebox{\textwidth}{!}{
            \begin{tikzpicture}
                
  \begin{scope}[shift={(0,0)}]
    \draw (0,0) circle (2);
    \foreach \i in {1,2,3} {
      \node[fill=white, draw, circle, inner sep=2pt] (n\i) at ({360/3 * (5-\i )}:2) {\i};
    }
  \end{scope}

   \path[->] (n3) edge[bend left=-20] (n1);
    \path[->] (n1) edge[bend left=-20] (n2);
     \path[->] (n2) edge[bend left=-20] (n3);

  \begin{scope}[shift={(8,3)}]
    \draw (0,0) circle (1.5);
    \foreach \i in {4,5} {
      \node[fill=white, draw, circle, inner sep=2pt] (n\i) at ({90* (-2*\i+9) }:1.5) {\i};
    }
  \end{scope}

    \path[->] (n4) edge[bend left=20] (n5);
     \path[->] (n5) edge[bend left=20] (n4);

  \begin{scope}[shift={(8,-3)}]
    \draw (0,0) circle (1.5);
    \foreach \i in {6,7} {
      \node[fill=white, draw, circle, inner sep=2pt] (n\i) at ({90 * (-2*(\i-2)+9 )}:1.5) {\i};
    }
  \end{scope}

     \path[->]  (n6) edge [out=250,in=290,looseness=8] node[above]{}(n6);
\path[->]  (n7) edge [out=110,in=70,looseness=8] node[above]{}(n7);

    \path[-,color=red, dashed] (n1) edge[bend left=110] (n2);
    
    \path[-,color=red, dashed] (n1) edge[bend left=-110] (n3);
 
\draw[color=red,dashed] (n2) to [out=0, in=180] (n6);
\draw[color=red,dashed] (n3) to [out=300, in=180] (n6);

\draw[color=blue,dashed] (n4) to [out=10, in=350] (n7);
\draw[color=blue,dashed] (n5) to [out=330, in=352] (n7);

            \end{tikzpicture}
        }
        \caption{$(\V,\pi).$}
       
    \end{subfigure}

    \vspace{0.5cm} 

    \begin{subfigure}{0.77\textwidth}
        \centering
        \resizebox{\textwidth}{!}{
            \begin{tikzpicture}
  \begin{scope}[shift={(0,0)}]
    \draw (0,0) circle (2.2);
    \foreach \i in {1,...,7} {
      \node[fill=white, draw, circle, inner sep=2pt] (n\i) at ({360/7 * (5-\i )+205}:2.2) {\i};
    }
  \end{scope}

    \path[->] (n1) edge[bend left=-20] (n2);
     \path[->] (n2) edge[bend left=-20] (n3);
  \path[->] (n3) edge[bend left=-20] (n4);
     \path[->] (n4) edge[bend left=-20] (n5);
       
    \path[->] (n5) edge[bend left=-20] (n6);
     \path[->] (n6) edge[bend left=-20] (n7);
  \path[->] (n7) edge[bend left=-20] (n1);

  \begin{scope}[shift={(8,3)}]
    \draw (0,0) circle (2.5);
    \foreach \i in {8,...,15} {
      \node[fill=white, draw, circle, inner sep=2pt] (n\i) at ({360/8 * (10-\i )}:2.5) {\i};}
  \end{scope}

  \path[->] (n8) edge[bend left=-20] (n9);
     \path[->] (n9) edge[bend left=-20] (n10);
  \path[->] (n10) edge[bend left=-20] (n11);
     \path[->] (n11) edge[bend left=-20] (n12);
       
    \path[->] (n12) edge[bend left=-20] (n13);
     \path[->] (n13) edge[bend left=-20] (n14);
  \path[->] (n14) edge[bend left=-20] (n15);
    \path[->] (n15) edge[bend left=-20] (n8);

      \begin{scope}[shift={(8,-3)}]
    \draw (0,0) circle (1.5);
    \foreach \i in {16,17,18} {
      \node[fill=white, draw, circle, inner sep=2pt] (n\i) at ({360/3 * ((\i-15) )}:1.5) {\i};
    }
  \end{scope}

   \path[->] (n16) edge[bend left=-20] (n17);
    \path[->] (n17) edge[bend left=-20] (n16);
     \path[->]  (n18) edge [out=160,in=200,looseness=8] node[above]{}(n18);

 \path[-,color=red, dashed] (n1) edge[bend left=100] (n2);
  \path[-,color=red, dashed] (n1) edge[bend left=-100] (n7);
   \path[-,color=red, dashed] (n7) edge[bend left=-100] (n6);
    \path[-,color=red, dashed] (n5) edge[bend left=100] (n6);
     \path[-,color=red, dashed] (n4) edge[bend left=100] (n5);
       
        \path[-,color=red, dashed] (n4) edge[bend left=-100] (n3);
\draw[color=red,dashed] (n2) to [out=0, in=180] (n16);
\draw[color=red,dashed] (n3) to [out=0, in=180] (n17);

 \path[-,color=blue, dashed] (n11) edge[bend left=-100] (n10);
  \path[-,color=blue, dashed] (n9) edge[bend left=100] (n10);
   \path[-,color=blue, dashed] (n9) edge[bend left=-100] (n8);
   \path[-,color=blue, dashed] (n8) edge[bend left=-100] (n15);
  \path[-,color=blue, dashed] (n15) edge[bend left=-100] (n14);
   \path[-,color=blue, dashed] (n14) edge[bend left=-100] (n13);
      \path[-,color=blue, dashed] (n13) edge[bend left=-100] (n12);

      \draw[color=blue,dashed] (n12) to [out=300, in=50] (n18);
\draw[color=blue,dashed] (n11) to [out=230, in=35] (n18);
            \end{tikzpicture}
        }
        \caption{($\V_{\vec{n}},\pi_{\vec{n}}$).}
        \label{fig:second}
    \end{subfigure}

    \caption{Permutations in solid lines and partitions in dashed lines.}
    \label{f3}
\end{figure}

\section{Partial order and other relations on $S_n$}\label{Section: Partial order on Sn}

We can extend the relation, $\leq$, given in \cite[Definition 17]{MST} to the set of permutations.
\begin{definition}\label{Definition: Partial order on Sn}
Let $\pi,\sigma \in S_{n}$. Suppose that each cycle of $\pi$ is contained in some cycle of $\sigma$ and for each cycle $C$ of $\sigma$ the enclosed cycles of $\pi$ form a non-crossing partition of $C$, that is, if $\pi_C$ is $\pi$ restricted to $C$ then $\pi_C \in \NC(C)$. Then we write $\pi \leq \sigma$.
\end{definition}

The next proposition shows that the relation, $\leq$ given in definition \ref{Definition: Partial order on Sn}, extends to $S_{n}$ the usual partial order on $\NC(n)$ given by inclusion of blocks.

\begin{proposition}
The relation, $\leq$, in $S_n$ given in definition \ref{Definition: Partial order on Sn} is a partial order in $S_n$.
\end{proposition}
\begin{proof}
First we show $\pi \leq \pi$ for any $\pi\in S_n$. This comes from the fact that any permutation consisting of a single cycle is always non-crossing with respect to itself. Let $\pi,\sigma \in S_n$ be such that $\pi\leq \sigma$ and $\sigma \leq \pi$. Then both have the same cycle type and each cycle $C$ of $\pi$ is also a cycle of $\sigma$ except possibly by the order of its elements. So we are reduced to prove that the order of its elements must be the same in both cycles. Let $C$ and $C^\prime$ be the cycles of $\pi$ and $\sigma$ consisting of the same elements. We know $C^\prime \in \NC(C)$, Biane showed in \cite{B} that $C^\prime$ must respect the cyclic order of $C$ and therefore $C^\prime = C$. Finally let $\sigma \leq \pi$ and $\pi \leq \rho$. It is clear that each cycle of $\sigma$ is contained in a cycle of $\rho$. Let $C$ be cycle of $\rho$ and let $\sigma|_C$ and $\pi|_C$ be $\sigma$ and $\pi$ restricted to $C$ respectively. By hypothesis each cycle of $\sigma|_C$ is contained in a cycle of $\pi|_C$ and for each cycle of $\pi|_C$ the enclosed cycles of $\sigma|_C$ form a non-crossing permutation, moreover $\pi|_C \in \NC(C)$. By Proposition \ref{Proposition: Smaller permutation of a non-crossing one is still non-crossing if connectivy is preserved} we have $\sigma|_C \in \NC(C)$ as required.
\end{proof}

\begin{lemma}\label{Lemma: Less or Equal sufficient and necessary condition}
Let $\pi,\sigma\in S_{n}$. Then $\pi\leq \sigma$ if and only if $|\pi|+|\pi^{-1}\sigma|=|\sigma|$.
\end{lemma}
\begin{proof}
Suppose $\pi\leq \sigma$, let $C_1\cdots C_w$ be the cycle decomposition of $\sigma$ and let $\pi_i$ be the restriction of $\pi$ to each $C_i$. Then,
$$\#(\pi_i)+\#(\pi_i^{-1}C_i)=|C_i|+1.$$
Summing over $i$ gives,
$$\#(\pi)+\#(\pi^{-1}\sigma)=p+q+l+\#(\sigma),$$
hence $|\pi|+|\pi^{-1}\sigma|=|\sigma|$. The converse follows directly from \cite[Lemma 8]{MST}.
\end{proof}

\begin{lemma}\label{Lemma: Less or Equal transitivity}
Let $\pi,\sigma\in S_{NC}(p,q,l)$ with $\pi\leq\sigma^{-1}\gamma_{p,q,l}$, then $\sigma \leq \gamma_{p,q,l}\pi^{-1}$.
\end{lemma}
\begin{proof}
\begin{eqnarray*}
|\sigma|+|\sigma^{-1}\gamma_{p,q,l}\pi^{-1}| &=& |\sigma|+|\pi^{-1}\sigma^{-1}\gamma_{p,q,l}| \phantom{aaa}(|\cdot|\text{{ is invariant under conjugation}}) \\
&=& |\sigma|+|\sigma^{-1}\gamma_{p,q,l}|-|\pi| \phantom{aa}(\text{{by hypothesis }}\text{{and Lemma \ref{Lemma: Less or Equal sufficient and necessary condition}}}) \\
&=& |\gamma_{p,q,l}|+4-(|\gamma_{p,q,l}|+4-|\pi^{-1}\gamma_{p,q,l}|) \\
&=& |\pi^{-1}\gamma_{p,q,l}|=|\gamma_{p,q,l}\pi^{-1}|.
\end{eqnarray*}
where in third line we use that $\pi,\sigma\in S_{NC}(p,q,l)$. Thanks to Lemma \ref{Lemma: Less or Equal sufficient and necessary condition} this concludes the proof.
\end{proof}

Let us define another relation $\lesssim^{(r)}$ on $S_n$ which might not necessarily be a partial order however it will show up later during the proof.

\begin{definition}\label{Definition: Another partial order}
For $\pi,\sigma\in S_n$ we say that $\pi \lesssim \sigma$ if for each block, $B$, of $\pi\vee\sigma$, $\pi|_B \in S_{NC}(\sigma|_B)$. If $\pi\lesssim \sigma$ and $\#(\pi\vee \sigma)=\#(\sigma)-r$ we write $\pi \lesssim^{(r)} \sigma$.
\end{definition}

\begin{remark}
In Definition $\ref{Definition: Another partial order}$ when $r=0$ we recover the partial order $\leq$ defined in \ref{Definition: Partial order on Sn}.
\end{remark}

With the relation, $\lesssim^{(r)}$, let us give very more general versions of Lemmas \ref{Lemma: Less or Equal sufficient and necessary condition} and \ref{Lemma: Less or Equal transitivity}.

\begin{proposition}\label{Proposition: Sufficient and neccesary conditions general version}
Let $\pi,\sigma\in S_n$ with $\pi\lesssim \sigma$, then,
$$|\sigma|+2(\#(\sigma)-\#(\pi\vee\sigma))=|\pi|+|\pi^{-1}\sigma|.$$
\end{proposition}
\begin{proof}
Let $B_1,\dots,B_w$ be the blocks of $\pi\vee\sigma$. At each block we have,
$$\#(\pi|_{B_i})+\#(\pi|_{B_i}^{-1}\sigma|_{B_i})+\#(\sigma|_{B_i})=|B_i|+2,$$
summing over yields,
$$\#(\pi)+\#(\pi^{-1}\sigma)+\#(\sigma)=n+2\#(\pi\vee\sigma),$$
which is equivalent to the desired expression.
\end{proof}

It is fundamental to note that in Lemma \ref{Lemma: Less or Equal transitivity} the condition that makes possible to reverse the inequality is the fact that both $\pi,\sigma \in S_{NC}(p,q,l)$. This make us to conjecture that if both connect the same circles in the sense that $\pi\vee\gamma_{p,q,l}=\sigma\vee\gamma_{p,q,l}$ and both $\pi,\sigma\lesssim \gamma_{p,q,l}$ then the same result should still holds. This motivates the following generalized version to Lemma \ref{Lemma: Less or Equal transitivity}.

\begin{lemma}\label{Lemma: Less or Equal transitivity general version when having the same type}
Let $\gamma =\gamma_{m_1,\dots,m_r} \in S_m$ be the permutation with cycle decomposition,
$$(1,\dots,m_1)\cdots (m_1+\cdots+m_{r-1}+1,\dots,m).$$
Let $\pi,\sigma \in S_m$ be such that $\pi,\sigma\lesssim \gamma$ and $\pi\vee\gamma=\sigma\vee\gamma$. If $\pi\leq \sigma^{-1}\gamma$, then, $\sigma\leq \gamma\pi^{-1}.$
\end{lemma}
\begin{proof}
By Proposition \ref{Proposition: Sufficient and neccesary conditions general version} we have,
$$|\pi|+|\pi^{-1}\gamma|=|\gamma|+2(\#(\gamma)-\#(\pi\vee\gamma))=|\gamma|+2(\#(\gamma)-\#(\sigma\vee\gamma))=|\sigma|+|\sigma^{-1}\gamma|.$$
From this point on the proof follows exactly as in Lemma \ref{Lemma: Less or Equal transitivity},
\begin{eqnarray*}
|\sigma|+|\sigma^{-1}\gamma\pi^{-1}| &=& |\sigma|+|\pi^{-1}\sigma^{-1}\gamma| \\
&=& |\sigma|+|\sigma^{-1}\gamma|-|\pi| \\
&=& |\pi|+|\pi^{-1}\gamma|-|\pi|=|\pi^{-1}\gamma|=|\gamma\pi^{-1}|.
\end{eqnarray*}
\end{proof}

\begin{remark}
One might be tempted to think that Lemma \ref{Lemma: Less or Equal transitivity general version when having the same type} still holds as long as $\#(\pi\vee\gamma)=\#(\sigma\vee\gamma)$ rather than the stronger condition $\pi\vee\gamma=\sigma\vee\gamma$. The answer is that the hypothesis $\pi\leq \sigma^{-1}\gamma$ only makes sense as long as $\pi\vee\gamma\leq \sigma\vee\gamma$. So the conditions $\#(\pi\vee\gamma)=\#(\sigma\vee\gamma)$ and $\pi\vee\gamma=\sigma\vee\gamma$ are actually equivalent under this setting. One may think further in cases where both $\pi,\sigma \lesssim \gamma$ but $\pi\vee\gamma \leq \sigma\vee\gamma$ in which the existence of such a $\pi$ that satisfies $\pi \leq \sigma^{-1}\gamma$ still makes sense. For the rest of this section let us address this question.
\end{remark}

\begin{lemma}\label{Lemma: Less or equal transitivity general version 1}
Let $\gamma =\gamma_{m_1,\dots,m_r} \in S_m$ be the permutation with cycle decomposition,
$$(1,\dots,m_1)\cdots (m_1+\cdots+m_{r-1}+1,\dots,m).$$
Let $\sigma \in S_{NC}(m_1,\dots,m_r)$ and $\pi \in S_m$ be such that $\pi \leq \sigma^{-1}\gamma$ and $\pi\lesssim \gamma$, then
$$|\sigma|+|\sigma^{-1}\gamma\pi^{-1}|=|\gamma\pi^{-1}|+2(\#(\pi\vee\gamma-1)),$$
and
$$|\gamma\pi^{-1}|+2(\#(\gamma\pi^{-1})-\#(\sigma\vee\gamma\pi^{-1}))\leq |\sigma|+|\sigma^{-1}\gamma\pi^{-1}|,$$
with equality if and only if,
$$\sigma \in \prod_{B\text{ block of }\sigma\vee\gamma\pi^{-1}}S_{NC}((\gamma\pi^{-1})|_B).$$
\end{lemma}
\begin{proof}
By Proposition \ref{Proposition: Sufficient and neccesary conditions general version},
\begin{equation}\label{aux1}
|\gamma|+2(\#(\gamma)-\#(\pi\vee\gamma))=|\pi|+|\pi^{-1}\gamma|.
\end{equation}
On the other hand, by Lemma \ref{Lemma: Less or Equal sufficient and necessary condition} $|\pi|+|\pi^{-1}\sigma^{-1}\gamma|=|\sigma^{-1}\gamma|$ as $\pi\leq \sigma^{-1}\gamma$. Thus
\begin{eqnarray*}
|\sigma|+|\sigma^{-1}\gamma\pi^{-1}| &=& |\sigma|+|\pi^{-1}\sigma^{-1}\gamma| \\
&=& |\sigma|+|\sigma^{-1}\gamma|-|\pi| \\
&=& |\gamma|+2(\#(\gamma)-1)-(|\gamma|+2(\#(\gamma)-\#(\pi\vee\gamma))-|\pi^{-1}\gamma|) \\
&=& |\pi^{-1}\gamma|+2(\#(\pi\vee\gamma)-1),
\end{eqnarray*}
where in third equality we use that $\sigma\in S_{NC}(\gamma)$ and Equation \ref{aux1}, this proves the first part. For the second part let $B_1,\dots,B_w$ be the blocks of $\sigma\vee\gamma\pi^{-1}$. \cite[Equation 2.9]{MN} says that for each block,
$$\#(\sigma|_{B_i})+\#(\sigma|_{B_i}^{-1}(\gamma\pi^{-1})|_{B_i})+\#((\gamma\pi^{-1})|_{B_i}) \leq |B_i|+2,$$
with equality if and only if $\sigma|_{B_i} \in S_{NC}((\gamma\pi^{-1})|_{B_i})$. Summing over $i$ yields,
$$\#(\sigma)+\#(\sigma^{-1}\gamma\pi^{-1})+\#(\gamma\pi^{-1})\leq m+2\#(\sigma\vee\gamma\pi^{-1}),$$
with equality if and only if $\sigma \in \prod_{B\text{ block of }\sigma\vee\gamma\pi^{-1}}S_{NC}((\gamma\pi^{-1})|_B).$ In terms of the length this is,
$$|\gamma\pi^{-1}|+2(\#(\gamma\pi^{-1})-\#(\sigma\vee\gamma\pi^{-1}))\leq |\sigma|+|\sigma^{-1}\gamma\pi^{-1}|.$$
\end{proof}

\begin{remark}
In Proposition \ref{Proposition: Sufficient and neccesary conditions general version} if $\pi \leq \sigma$ then we recover the result of Lemma \ref{Lemma: Less or Equal sufficient and necessary condition}. Likewise, in Lemma \ref{Lemma: Less or equal transitivity general version 1} when $\pi \vee\gamma =1_m$, i.e. it is non-crossing then we get $|\sigma|+|\sigma^{-1}\gamma\pi^{-1}|=|\gamma\pi^{-1}|$ which recovers the result of Lemma \ref{Lemma: Less or Equal transitivity}.
\end{remark}

The next result is a modified version to Lemma \ref{Lemma: Less or Equal transitivity} where $\pi$ doesn't necessarily meets all three circles, i.e. $\pi \vee\gamma_{p,q,l} =1_{p+q+l}$ is not necessarily satisfied, but it is still non-crossing at each block of $\pi \vee\gamma_{p,q,l}$. This will address the case $\pi,\sigma\lesssim \gamma_{p,q,l}$ but while $\sigma\vee\gamma_{p,q,l} =1$, $\pi\vee\gamma_{p,q,l}<1$. In other words, $\pi\vee\gamma_{p,q,l}<\sigma\vee\gamma_{p,q,l}$. Where by strictly less we mean less or equal but not the same partition.

\begin{corollary}\label{Corollary: Less or equal transivitiy general version}
Let $\pi\in S_{p+q+l}$ and $\sigma \in S_{NC}(p,q,l)$ with $\pi \leq \sigma^{-1}\gamma_{p,q,l}$. The following are satisfied
\begin{enumerate}
    \item If $\pi \in \NC(p)\times S_{NC}(q,l)$ then $\sigma \lesssim^{(1)} \gamma_{p,q,l}\pi^{-1}$.
    \item If $\pi \in \NC(p)\times \NC(q)\times \NC(l)$ then $\sigma \lesssim^{(2)} \gamma_{p,q,l}\pi^{-1}$.
\end{enumerate}
\end{corollary}
\begin{proof}
We prove $(1)$ first. Let $\gamma=\gamma_{p,q,l}$. We have $\gamma\pi^{-1} \in \NC(p)\times S_{NC}(q,l)$ since $\pi \in \NC(p)\times S_{NC}(q,l)$. By hypothesis $\sigma \vee\gamma =1_{p,q,l}$ and therefore there exist a cycle of $\sigma$ that must intersect $[p]$ and $[p+1,p+q+l]$, this cycle must meet more than one cycle of $\gamma\pi^{-1}$, thus, $\#(\gamma\pi^{-1})-\#(\sigma\vee\gamma\pi^{-1}) \geq 1$. Lemma \ref{Lemma: Less or equal transitivity general version 1} says,
\begin{multline*}
|\gamma\pi^{-1}|+2\leq |\gamma\pi^{-1}|+2(\#(\gamma\pi^{-1})-\#(\sigma\vee\gamma\pi^{-1})) \\
\leq |\sigma|+|\sigma^{-1}\gamma\pi^{-1}|=|\gamma\pi^{-1}|+2(\#(\pi\vee\sigma)-1)=|\gamma\pi^{-1}|+2,
\end{multline*}
so all must be equality, which means $\#(\gamma\pi^{-1})-\#(\sigma\vee\gamma\pi^{-1}) = 1$ and for each block, $C$, of $\sigma\vee\gamma\pi^{-1}$, $\sigma|_C\in S_{NC}((\gamma\pi^{-1})|_C)$, i.e. $\sigma \lesssim^{(1)} \gamma\pi^{-1}$. For $(2)$ we proceed similarly, $\#(\gamma\pi^{-1})-\#(\sigma\vee\gamma\pi^{-1}) \geq 2$, thus,
\begin{multline*}
|\gamma\pi^{-1}|+4\leq |\gamma\pi^{-1}|+2(\#(\gamma\pi^{-1})-\#(\sigma\vee\gamma\pi^{-1})) \\
\leq |\sigma|+|\sigma^{-1}\gamma\pi^{-1}|=|\gamma\pi^{-1}|+2(\#(\pi\vee\sigma)-1)=|\gamma\pi^{-1}|+4,
\end{multline*}
so all must be equality and we conclude as before.
\end{proof}

\section{Topology of non-crossing permutations}\label{Section: Topology of noncrossing permutations}

This section aims to prove some combinatorial results of non-crossing permutations that generalize some of the results in \cite{MST}, specifically \cite[Lemma 5]{MST}. In this lemma the authors proved that if we restrict a non-crossing permutation in the $(m,n)$-annulus to a subset $M$ of $[m+n]$, then it is still either non-crossing on the corresponding annulus restricted to $M$ or it is the product of two non-crossing partitions, one on each circle restricted to $M$. The tools that they use to prove it rely on the topological properties of non-crossing annular permutations obtained in \cite{MN}. Although we should expect a similar pattern on three circles and even in an arbitrary number of circles, we do not have such a tools for more than a two circles annulus. So this section is devoted to fill those gaps. At the end of the section we will prove (Lemma \ref{Lemma: Restriction to non-crossing is still non-crossing if connectivity is preserved generalized}) that the pattern we expected is indeed true for an arbitrary number of circles even though for this paper we will only need these results for the three circles case.

\begin{notation}
For this section we let $m_1,\dots,m_r \in \mathbb{Z}$ and $m=\sum m_i$. We let $\gamma_{m_1,\dots,m_r} \in S_m$ be the permutation with cycle decomposition,
$$(1,\dots,m_1)\cdots (m_1+\cdots+m_{r-1}+1,\dots,m).$$
We label the cycle $(m_1+\dots +m_{i-1}+1,\dots,m_1+\cdots +m_i)$ of $\gamma_{m_1,\dots,m_r}$ as $[m_i]$.
\end{notation}

\begin{proposition}\label{Proposition: Smaller permutation of a non-crossing one is still non-crossing if connectivy is preserved}
Let $\pi \in S_{NC}(m_1,\dots,m_r)$ and let $\sigma\in S_m$ be such that $\sigma \leq \pi$. If $\sigma \vee \gamma_{m_1,\dots,m_r} =1_m$ then $\sigma \in S_{NC}(m_1,\dots,m_r)$.
\end{proposition}
\begin{proof}
Set $\gamma=\gamma_{m_1,\dots,m_r}. $Recall \cite[Equation 2.9]{MN},
\begin{eqnarray*}
\#(\pi^{-1}\gamma)+\#(\sigma^{-1}\pi)+\#(\sigma^{-1}\gamma) \leq m+2\#(\sigma^{-1}\pi \vee \sigma^{-1}\gamma) \\ \leq n+2\#(\sigma^{-1}\gamma).
\end{eqnarray*}
On the other hand, for each cycle, $C_i$ of $\pi$ we have $\#(\sigma|_{C_i})+\#(\sigma|_{C_i}^{-1}C_i)=|C_i|+1$, summing over $i$ gives $\#(\sigma)+\#(\sigma^{-1}\pi)=n+\#(\pi)$. Hence
\begin{eqnarray*}
\#(\sigma^{-1}\gamma) & \geq & \pi^{-1}\gamma + \#(\sigma^{-1}\gamma) - n \\
& = & (n+2-\#(\gamma)-\#(\pi))+(n+\#(\pi)-\#(\sigma))-n \\
& = & n+2-\#(\gamma)-\#(\sigma).
\end{eqnarray*}
Thus, $\#(\sigma)+\#(\sigma^{-1}\gamma)+\#(\gamma) \geq n+2$, but \cite[Equation 2.9]{MN} says $\#(\sigma)+\#(\sigma^{-1}\gamma)+\#(\gamma) \leq n+2\#(\sigma \vee \gamma)=n+2$ so it must be equality.
\end{proof}

\begin{lemma}\label{Lemma: Removing a point is still non-crossing}
Let $M\subset \mathbb{Z}$ and let $\pi,\gamma \in S_M$ be such that $\pi\in S_{NC}(\gamma)$. Let $p\in M$ be such that $\pi(p)=p$ and we denote by $M^c$ to $M\setminus \{p\}$. Then $\pi|_{M^c}\in S_{NC}(\gamma|_{M^c})$.
\end{lemma}
\begin{proof}
It is clear that $\pi|_{M^c} \vee \gamma|_{M^c} = 1_{M^c}$. Let $\tilde{\gamma}$ be the permutation on $S_M$ whose cycle decomposition is the same as $\gamma|_{M^c}$ and $(p)$ is a singleton. Observe that the cycle decomposition of $\pi|_{M^c}$ is the same as $\pi$ except by the singleton $(p)$. Then
$$\#(\pi|_{M^c})+\#(\pi|_{M^c}^{-1}\gamma|_{M^c}^{-1})+\#(\gamma|_{M^c})=\#(\pi)-1+\#(\pi^{-1}\tilde{\gamma}^{-1})-1+\#(\gamma|_{M^c}).$$
Observe that $\gamma(p,\gamma^{-1}(p))=\tilde{\gamma}$, thus $\#(\pi^{-1}\tilde{\gamma})=\#(\pi^{-1}\gamma(p,\gamma^{-1}(p)))=\#(\pi^{-1}\gamma)+1$, hence
\begin{eqnarray*}
\#(\pi|_{M^c})+\#(\pi|_{M^c}^{-1}\gamma|_{M^c}^{-1})+\#(\gamma|_{M^c}) &=& \#(\pi)-1+\#(\pi^{-1}\gamma)+\#(\gamma|_{M^c}) \\
& = & \#(\pi)+\#(\pi^{-1}\gamma)+\#(\gamma)-1 \\
& = & |M|+2-1 = |M^c|+2.
\end{eqnarray*}
\end{proof}

\begin{lemma}\label{Lemma: Restriction to non-crossing is still non-crossing if connectivity is preserved}
Let $\sigma \in S_m$ be such that $\sigma\in S_{NC}(m_1,\dots,m_r)$ and let $M\subset [m]$ be such that $M \cap [m_i] \neq \emptyset$ for any $1\leq i\leq r$. If $\sigma|_M \vee \gamma|_M =1_M$ then $\sigma|_M \in S_{NC}(\gamma|_M)$ with $\gamma = \gamma_{m_1,\dots,m_r}$ defined as before.
\end{lemma}
\begin{proof}
Let $\tau \in S_n$ be given by $\tau(m)=m$ if $m \notin M$ and $\tau(m)=\sigma|_M(m)$ if $m\in M$. We will prove that $\tau\in S_{NC}(m_1,\dots,m_r)$. By definition any cycle of $\tau$ is contained in a cycle of $\sigma$ and $\tau \vee \gamma=1_m$ as $\gamma|_M \vee \gamma|_M =1_M$. So by Proposition \ref{Proposition: Smaller permutation of a non-crossing one is still non-crossing if connectivy is preserved} it suffices to show that for any cycle, $C$, of $\sigma$ the restriction of $\tau$ to $C$ denoted by $\tau|_C$ satisfies $\tau|_C \in \NC(C)$. Let $C=(a_1,a_2,\dots, a_w)$, if $C$ contains no elements of $M$ then the result is clear, so we may assume that $a_{i_1},\dots, a_{i_v} \in M$ with $i_1 < i_2 <\cdots < i_v$ and $a_{i_1}=a_1$. For $j \in C$, $\tau|_C^{-1}C(j)=C(j)$ whenever $C(j)\notin M$ and if $C(j)\in M$, say $C(j)=a_{i_r}$, then $\tau|_C^{-1}C(j)=a_{i_{r-1}}$. Therefore, the cycles of $\tau|_C^{-1}C$ are precisely $(a_{i_1},\dots,a_{i_2-1}), (a_{i_2},\dots,a_{i_3-1}),\dots, (a_{i_r},\dots,a_{w})$. Hence,
$$\#(\tau|_C)+\#(\tau|_C^{-1}C)=|C|-r+1+r=|C|+1,$$
as required, thus $\tau \in S_{NC}(m_1,\dots,m_r)$. We conclude by applying Lemma \ref{Lemma: Removing a point is still non-crossing} to $\tau$ for any singleton of $\tau$ to conclude that $\sigma|_M \in S_{NC}(\gamma|_M)$.
\end{proof}

Our next goal is to prove a more general result to Lemma \ref{Lemma: Restriction to non-crossing is still non-crossing if connectivity is preserved} where the condition $\sigma|_M \vee \gamma|_M =1_M$ doesn't necessarily satisfies. We will do this by induction over the number of blocks of $\sigma|_{M}\vee \gamma|_M$, so before proving it we need the following result which corresponds to the case where we have two blocks.

\begin{proposition}\label{Proposition: Restriction to non-crossing is still non-crossing if connectivity is preserved with two cycles}
Let $M\subset \mathbb{Z}$ and let $\gamma \in S_M$ be a permutation with cycle decomposition $\gamma= C_1\cdots C_n$. Let $\sigma \in S_M$ be such that $\sigma \in S_{NC}(\gamma)$. Let $M^{\prime} \subset M$ be such that $M^{\prime}\cap C_i \neq \emptyset$ for any $1\leq i\leq n$. If $\sigma|_{M^{\prime}} \vee \gamma|_{M^{\prime}}$ has two blocks, $A$ and $B$, then, $\sigma|_{M^{\prime}} \in S_{NC}(\gamma|_A)\times S_{NC}(\gamma|_B)$.
\end{proposition}
\begin{proof}
Let $\tilde{C}_i^{\prime}=C_i \cap M^{\prime}$ and let $C_i^{\prime}=\gamma|_{\tilde{C}_i^{\prime}}$. The blocks of $\sigma|_{M^{\prime}} \vee \gamma|_{M^{\prime}}$ are union of cycles of $\gamma|_{M^{\prime}}$, so suppose with out loss of generality that $A=\tilde{C}_1^{\prime}\cup \cdots \cup \tilde{C}_p^{\prime}$ and $B=\tilde{C}_{p+1}^{\prime}\cup \cdots \cup \tilde{C}_n^{\prime}$ for some $1<p<n$. Let $\tilde{A}$ be the set consisting of all elements in $C_1\cup \cdots \cup C_p$ and $\tilde{B}$ the one consisting of all elements in $C_{p+1}\cup \cdots \cup C_n$. Since $\sigma \in S_{NC}(\gamma)$ it must have a cycle that intersects $\tilde{A}$ and $\tilde{B}$, we call this cycle $D$. We must be in one of the following scenarios, either exactly one of $A \cap D$ and $B \cap D$ is non-empty or both are empty, the case where both are non-empty is not possible as that would mean that $\sigma|_{M^{\prime}}$ has a cycle that intersects $A$ and $B$ which is not possible. Assume then, we are in the former case, say $A \cap D \neq \emptyset$ and $B \cap D = \emptyset$. Let $a \in \tilde{B} \cap D$ so that $a \notin M^{\prime}$ and let $M^{\prime\prime}= M^{\prime}\cup \{a\}$. Thus the cycle of $\sigma|_{M^{\prime\prime}}$ that contains $a$ is a cycle that intersects $A$ and $B\cup\{a\}$ and therefore $\sigma|_{M^{\prime\prime}} \vee \gamma|_{M^{\prime\prime}}=1_{M^{\prime\prime}}$, hence by Lemma \ref{Lemma: Restriction to non-crossing is still non-crossing if connectivity is preserved}, $\sigma|_{M^{\prime\prime}}\in S_{NC}(\gamma|_{M^{\prime\prime}})$. Let $U$ be the cycle of $\sigma|_{M^{\prime\prime}}$ that contains $a$, we write this cycle as $(a,b_1,\dots,b_s)$ whit $b_i \in A$. Since $(a,b_1,\dots,b_s)(a,b_s)=(a)(b_1,\dots,b_s)$, $\sigma|_{M^{\prime\prime}}(a,b_s)$ has exactly the same cycles as $\sigma|_{M^{\prime}}$ and the extra cycle that only contains $a$. In the same way, any cycle of $\gamma|_{M^{\prime}}$ is a cycle of $\gamma|_{M^{\prime\prime}}$ except by the cycle of $\gamma|_{M^{\prime\prime}}$ that contains $a$ which we can write as $(a,d_1,\dots,d_l)$ so that $(d_1,\dots,d_l) \subset B$ is a cycle of $\gamma|_{M^{\prime}}$. As $(a,d_1,\dots,d_l)(a,d_l)=(a)(d_1,\dots,d_l)$, the cycles of $\gamma|_{M^{\prime\prime}}(a,d_l)$ are the same cycles of $\gamma|_{M^{\prime}}$ and the extra cycle that consist only of $a$. We let $\hat{\sigma}|_{M^{\prime}}$ and $\hat{\gamma}|_{M^{\prime}}$ be the permutations whose cycles are the cycles of $\sigma|_{M^{\prime}}$ and $\gamma|_{M^{\prime}}$ and the extra cycle that only contains $a$. Thus, $\sigma|_{M^{\prime\prime}}(a,b_s)=\hat{\sigma}|_{M^{\prime}}$ and $\gamma|_{M^{\prime\prime}}(a,d_l)=\hat{\gamma}|_{M^{\prime}}$. On the ohter hand, by \cite[Equation 2.9]{MN},
$$\#(\sigma|_A) +\#(\sigma|_A^{-1}\gamma|_A) + \#(\gamma|_A) \leq |A|+2,$$
and
$$\#(\sigma|_B) +\#(\sigma|_B^{-1}\gamma|_B) + \#(\gamma|_B) \leq |B|+2.$$
Summing both expression yields
$$\#(\sigma|_{M^{\prime}}) +\#(\sigma|_{M^{\prime}}^{-1}\gamma|_{M^{\prime}}) + \#(\gamma|_{M^{\prime}}) \leq |M^{\prime}|+4,$$
with equality only if $\sigma|_A \in S_{NC}(\gamma|_A)$ and $\sigma|_B \in S_{NC}(\gamma|_B)$, so we are reduce to prove last inequality must be equality. The permutations $\hat{\sigma}|_{M^{\prime}}^{-1}$ and $\hat{\gamma}|_{M^{\prime}}$ acts disjointly in the sets $A$ and $B$, thus in $\hat{\sigma}|_{M^{\prime}}^{-1}\hat{\gamma}|_{M^{\prime}}$, $a,b_s$ and $d_l$ are all in distinct cycles because $b_s \in A$, $d_l\in B$ and $a$ is a singleton which by construction is not in $M^{\prime}=A\cup B$. Thus
\begin{eqnarray*}
\#(\sigma|_{M^{\prime\prime}}^{-1}\gamma|_{M^{\prime\prime}}) &=& \#((a,b_s)\hat{\sigma}|_{M^{\prime}}^{-1}\hat{\gamma}|_{M^{\prime}}(a,d_l)) \\
&=& \#(\hat{\sigma}|_{M^{\prime}}^{-1}\hat{\gamma}|_{M^{\prime}})-2 \\
&=& \#(\sigma|_{M^{\prime}}^{-1}\gamma|_{M^{\prime}})+1-2.
\end{eqnarray*}
Therefore
\begin{eqnarray*}
&& \#(\sigma|_{M^{\prime}}) +\#(\sigma|_{M^{\prime}}^{-1}\gamma|_{M^{\prime}}) + \# (\gamma|_{M^{\prime}}) \\
&=& \#(\sigma|_{M^{\prime\prime}})+\#(\sigma|_{M^{\prime\prime}}^{-1}\gamma|_{M^{\prime\prime}})+1+\#(\gamma|_{M^{\prime\prime}}) \\
&=& |M^{\prime\prime}|+3 = |M^{\prime}|+4,    
\end{eqnarray*}
where in second equality we use that $\sigma|_{M^{\prime\prime}}\in S_{NC}(\gamma|_{M^{\prime\prime}})$ and then
$\#(\sigma|_{M^{\prime\prime}}) +\#(\sigma|_{M^{\prime\prime}}^{-1}\gamma|_{M^{\prime\prime}}) + \#(\gamma|_{M^{\prime\prime}}) = |M^{\prime\prime}|+2.$

For the case when $A\cap D=B\cap D=\emptyset$ we take $a\in \tilde{B}\cap D$ and $b\in \tilde{A}\cap D$ and we let $M^{\prime\prime}=M^{\prime}\cup\{a,b\}$. In this case $\sigma|_{M^{\prime\prime}}$ has exactly the same cycles as $\sigma|_{M^{\prime}}$ and the extra cycle $(a,b)$, we let $\hat{\sigma}|_{M^{\prime}}$ be the permutation with the same cycles as $\sigma|_{M^{\prime}}$ and the extra cycles $(a)$ and $(b)$, so that $\sigma|_{M^{\prime\prime}}=\hat{\sigma}|_{M^{\prime}}(a,b)$. Similarly, $\gamma|_{M^{\prime\prime}}$ has the same cycles as $\gamma|_{M^{\prime}}$ except by the cycles that contains $a$ and $b$ which can be writen as $(a,b_1,\dots,b_s)$ and $(b,d_1,\dots,d_l)$ so that $(b_1,\dots,b_s) \subset B$ and $(d_!,\dots,d_l) \subset A$ are cycles of $\gamma|_{M^{\prime}}$. We let $\hat{\gamma}|_{M^{\prime}}$ be the permutation with the same cycles as $\gamma|_{M^{\prime}}$ and the extra cycles $(a)$ and $(b)$. Thus, $\gamma|_{M^{\prime\prime}}(a,b_s)(b,d_l)=\hat{\gamma}|_{M^{\prime}}$, equivalently, $\gamma|_{M^{\prime\prime}}=\hat{\gamma}|_{M^{\prime}}(b,d_l)(a,b_s)$. Therefore
$$\sigma|_{M^{\prime\prime}}^{-1}\gamma|_{M^{\prime\prime}}=(a,b)\hat{\sigma}|_{M^{\prime}}^{-1}\hat{\gamma}|_{M^{\prime}}(b,d_l)(a,b_s).$$
In the permutation $\hat{\sigma}|_{M^{\prime}}^{-1}\hat{\gamma}|_{M^{\prime}}$, $(a)$ and $(b)$ are cycles and $b_s \in B$ and $d_l\in A$ are in distinct cycles, therefore
$$\#(\hat{\sigma}|_{M^{\prime}}^{-1}\hat{\gamma}|_{M^{\prime}})-3=\#((a,b)\hat{\sigma}|_{M^{\prime}}^{-1}\hat{\gamma}|_{M^{\prime}}(a,b_s)(b,d_l))=\#(\sigma|_{M^{\prime\prime}}^{-1}\gamma|_{M^{\prime\prime}}).$$
Therefore,
\begin{eqnarray*}
&& \#(\sigma|_{M^{\prime}}) +\#(\sigma|_{M^{\prime}}^{-1}\gamma|_{M^{\prime}}) + \# (\gamma|_{M^{\prime}}) \\
&=& \#(\sigma|_{M^{\prime\prime}})-1+\#(\hat{\sigma}|_{M^{\prime}}^{-1}\hat{\gamma}|_{M^{\prime}})-2+\#(\gamma|_{M^{\prime\prime}}) \\
&=&\#(\sigma|_{M^{\prime\prime}})+\#(\sigma|_{M^{\prime\prime}}^{-1}\gamma|_{M^{\prime\prime}})+\#(\gamma|_{M^{\prime\prime}}) \\
&=& |M^{\prime\prime}|+2=|M^{\prime}|+4,
\end{eqnarray*}
where in second equlity we use again that $\sigma|_{M^{\prime\prime}} \in S_{NC}(\gamma|_{M^\prime\prime})$, this concludes the proof.
\end{proof}

\begin{lemma}\label{Lemma: Restriction to non-crossing is still non-crossing if connectivity is preserved generalized}
Let $\sigma \in S_m$ be such that $\sigma\in S_{NC}(m_1,\dots,m_r)$ and let $M\subset [m]$ be such that $M \cap [m_i] \neq \emptyset$ for any $1\leq i\leq r$. Then
$$\sigma|_M \in \prod_{B \text{ block of }\sigma|_M \vee \gamma|_M} S_{NC}(\gamma|_B),$$
with $\gamma = \gamma_{m_1,\dots,m_r}$ defined as before.
\end{lemma}
\begin{proof}
We prove this by induction over $\#(\sigma|_M \vee \gamma|_M)$. The base case was done in Lemma \ref{Lemma: Restriction to non-crossing is still non-crossing if connectivity is preserved}. So let us assume this is true for $\#(\sigma|_M \vee \gamma|_M)=n$ with $1<n<r$ and we aim to prove it for $\#(\sigma|_M \vee \gamma|_M)=n+1$. Let $B_1,\dots,B_{n+1}$ be the blocks of $\#(\sigma|_M \vee \gamma|_M)$. Each block of $\#(\sigma|_M \vee \gamma|_M)$ is a union of cycles of $\gamma|_M$, and each cycle of $\gamma|_M$ corresponds to the restriction of a cycle of $\gamma$ to $M$, so we can write each $B_i$ as, $\cup_{j} [M_{k^{(i)}_j}]|_M $ with $(k^{(i)}_j)_j \subset [r]$. We let $\hat{B}_i = \cup_{j} [M_{k^{(i)}_j}]$. It is clear that the disjoint union of $\hat{B}_i$ equals $[m]$. Since $\sigma \in S_{NC}(m_1,\dots,m_r)$ there must have a through cycle that intersects $\hat{B}_1$ and $\hat{B}_j$ for some $1<j<n+1$, assume without loss of generality this cycle, $D$, intersects $\hat{B}_1$ and $\hat{B}_2$. Similarly to proposition \ref{Proposition: Restriction to non-crossing is still non-crossing if connectivity is preserved with two cycles} we have either $D \cap B_1$ and $D \cap B_2$ are both empty or exactly one is non-empty. When both are empty we take $a\in D\cap \hat{B}_1$ and $b \in D\cap \hat{B}_2$ so that $a,b \notin M$ as we let $M^{\prime} = M\cup \{a,b\}$. Now $\sigma|_{M^{\prime}}$ has a block that intersects both $B_1\cup \{a\}$ and $B_2 \cup \{b\}$ which means that $\#(\sigma|_{M^{\prime}}\vee \gamma|_{M^{\prime}})$ has $n$ blocks which are $B_1\cup \{a\} \cup B_2 \cup \{b\}$ and $B_3,\dots,B_{n+1}$. By induction hypothesis $\sigma|_{B_j}\in S_{NC}(\gamma|_{B_j})$ for $j=3,\dots,n+1$ and $\sigma|_{B_1\cup B_2\cup \{a,b\}} \in S_{NC}(\gamma|_{B_1\cup B_2\cup \{a,b\}})$. Now we use proposition \ref{Proposition: Restriction to non-crossing is still non-crossing if connectivity is preserved with two cycles} over $\sigma|_{B_1\cup B_2\cup \{a,b\}}$ to get $\sigma|_{B_1\cup B_2} \in S_{NC}(\gamma|_{B_1})\times S_{NC}(\gamma|_{B_2})$. The case when exactly one of $D \cap B_1$ and $D \cap B_2$ is non-empty proceeds similarly. Assume $D \cap B_1 \neq \emptyset$ and $D\cap B_2 = \emptyset$. Let $a\in D\cap \hat{B}_2$ so that $a\notin M$ and let $M^{\prime}=M\cup\{a\}$. Again, $\sigma$ now has a block that intersects $B_1$ and $B_{2}\cup \{a\}$, hence $\#(\sigma|_{M^{\prime}}\vee \gamma|_{M^{\prime}})$ has $n$ blocks which are $B_1 \cup B_2 \cup \{a\}$ and $B_3,\dots,B_{n+1}$. By induction each $\sigma|_{B_j} \in S_{NC}(\gamma|_{B_j})$ for $j\geq 3$ and $\sigma|_{B_1\cup B_2\cup \{a\}}\in S_{NC}(\gamma|_{B_1\cup B_2\cup \{a\}})$. By proposition \ref{Proposition: Restriction to non-crossing is still non-crossing if connectivity is preserved with two cycles} we get that $\sigma|_{B_1\cup B_2} \in S_{NC}(\gamma|_{B_1})\times S_{NC}(\gamma|_{B_2})$.
\end{proof}

\section{Preliminary results}\label{Section: Preliminary results}

We are ready to give some preliminary results for our case of interest, the $(p,q,l)$-annulus. These results will be necessary to prove our main theorem. 

\begin{lemma}\label{Lemma: Some properties of pi_n}
Let $\pi \in S_{r+s+t}$ and $\pi_{\vec{n}}$ be defined as before. Let $\psi :[r+s+t]\rightarrow [p+q+l]$ be given by $\psi(i)=n_1+\cdots+n_i$. The following are satisfied
\begin{enumerate}
    \item $\psi\pi^{-1}\gamma_{r,s,t}=\pi_{\vec{n}}^{-1}\gamma_{p,q,l}\psi$.
    \item $\#(\pi)=\#(\pi_{\vec{n}})$ and $\#(\pi^{-1}\gamma_{r,s,t})+(p+q+l)=\#(\pi_{\vec{n}}^{-1}\gamma_{p,q,l})+(r+s+t)$.
    \item $\pi_{\vec{n}} \in S_{NC}(p,q,l)$ provided $\pi \in S_{NC}(r,s,t)$.
\end{enumerate}
\end{lemma}
\begin{proof}
Observe that, $\pi_{\vec{n}}^{-1}\gamma_{p,q,l}\psi(i)=\pi_{\vec{n}}^{-1}\gamma_{p,q,l}(n_1+\cdots+n_i)$, however $\gamma_{p,q,l}$ send the last element of $T_i$ to the first element of $T_{\gamma_{r,s,t}(i)}$ which is $\gamma(n_1+\cdots+n_{\gamma_{r,s,t}(i)})$, therefore, $\pi_{\vec{n}}^{-1}(\gamma_{p,q,l}(n_1+\cdots+n_i))=\pi_{\vec{n}}^{-1}(\gamma(n_1+\cdots+n_{\gamma_{r,s,t}(i)}))=n_1+\cdots+n_{\pi^{-1}(\gamma_{r,s,t}(i))}$ where the last equality follows from Remark \ref{Remark: How Pin works}, hence $(1)$.

To prove $(2)$ observe that $\#(\pi)=\#(\pi_{\vec{n}})$ is clear by definition of $\pi_{\vec{n}}$. On the other hand, observe that if $\pi^{-1}\gamma_{r,s,t}(i)=j$ then by $(1)$, $\pi_{\vec{n}}^{-1}\gamma_{p,q,l}(n_1+\cdots+n_i)=n_1+\cdots+n_j$. The latter means that if $(i_1,\dots,i_s)$ is a cycle of $\pi^{-1}\gamma_{r,s,t}$ then $(n_1+\cdots+n_{i_1},\dots,n_1+\cdots+n_{i_s})$ is a cycle f $\pi_{\vec{n}}^{-1}\gamma_{p,q,l}$. Thus, $\pi^{-1}\gamma_{r,s,t}$ and $\pi_{\vec{n}}^{-1}\gamma_{p,q,l}$ restricted to $N$ have the same number of cycles. Moreover, if $i\notin N$ then $\pi_{\vec{n}}^{-1}\gamma_{p,q,l}(i)=i$ which means that $\pi_{\vec{n}}^{-1}\gamma_{p,q,l}$ restricted to $N^c =[n]\setminus N$ has as many cycles as $|N^c|=(p+q+l)-(r+s+t)$, this proves $(2)$.

Finally, it is clear that $\pi_{\vec{n}}\vee \gamma_{p,q,l}=1$ as $\pi\vee \gamma_{r,s,t}=1$, and
\begin{multline*}
\#(\pi_{\vec{n}})+\#(\pi_{\vec{n}}^{-1}\gamma_{p,q,l})=\#(\pi)+\#(\pi^{-1}\gamma_{r,s,t})-(r+s+t)+(p+q+l)
=p+q+l-1.
\end{multline*}
\end{proof}

\begin{lemma}\label{Lemma: Pin and Vn are non crossing}
Let $(\V,\pi) \in PS_{NC}^{(1)}(r,s,t)\cup PS_{NC}^{(2)}(r,s,t)\cup PS_{NC}^{(3)}(r,s,t)$ and $(\V_{\vec{n}},\pi_{\vec{n}})$ be defined as before. If $(\V,\pi)\in PS_{NC}^{(j)}(r,s,t)$ then, $$(\V_{\vec{n}},\pi_{\vec{n}}) \in PS_{NC}^{(j)}(p,q,l),$$ 
for $j=1,2,3$.
\end{lemma}
\begin{proof}
We prove it for $j=1$ and essentially the same proof follows for $j=2$ and $3$. Let $C_1\cdots C_wC^{\prime}C^{\prime}$ be the cycle decomposition of $\pi$ with $C^{\prime}$ and $C^{\prime\prime}$ being the marked cycles of $\pi$, that is the blocks of $\U$ are $0_{C_1},\dots,0_{C_w},0_{C^{\prime}\cup C^{\prime\prime}}$. We assume with out loss of generality that $\pi\in \NC(r)\times S_{NC}(s,t)$ so that $C^{\prime}\subset [r]$ and $C^{\prime\prime}\subset [r+1,r+s+t]$. Let $\tilde{C}_1\cdots \tilde{C}_w\tilde{C}^{\prime}\tilde{C}^{\prime\prime}$ being the cycle decomposition of $\pi_{\vec{n}}$ as in Notation \ref{Remark: Decomposition of Pin}. We firstly prove that $\pi_{\vec{n}}\in \NC(p)\times S_{NC}(q,l)$. We write $\pi = \pi_1\times \pi_2$ with $\pi_1\in \NC(r)$ and $\pi_2\in S_{NC}(s,t)$. Let $\pi_{\vec{1}}^{(1)}$ and $\pi_{\vec{n}}^{(2)}$ being defined as follows: whenever $C$ is a cycle of $\pi_1$ then we let $\tilde{C}$ being a cycle of $\pi_{\vec{1}}^{(1)}$ and similarly if $\tilde{C}$ is a cycle of $\pi_2$ we let $\tilde{C}$ being a cycle of $\pi_{\vec{n}}^{(2)}$. Thus $\pi_{\vec{n}}= \pi_{\vec{n}}^{(1)}\times \pi_{\vec{n}}^{(2)}$. Moreover $\pi_{\vec{n}}^{(1)} \in S_p$ while $\pi_{\vec{n}}^{(2)}\in S_{[p+1,p+q+l]}$. By \cite[Equation 2.9]{MN},
$$\#(\pi_{\vec{n}}^{(1)})+\#({\pi_{\vec{n}}^{(1)}}^{-1}\gamma_{p,q,l}|_{[p]})+\#(\gamma_{p,q,l}|_{[p]})\leq p+2,$$
with equality if and only if $\pi_{\vec{n}}^{(1)}\in \NC(p)$. On the other hand, there is a cycle, $C$, of $\pi_2$ that meets $[r+1,r+s]$ and $[r+s+1,r+s+t]$ and then $\tilde{C}$ is a cycle of $\pi_{\vec{n}}^{(2)}$ that meets $[p+1,p+q]$ and $[p+q+1,p+q+l]$, thus $\pi_{\vec{n}}^{(2)}\vee \gamma_{p,q,l}|_{[p+1,p+q+l]} =1_{[p+1,p+q+l]}$. By \cite[Equation 2.9]{MN},
$$\#(\pi_{\vec{n}}^{(2)})+\#({\pi_{\vec{n}}^{(2)}}^{-1}\gamma_{p,q,l}|_{[p+1,p+q+l]})+\#(\gamma_{p,q,l}|_{[p+1,p+q+l]})\leq q+l+2,$$
with equality if and only if $\pi_{\vec{n}}^{(2)}\in S_{NC}(q,l)$.
Summing the inequalities up yields
\begin{multline*}
\#(\pi_{\vec{n}}^{(1)})+\#({\pi_{\vec{n}}^{(1)}}^{-1}\gamma_{p,q,l}|_{[p]})+\#(\gamma_{p,q,l}|_{[p]}) \\
+ \#(\pi_{\vec{n}}^{(2)})+\#({\pi_{\vec{n}}^{(2)}}^{-1}\gamma_{p,q,l}|_{[p+1,p+q+l]})+\#(\gamma_{p,q,l}|_{[p+1,p+q+l]}) \\
\leq p+q+l+4.
\end{multline*}
However the left hand side of last inequality simplifies to
$$\#(\pi_{\vec{n}})+\#(\pi_{\vec{n}}^{-1}\gamma_{p,q,l})+\#(\gamma_{p,q,l}) \leq p+q+l+4.$$
Lemma \ref{Lemma: Some properties of pi_n} shows that
\begin{multline*}
\#(\pi)+\#(\pi^{-1}\gamma_{r,s,t})+\#(\gamma_{r,s,t})
=\#(\pi_{\vec{n}})+\#(\pi_{\vec{n}}^{-1}\gamma_{p,q,l})+\#(\gamma_{p,q,l})+(r+s+t)-(p+q+l),
\end{multline*}
and since $\pi\in \NC(r)\times S_{NC}(s,t)$ then $\#(\pi)+\#(\pi^{-1}\gamma_{r,s,t})+\#(\gamma_{r,s,t})=r+s+t+4$, thus
$$\#(\pi_{\vec{n}})+\#(\pi_{\vec{n}}^{-1}\gamma_{p,q,l})+\#(\gamma_{p,q,l})=p+q+l+4,$$
which means the last inequality must actually be equality, hence $\pi_{\vec{n}}\in \NC(p)\times S_{NC}(q,l)$. To finish our proof it is enough to note that by definition the blocks of $\V_{\vec{n}}$ are $0_{\tilde{C}_1},\dots,0_{\tilde{C}_w},0_{\tilde{C}^{\prime}\cup \tilde{C}^{\prime\prime}}$, with $\tilde{C}^{\prime}\subset [p]$ and $\tilde{C}^{\prime\prime}\subset [p+1,p+q+l]$.
\end{proof}

The following is the analogous to \cite[Proposition 24]{MST} in the three circles case.

\begin{proposition}\label{Proposition: Suficiente conditions version 1}
Let $\pi \in S_{NC}(r,s,t)$ and $\pi_{\vec{n}}\in S_{NC}(p,q,l)$. If $\sigma\in S_{n}$ is such that satisfies all followings
\begin{enumerate}
    \item $\sigma \leq \pi_{\vec{n}}$, and,
    \item $\sigma_i \vee \gamma_i =\tilde{C_i}$ or equivalently $\sigma_i^{-1}\tilde{C_i}$ separates the points of $N \cap \tilde{C_i}$ for each cycle $\tilde{C_i}$ of $\pi_{\vec{n}}$, with $\sigma_i$ being the restriction of $\sigma$ to $\tilde{C_i}$.
\end{enumerate}
Then $\sigma\in S_{NC}(p,q,l)$ and $\sigma^{-1}\pi_{\vec{n}}$ separates the points of $N$.
\end{proposition}
\begin{proof}
Firstly, it is clear that $\sigma^{-1}\pi_{\vec{n}}$ separates $N$ as it does separate $N\cap B$ for every cycle, $B$, of $\pi_{\vec{n}}$. So it remains to verify $\sigma\in S_{NC}(p,q,l)$ which by Proposition \ref{Proposition: Smaller permutation of a non-crossing one is still non-crossing if connectivy is preserved} it is enough to verify $\sigma \vee \gamma_{p,q,l}=1_n$. By \cite[Lemma 6]{MST},
$$\sigma^{-1}\gamma_{p,q,l}|_N = \sigma^{-1}\pi_{\vec{n}}|_N \pi_{\vec{n}}^{-1}\gamma_{p,q,l}|_N = \pi_{\vec{n}}^{-1}\gamma_{p,q,l}|_N.$$
Lemma \ref{Lemma: Some properties of pi_n} says
$$\pi_{\vec{n}}^{-1}\gamma_{p,q,l}(n_1+\cdots+n_i)=n_1+\cdots+n_{\pi^{-1}\gamma_{r,s,t}(i)},$$
moreover, $\pi\in S_{NC}(r,s,t)$, so there exist $a \in [r]$ and $b\in [r+1,r+s+t]$ such that $\pi(a)=b$. Either $b\in [r+1,r+s]$ or $b\in [r+s+1,r+s+t]$, assume with out loss of generality we are in the former case. Then
$$n_1+\cdots+n_a = n_1+\cdots+ n_{\pi^{-1}\gamma_{r,s,t}\gamma^{-1}_{r,s,t}(b)}=\pi_{\vec{n}}^{-1}\gamma_{p,q,l}(n_1+\cdots+n_{\gamma_{r,s,t}^{-1}(b)}).$$
Hence,
$$\sigma^{-1}\gamma_{p,q,l}|_N(n_1+\cdots+n_{\gamma_{r,s,t}^{-1}(b)})=n_1+\cdots+n_a.$$
The latter means that $\sigma^{-1}\gamma_{p,q,l}$ has a cycle that contains $n_1+\cdots+n_a \in [p]$ and $n_1+\cdots+n_{\gamma_{r,s,t}^{-1}(b)}\in [p+1,p+q]$. Thus there exist $\hat{a}\in [p]$ and $\hat{b}\in [p+1,p+q]$ with $\sigma^{-1}\gamma_{p,q,l}(\hat{a})=\hat{b}$ or equivalently $\sigma(\hat{b})=\gamma_{p,q,l}(\hat{a})$, this proves that $\sigma$ has a cycle that meets $[p]$ and $[p+1,p+q]$. Similarly, since $\pi\in S_{NC}(r,s,t)$ there must exist $c\in [r+s+1,r+s+t]$ and $d\in [r+s]$ with $\pi(c)=d$, we proceed as before to show that $\sigma$ must have a cycle that meets $[p+q+1,p+q+l]$ and either $[p]$ or $[p+1,p+q]$ depending whether $d\in [r]$ or $d\in [r+1,r+s]$. This proves $\sigma\vee\gamma_{p,q,l}=1_n$.
\end{proof}

\begin{proposition}\label{Proposition: Suficiente conditions version 2}
Let $\pi \in \NC(r)\times S_{NC}(s,t)$ and $\pi_{\vec{n}}\in \NC(p)\times S_{NC}(q,l)$. If $\sigma\in S_{n}$ is such that satisfies all followings
\begin{enumerate}
    \item $\sigma \lesssim^{(1)} \pi_{\vec{n}}$, and,
    \item The block of $\sigma\vee\pi_{\vec{n}}$ which is the union of two cycles of $\pi_{\vec{n}}$ is such that one of these cycles is contained in $[p]$ and the other in $[p+1,p+q+l]$, and,
    \item For each block, $B$, of $\sigma\vee\pi_{\vec{n}}$, $\sigma|_B^{-1}\pi_{\vec{n}}|_B$ separates the points of $N \cap B$.
\end{enumerate}
Then $\sigma\in S_{NC}(p,q,l)$ and $\sigma^{-1}\pi_{\vec{n}}$ separates the points of $N$.
\end{proposition}
\begin{proof}
The condition $\sigma^{-1}\pi_{\vec{n}}$ separates $N$ is clearly satisfied so it remains to verify $\sigma\in S_{NC}(p,q,l)$. We first verify $\sigma\vee\gamma_{p,q,l}=1_{n}$. By \cite[Lemma 6]{MST},
$$\sigma^{-1}\gamma_{p,q,l}|_N = \sigma^{-1}\pi_{\vec{n}}|_N \pi_{\vec{n}}^{-1}\gamma_{p,q,l}|_N = \pi_{\vec{n}}^{-1}\gamma_{p,q,l}|_N.$$
Lemma \ref{Lemma: Some properties of pi_n} says
$$\pi_{\vec{n}}^{-1}\gamma_{p,q,l}(n_1+\cdots+n_i)=n_1+\cdots+n_{\pi^{-1}\gamma_{r,s,t}(i)},$$
moreover, $\pi\in \NC(r)\times S_{NC}(s,t)$, so there exist $a,b$ such that $a\in [r+1,r+s]$, $b\in [r+s+1,r+s+t]$ and $\pi(a)=b$, thus
$$n_1+\cdots+n_a = n_1+\cdots+ n_{\pi^{-1}\gamma_{r,s,t}\gamma^{-1}_{r,s,t}(b)}=\pi_{\vec{n}}^{-1}\gamma_{p,q,l}(n_1+\cdots+n_{\gamma_{r,s,t}^{-1}(b)}).$$
Hence,
$$\sigma^{-1}\gamma_{p,q,l}|_N(n_1+\cdots+n_{\gamma_{r,s,t}^{-1}(b)})=n_1+\cdots+n_a.$$
The latter means that $\sigma^{-1}\gamma_{p,q,l}$ has a cycle that contains $n_1+\cdots+n_a \in [p+1,p+q]$ and $n_1+\cdots+n_{\gamma_{r,s,t}^{-1}(b)} \in [p+q+1,p+q+l]$, thus there exist $\hat{a},\hat{b}$ such that $\hat{a}\in [p+1,p+q]$ and $\hat{b}\in [p+q+1,p+q+l]$ with $\sigma^{-1}\gamma_{p,q,l}(\hat{a})=\hat{b}$, or equivalently, $\sigma(\hat{b})=\gamma_{p,q,l}(\hat{a})$, this proves that $\sigma$ has a cycle that meets $[p+1,p+q]$ and $[p+q+1,p+q+l]$. On the other hand, let $B_0$ be the block of $\sigma\vee\pi_{\vec{n}}$ which is the union of two cycles of $\pi_{\vec{n}}$, and let us denote these cycles as $\tilde{C}^{\prime}$ and $\tilde{C}^{\prime\prime}$. By hypothesis, $\sigma\in S_{NC}(\tilde{C}^{\prime},\tilde{C}^{\prime\prime})$ with $\tilde{C}^{\prime}\subset [p]$ and $\tilde{C}^{\prime\prime}\subset [p+1,p+q+l]$, hence, $\sigma$ must have a cycle that meets $[p]$ and $[p+1,p+q+l]$, this proves $\sigma\vee\gamma_{p,q,l}=1_n$. We are reduce to prove $\#(\sigma) +\#(\sigma^{-1}\gamma_{p,q,l})+\#(\gamma_{p,q,l})=n+2$ which is equivalent to show $\#(\sigma) +\#(\sigma^{-1}\gamma_{p,q,l})+\#(\gamma_{p,q,l}) \geq n+2$ as the reverse inequality is always true (\cite[Equation 2.9]{MN}). By Proposition \ref{Proposition: Sufficient and neccesary conditions general version},
$$\#(\sigma)+\#(\sigma^{-1}\pi_{\vec{n}})+\#(\pi_{\vec{n}})= n+2\#(\sigma\vee\pi_{\vec{n}}).$$
By \cite[Equation 2.9]{MN},
\begin{eqnarray*}
\#(\pi_{\vec{n}}^{-1}\gamma_{p,q,l})+\#(\sigma^{-1}\pi_{\vec{n}})+\#(\sigma^{-1}\gamma_{p,q,l}) &\leq & n+2\#(\sigma^{-1}\pi_{\vec{n}}\vee \sigma^{-1}\gamma_{p,q,l}) \\ & \leq & n+2\#(\sigma^{-1}\gamma_{p,q,l}),
\end{eqnarray*}
Therefore,
\begin{eqnarray*}
\#(\sigma^{-1}\gamma_{p,q,l}) &\geq&  \#(\pi_{\vec{n}}^{-1}\gamma_{p,q,l})+\#(\sigma^{-1}\pi_{\vec{n}})-n \\
&=& (n+1-\#(\pi_{\vec{n}}))+(2\#(\sigma\vee\pi_{\vec{n}})-\#(\pi_{\vec{n}})-\#(\sigma)) \\
&=& n+1-\#(\pi_{\vec{n}})+\#(\pi_{\vec{n}})-2-\#(\sigma),
\end{eqnarray*}
where in second line we use that $\pi_{\vec{n}}\in \NC(p)\times S_{NC}(q,l)$. Thus, $\#(\sigma)+\#(\sigma^{-1}\gamma_{p,q,l})+\#(\gamma_{p,q,l})\geq n+2$.
\end{proof}

\begin{lemma}\label{Lemma: Relation A and a, version 1}
Let $\pi\in S_{NC}(r,s,t)$, then
$$\kappa_{\pi}(A_1,\dots,A_{r+s+t})=\sum_{\substack{\sigma\in S_{NC}(p,q,l) \\ \sigma \leq \pi_{\vec{n}} \\ \sigma^{-1}\pi_{\vec{n}}\text{ separates }N}} \kappa_{\sigma}(\vec{a}).$$
\end{lemma}
\begin{proof}
Let $C_1\cdots C_w$ be the cycle decomposition of $\pi$. If $C_i=\{j_1,\dots,j_v\}$, we denote by, $\kappa_{|C_i|}(\vec{A})$ to $\kappa_{|C_i|}(A_{j_1},\dots,A_{j_v})$. Then
$$\kappa_{\pi}(A_1,\dots,A_{r+s+t})=\kappa_{|C_1|}(\vec{A})\cdots \kappa_{|C_w|}(\vec{A}).$$
For each $1\leq i\leq w$ by \cite[Theorem 2.2]{KS} we have
$$\kappa_{|C_i|}(\vec{A})=\sum_{\substack{\sigma_i \in \NC(\tilde{C_i}) \\ \sigma_i \vee \gamma_i=\tilde{C_i}}}\kappa_{\sigma}(\vec{a})=\sum_{\substack{\sigma_i \in \NC(\tilde{C_i}) \\ \sigma_i^{-1}\tilde{C_i}\text{ separates }N\cap\tilde{C_i}}}\kappa_{\sigma_i}(\vec{a}),$$
where in above expression $\kappa_{\sigma_i}(\vec{a})$ means evaluating the cumulant $\kappa_{\sigma_i}$ only on the set of indices corresponding to $\tilde{C}_i$. Then
$$\kappa_{\pi}(\vec{A})=\sum_{\substack{\sigma_1 \in \NC(\tilde{C_1}) \\ \sigma_1^{-1}\tilde{C_1}\text{ separates }N\cap\tilde{C_1}}}\cdots \sum_{\substack{\sigma_w \in \NC(\tilde{C_w}) \\ \sigma_w^{-1}\tilde{C_w}\text{ separates }N\cap\tilde{C_w}}}\kappa_{\sigma_1}(\vec{a})\cdots \kappa_{\sigma_w}(\vec{a}).$$
We let $\sigma=\sigma_1\times \cdots \times \sigma_w$ then by Proposition \ref{Proposition: Suficiente conditions version 1} we have $\sigma\in S_{NC}(p,q,l)$, $\sigma\leq \pi_{\vec{n}}$ and $\sigma^{-1}\pi_{\vec{n}}$ separates the points of $N$. Conversely if $\sigma\in S_{NC}(p,q,l)$ is such that $\sigma\leq \pi_{\vec{n}}$ and $\sigma^{-1}\pi_{\vec{n}}$ separates the points of $N$ then each $\sigma_i = \sigma|_{\tilde{C_i}}$ satisfies $\sigma_i\in \NC(\tilde{C_i})$ and $\sigma_i^{-1}\tilde{C_i}$ separates the points of $N\cap\tilde{C_i}$. 
\end{proof}

\begin{lemma}\label{Lemma: Relation A and a, version 2}
Let $(\V,\pi) \in \PS_{NC}^{(1)}(r,s,t)$, then
$$\kappa_{(\V,\pi)}(A_1,\dots,A_{r+s+t})=\sum_{\substack{\sigma\in S_{NC}(p,q,l) \\ \sigma \lesssim^{(1)} \pi_{\vec{n}} \\ \sigma^{-1}\pi_{\vec{n}}\text{ separates }N \\ \sigma \vee \pi_{\vec{n}}=\V_{\vec{n}}}} \kappa_{\sigma}(\vec{a}) + \sum_{\substack{(\U,\sigma)\in \PS_{NC}^{(1)}(p,q,l) \\ \sigma \leq \pi_{\vec{n}} \\ \sigma^{-1}\pi_{\vec{n}}\text{ separates }N \\ \U\vee\pi_{\vec{n}}=\V_{\vec{n}}}} \kappa_{(\U,\sigma)}(\vec{a}).$$
\end{lemma}
\begin{proof}
Let us assume $\pi \in \NC(r)\times S_{NC}(s,t)$ and let $C_1\cdots C_wC^{\prime}C^{\prime\prime}$ be the cycle decomposition of $\pi$ so that $0_{C_i}$ are all blocks of $\V$ for $1\leq i\leq w$ and $0_{C^{\prime}\cup C^{\prime\prime}}$ is the block of $\V$ which is the union of two cycles of $\pi$; $C^{\prime}\subset [r]$ and $C^{\prime\prime}\subset [r+1,r+s+t]$. Thus
$$\kappa_{(\V,\pi)}(\vec{A})=\kappa_{|C_1|}(\vec{A})\cdots \kappa_{|C_w|}(\vec{A})\kappa_{|C^{\prime}|,|C^{\prime\prime}|}(\vec{A}).$$
By \cite[Theorem 2.2]{KS},
$$\kappa_{|C_i|}(\vec{A})=\sum_{\substack{\sigma_i \in \NC(\hat{C}_i) \\ \sigma_i^{-1} \hat{C}_i \text{ separates }N\cap \tilde{C}_i}}\kappa_{\sigma_i}(\vec{a}).$$
Similarly, by \cite[Theorem 3]{MST},
$$\kappa_{|C^{\prime}|,|C^{\prime\prime}|}(\vec{A}) = \sum_{\substack{(\V_0,\sigma_0)\in PS_{NC}(\tilde{C}^{\prime},\tilde{C}^{\prime\prime}) \\ \sigma_0^{-1}\tilde{C}^{\prime}\tilde{C}^{\prime\prime}\text{ separates }N\cap\tilde{C}^{\prime}\cap \tilde{C}^{\prime\prime}}}\kappa_{(\V_0,\sigma_0)}(\vec{a}).$$
Remind that the set $PS_{NC}(\tilde{C}^{\prime},\tilde{C}^{\prime\prime})$ can be written as the union of two sets \cite[Proposition 5.11]{CMSS}, these are $S_{NC}(\tilde{C}^{\prime},\tilde{C}^{\prime\prime})$ and $PS_{NC}^{\prime}(\tilde{C}^{\prime},\tilde{C}^{\prime\prime})$, where the last set consists of those $(\U,\pi)$ such that $\pi=\pi_1\times\pi_2\in\NC(\tilde{C}^{\prime})\times \NC(\tilde{C}^{\prime\prime})$ and any block of $\U$ is a cycle of $\pi$ except by one block which is the union of two cycles of $\pi$ one from each $\pi_1$ and $\pi_2$. Thus
$$\kappa_{(\V,\pi)}(\vec{A}) = \mathcal{P}\mathcal{A}+\mathcal{P}\mathcal{B},$$
with,
$$\mathcal{P}=\sum_{\substack{\sigma_1 \in \NC(\hat{C}_1) \\ \sigma_1^{-1} \hat{C}_1 \text{ separates }N\cap \tilde{C}_1}}\kappa_{\sigma_1}(\vec{a}) \cdots \sum_{\substack{\sigma_w \in \NC(\hat{C}_w) \\ \sigma_w^{-1} \hat{C}_w \text{ separates }N\cap \tilde{C}_w}}\kappa_{\sigma_w}(\vec{a}),$$
$$\mathcal{A}=\sum_{\substack{\sigma_0 \in S_{NC}(\tilde{C}^{\prime},\tilde{C}^{\prime\prime}) \\ \sigma_0^{-1}\tilde{C}^{\prime}\tilde{C}^{\prime\prime}\text{ separates }N\cap\tilde{C}^{\prime}\cap \tilde{C}^{\prime\prime}}}\kappa_{\sigma_0}(\vec{a}),$$
and,
$$\mathcal{B}=\sum_{\substack{(\V_0,\sigma_0)\in PS_{NC}^{\prime}(\tilde{C}^{\prime},\tilde{C}^{\prime\prime}) \\ \sigma_0^{-1}\tilde{C}^{\prime}\tilde{C}^{\prime\prime}\text{ separates }N\cap\tilde{C}^{\prime}\cap \tilde{C}^{\prime\prime}}}\kappa_{(\V_0,\sigma_0)}(\vec{a}).$$
In $\mathcal{P}\mathcal{A}$, we let $\sigma=\sigma_0\times \cdots\times \sigma_w$, then $\sigma^{-1}\pi_{\vec{n}}$ separates $N$ and $\sigma \vee\pi_{\vec{n}}=\V_{\vec{n}}$. At each block, $B$, of $\V_{\vec{n}}$, we have $\sigma|_{B}\in S_{NC}(\pi_{\vec{n}}|_B)$, i.e. $\sigma \lesssim^{(1)}\pi_{\vec{n}}.$ Finally by Proposition \ref{Proposition: Suficiente conditions version 2} we have $\sigma\in S_{NC}(p,q,l)$. Conversely, if $\sigma\in S_{NC}(p,q,l)$ is such that $\sigma \lesssim^{(1)} \pi_{\vec{n}}$, $\sigma^{-1}\pi_{\vec{n}}$ separates $N$ and $\sigma\vee\pi_{\vec{n}}=\V_{\vec{n}}$ then at each block, $B$, of $\V_{\vec{n}}$, we have $\sigma|_{B}\in S_{NC}(\pi_{\vec{n}}|_B)$ and $\sigma|_B^{-1}\pi_{\vec{n}}|_B$ separates $N\cap B$, thus
$$\mathcal{P}\mathcal{A}=\sum_{\substack{\sigma\in S_{NC}(p,q,l) \\ \sigma \lesssim^{(1)} \pi_{\vec{n}} \\ \sigma^{-1}\pi_{\vec{n}}\text{ separates }N \\ \sigma \vee \pi_{\vec{n}}=\V_{\vec{n}}}} \kappa_{\sigma}(\vec{a}).$$
Similarly, in the second sum, we let $\sigma=\sigma_0\times \cdots \times \sigma_w$, and $\U$ to be the partition of $[n]$ such that each cycle of $\pi_i$ is a block of $\U$ for $1\leq i\leq w$ and the blocks of $\U$ restricted to $\tilde{C}^{\prime}\cup \tilde{C}^{\prime\prime}$ are precisely the blocks of $\V_0$. In this case, $\sigma\leq \pi_{\vec{n}}$ and $\pi_{\vec{n}} \in \NC(p)\times S_{NC}(q,l)$ thus by Proposition \ref{Proposition: Smaller permutation of a non-crossing one is still non-crossing if connectivy is preserved} $\sigma\in \NC(p)\times S_{NC}(q,l)$. Each block of $\U$ is a cycle of $\sigma$ except by one which is the union of two cycles of $\sigma$, one contained in $\tilde{C}^{\prime} \subset [p]$ and the other contained in $\tilde{C}^{\prime\prime}\subset [p+1,p+q+l]$, thus $(\U,\sigma)\in PS_{NC}^{(1)}(p,q,l)$. Furthermore, $\U\vee\pi_{\vec{n}}=\V_{\vec{n}}$. Conversely if $(\U,\sigma)\in PS_{NC}^{(1)}(p,q,l)$ satisfies all mentioned before then for each cycle, $B$, of $\pi_{\vec{n}}$, we have $\sigma|_{B}\in \NC(B)$. There are exactly two cycles of $\sigma$ that are in the same block of $\U$, these cycles must lie inside distinct cycles of $\pi_{\vec{n}}$ as otherwise any block of $\U$ is contained in a cycle of $\pi_{\vec{n}}$ which contradicts $\U \vee \pi_{\vec{n}}=\V_{\vec{n}}$. Let $\tilde{C}^{\prime}$ and $\tilde{C}^{\prime\prime}$ be these cycles of $\pi_{\vec{n}}$. For any other cycle of $\pi_{\vec{n}}$, $B$, we have that any cycle of $\sigma|_B$ is also a block of $\U$ and $\sigma|_B^{-1}B$ separates $N\cap B$. Only within these cycles there are exactly two cycles of $\sigma|_{\tilde{C}^{\prime}\cup \tilde{C}^{\prime\prime}}$ which are joined into the same block of $\U$, that is, if $(\U_0,\sigma_0)$ is the restriction of $(\U,\sigma)$ to $\tilde{C}^{\prime}\cup \tilde{C}^{\prime\prime}$ then $(\U_0,\sigma_0)\in PS_{NC}^{\prime}(\tilde{C}^{\prime}, \tilde{C}^{\prime\prime})$ and clearly $\sigma_0^{-1}\tilde{C}^{\prime}\tilde{C}^{\prime\prime}$ separates $N\cap \tilde{C}^{\prime}\cap \tilde{C}^{\prime\prime}$, thus
$$\mathcal{P}\mathcal{B}=\sum_{\substack{(\U,\sigma)\in \PS_{NC}^{(1)}(p,q,l) \\ \sigma \leq \pi_{\vec{n}} \\ \sigma^{-1}\pi_{\vec{n}}\text{ separates }N \\ \U\leq \V_{\vec{n}}}} \kappa_{(\U,\sigma)}(\vec{a}).$$
\end{proof}

\begin{lemma}\label{Lemma: Relation A and a, version 3}
Let $(\V,\pi) \in \PS_{NC}^{(2)}(r,s,t)$, then
\begin{multline*}
\kappa_{(\V,\pi)}(A_1,\dots,A_{r+s+t})= \\
\sum_{\substack{\sigma\in S_{NC}(p,q,l) \\ \sigma \lesssim^{(2)} \pi_{\vec{n}} \\ \sigma^{-1}\pi_{\vec{n}}\text{ separates }N \\ \sigma \vee \pi_{\vec{n}}=\V_{\vec{n}}}} \kappa_{\sigma}(\vec{a}) + \sum_{\substack{(\U,\sigma)\in \PS_{NC}^{(1)}(p,q,l) \\ \sigma \lesssim^{(1)} \pi_{\vec{n}} \\ \sigma^{-1}\pi_{\vec{n}}\text{ separates }N \\ \U\vee \pi_{\vec{n}}=\V_{\vec{n}}}} \kappa_{(\U,\sigma)}(\vec{a})
+ \sum_{\substack{(\U,\sigma)\in \PS_{NC}^{(2)}(p,q,l) \\ \sigma \leq \pi_{\vec{n}} \\ \sigma^{-1}\pi_{\vec{n}}\text{ separates }N \\ \U\vee \pi_{\vec{n}}=\V_{\vec{n}}}} \kappa_{(\U,\sigma)}(\vec{a}).
\end{multline*}
\end{lemma}
\begin{proof}
Let $C_1\cdots C_wA_1A_2B_1B_2$ be the cycle decomposition of $\pi$ such that $0_{C_i}$, $0_{A_1\cup B_1}$ and $0_{A_2\cup B_2}$ are the blocks of $\V$. Let us assume without loss of generality $A_1,A_2 \subset [r]$, $B_1 \subset [r+1,r+s]$ and $B_2 \subset [r+s+1,r+s+t]$. We let $\tilde{C}_1\cdots\tilde{C}_w\tilde{A}_1\tilde{A}_2\tilde{B}_1\tilde{B}_2$ be the cycle decomposition of $\pi_{\vec{n}}$. Thus
$$\kappa_{(\V,\pi)}(\vec{A})=\kappa_{|C_1|}(\vec{A})\cdots \kappa_{|C_w|}(\vec{A})\kappa_{|A_1|,|B_1|}(\vec{A})\kappa_{|A_2|,|B_2|}(\vec{A}).$$
By \cite[Theorem 2.2]{KS}, \cite[Theorem 3]{MST}, 
$$\kappa_{|C_i|}(\vec{A})=\sum_{\substack{\sigma_i \in \NC(\hat{C}_i) \\ \sigma_i^{-1} \hat{C}_i \text{ separates }N\cap \tilde{C}_i}}\kappa_{\sigma_i}(\vec{a}).$$
By \cite[Theorem 3]{MST},
$$\kappa_{|A_i|,|B_i|}(\vec{A}) = \sum_{\substack{(\V_{0,i},\sigma_{0,i})\in PS_{NC}(\tilde{A}_i,\tilde{B}_i) \\ \sigma_{0,i}^{-1}\tilde{A}_i\tilde{B}_i\text{ separates }N\cap\tilde{A}_i\cap \tilde{B}_i}}\kappa_{(\V_{0,i},\sigma_{0,i})}(\vec{a}),$$
for $i=1,2$. We proceed as in Lemma \ref{Lemma: Relation A and a, version 2}, writing each sum above as a double sum so that,
$$\kappa_{(\V,\pi)}(\vec{A})=\mathcal{P}\mathcal{A}+\mathcal{P}\mathcal{B}+\mathcal{P}\mathcal{C}+\mathcal{P}\mathcal{D},$$
with $\mathcal{P}$ as in the proof of Lemma \ref{Lemma: Relation A and a, version 2} and,
$$\mathcal{A}=\sum_{\substack{\sigma_{0,1} \in S_{NC}(\tilde{A}_1,\tilde{B}_1) \\ \sigma_{0,1}^{-1}\tilde{A}_1\tilde{B}_1\text{ separates }N\cap\tilde{A}_1\cap \tilde{B}_1}}\kappa_{\sigma_{0,1}}(\vec{a})\sum_{\substack{\sigma_{0,2} \in S_{NC}(\tilde{A}_2,\tilde{B}_2) \\ \sigma_{0,2}^{-1}\tilde{A}_2\tilde{B}_2\text{ separates }N\cap\tilde{A}_2\cap \tilde{B}_2}}\kappa_{\sigma_{0,2}}(\vec{a}),$$
$$\mathcal{B}=\sum_{\substack{\sigma_{0,1} \in S_{NC}(\tilde{A}_1,\tilde{B}_1) \\ \sigma_{0,1}^{-1}\tilde{A}_1\tilde{B}_1\text{ separates }N\cap\tilde{A}_1\cap \tilde{B}_1}}\kappa_{\sigma_{0,1}}(\vec{a})\sum_{\substack{(\V_{0,2},\sigma_{0,2}) \in PS_{NC}^{\prime}(\tilde{A}_2,\tilde{B}_2) \\ \sigma_{0,2}^{-1}\tilde{A}_2\tilde{B}_2\text{ separates }N\cap\tilde{A}_2\cap \tilde{B}_2}}\kappa_{(\V_{0,2},\sigma_{0,2})}(\vec{a}),$$
$$\mathcal{C}=\sum_{\substack{(\V_{0,1},\sigma_{0,1}) \in PS_{NC}^{\prime}(\tilde{A}_1,\tilde{B}_1) \\ \sigma_{0,1}^{-1}\tilde{A}_1\tilde{B}_1\text{ separates }N\cap\tilde{A}_1\cap \tilde{B}_1}}\kappa_{(\V_{0,1},\sigma_{0,1})}(\vec{a})\sum_{\substack{\sigma_{0,2} \in S_{NC}(\tilde{A}_2,\tilde{B}_2) \\ \sigma_{0,2}^{-1}\tilde{A}_2\tilde{B}_2\text{ separates }N\cap\tilde{A}_2\cap \tilde{B}_2}}\kappa_{\sigma_{0,2}}(\vec{a}),$$
and,
$$\mathcal{D}=\sum_{\substack{(\V_{0,1},\sigma_{0,1}) \in PS_{NC}^{\prime}(\tilde{A}_1,\tilde{B}_1) \\ \sigma_{0,1}^{-1}\tilde{A}_1\tilde{B}_1\text{ separates }N\cap\tilde{A}_1\cap \tilde{B}_1}}\kappa_{(\V_{0,1},\sigma_{0,1})}(\vec{a})\sum_{\substack{(\V_{0,2},\sigma_{0,2}) \in PS_{NC}^{\prime}(\tilde{A}_2,\tilde{B}_2) \\ \sigma_{0,2}^{-1}\tilde{A}_2\tilde{B}_2\text{ separates }N\cap\tilde{A}_2\cap \tilde{B}_2}}\kappa_{(\V_{0,2},\sigma_{0,2})}(\vec{a}).$$
In the first sum, $\mathcal{P}\mathcal{A}$, we let $\sigma=\sigma_{0,1}\times\sigma_{0,2}\sigma_1\times\cdots\sigma_w$. Thus $\sigma\lesssim^{(2)} \pi_{\vec{n}}$, $\sigma^{-1}\pi_{\vec{n}}$ separates $N$ and $\sigma\vee\pi_{\vec{n}}=\V_{\vec{n}}$. It remains to prove $\sigma\in S_{NC}(p,q,l)$. Since $\sigma_{0,1}$ has a cycle that meets $\hat{A}_1 \subset [p]$ and $\hat{B}_1 \subset [p+1,p+q]$, while $\sigma_{0,2}$ has a cycle that meets $\hat{A}_2 \subset [p]$ and $\hat{B}_2 \subset [p+q+1,p+q+l]$, then $\sigma \vee\gamma_{p,q,l}=1_n$. So we are reduce to verify $\#(\sigma)+\#(\sigma^{-1}\gamma_{p,q,l})+\#(\gamma_{p,q,l})\geq n+2$. By \cite[Equation 2.9]{MN},
$$\#(\pi_{\vec{n}}^{-1}\gamma_{p,q,l})+\#(\sigma^{-1}\pi_{\vec{n}})+\#(\sigma^{-1}\gamma_{p,q,l}) \leq n+2\#(\sigma^{-1}\gamma_{p,q,l}),$$
thus
\begin{eqnarray*}
\#(\sigma^{-1}\gamma_{p,q,l}) & \geq & \#(\pi_{\vec{n}}^{-1}\gamma_{p,q,l})+\#(\sigma^{-1}\pi_{\vec{n}})-n \\
&=& (n+3-\#(\pi_{\vec{n}}))+(2\#(\sigma\vee\pi_{\vec{n}})-\#(\pi_{\vec{n}})-\#(\sigma)) \\
&=& n+3-\#(\pi_{\vec{n}})+\#(\pi_{\vec{n}})-4-\#(\sigma),
\end{eqnarray*}
where in second equality we use that $\pi_{\vec{n}}\in \NC(p)\times\NC(q)\times\NC(l)$ and Proposition \ref{Proposition: Sufficient and neccesary conditions general version} with $\sigma\lesssim \pi_{\vec{n}}$. We conclude, $\#(\sigma)+\#(\sigma^{-1}\gamma_{p,q,l})+\#(\gamma_{p,q,l})\geq n+2$. Conversely, if $\sigma\in S_{NC}(p,q,l)$ is such that $\sigma \lesssim^{(2)} \pi_{\vec{n}}$, $\sigma^{-1}\pi_{\vec{n}}$ separates $N$ and $\sigma\vee\pi_{\vec{n}}=\V_{\vec{n}}$, there must be $4$ cycles of $\pi_{\vec{n}}$, $A_1,A_2,B_1,B_2$ such that any other cycle of $\pi_{\vec{n}}$ is a block of $\V_{\vec{n}}$ and $0_{A_1\cup B_1}$ and $0_{A_2\cup B_2}$ are both blocks of $\V_{\vec{n}}$. The condition $\sigma\vee\pi_{\vec{n}}=\V_{\vec{n}}$ means that for any cycle of $\pi_{\vec{n}}$, $B$, distinct of $A_1,B_1,A_2,B_2$, the restriction $\sigma|_B \in \NC(B)$, and for these 4 cycles we have $\sigma|_{A_1\cup B_1}\in S_{NC}(A_1.B_1)$ and $\sigma|_{A_2,B_2}\in S_{NC}(A_2,B_2)$. Moreover as $\sigma^{-1}\pi_{\vec{n}}$ separates $N$ then $\sigma|_{D}^{-1}\pi_{\vec{n}}|_D$ separates $N\cap D$ for any $D$ block of $\V_{\vec{n}}$. Therefore
$$\mathcal{P}\mathcal{A}=\sum_{\substack{\sigma\in S_{NC}(p,q,l) \\ \sigma \lesssim^{(2)} \pi_{\vec{n}} \\ \sigma^{-1}\pi_{\vec{n}}\text{ separates }N \\ \sigma \vee \pi_{\vec{n}}=\V_{\vec{n}}}} \kappa_{\sigma}(\vec{a}).$$
To finish the proof we claim that,
\begin{equation}\label{aux2}
\mathcal{P}\mathcal{B}+\mathcal{P}\mathcal{C}=\sum_{\substack{(\U,\sigma)\in \PS_{NC}^{(1)}(p,q,l) \\ \sigma \lesssim^{(1)} \pi_{\vec{n}} \\ \sigma^{-1}\pi_{\vec{n}}\text{ separates }N \\ \U\vee \pi_{\vec{n}}=\V_{\vec{n}}}} \kappa_{(\U,\sigma)}(\vec{a}),
\end{equation}
and
\begin{equation}\label{aux3}
\mathcal{P}\mathcal{D}=\sum_{\substack{(\U,\sigma)\in \PS_{NC}^{(2)}(p,q,l) \\ \sigma \leq \pi_{\vec{n}} \\ \sigma^{-1}\pi_{\vec{n}}\text{ separates }N \\ \U\vee \pi_{\vec{n}}=\V_{\vec{n}}}} \kappa_{(\U,\sigma)}(\vec{a}).
\end{equation}
Let us prove first Equation \ref{aux2}. In the sum $\mathcal{P}\mathcal{B}$ we let $\sigma=\sigma_{0,1}\times\sigma_{0,2}\sigma_1\times\cdots\times\sigma_w$ and $\U$ be the partition whose blocks are all cycles of $\sigma$ except by one block which is the union of two cycles of $\sigma_{0,2}$ which is the block of $\V_{0,2}$ that is the union of two cycles of $\sigma_{0,2}$. We have $\sigma\lesssim^{(1)}\pi_{\vec{n}}$, $\sigma^{-1}\pi_{\vec{n}}$ separates $N$ and $\U\vee\pi_{\vec{n}}=\V_{\vec{n}}$. Now we verify $(\U,\sigma)\in PS_{NC}^{(1)}(p,q,l)$. Observe that any block of $\sigma\vee\pi_{\vec{n}}$ is a cycle of $\pi_{\vec{n}}$ except the block $0_{\hat{A}_1\cup\hat{B}_1}$ which is the union of two cycles of $\pi_{\vec{n}}$. So we let $\sigma^{(1)}=\sigma|_{[p+q]}$ and $\sigma^{(2)}=\sigma|_{[p+q+1,p+q+l]}$ and similarly $\pi^{(1)}=\pi_{\vec{n}}|_{[p+q]}$ and $\pi^{(2)}=\pi_{\vec{n}}|_{[p+q+1,p+q+l]}$. In the set $[p+q+1,p+q+l]$ we have that $\sigma^{(2)}\leq \pi^{(2)}$ and since $\pi^{(2)}\in \NC(l)$ then so is $\sigma^{(2)}$. In the set $[p+q]$ we have $\sigma^{(1)}\lesssim^{(1)}\pi^{(1)}$ so proceeding as in Proposition \ref{Proposition: Suficiente conditions version 2} we get $\sigma^{(1)}\in S_{NC}(p,q)$. We have proved $\sigma\in S_{NC}(p,q)\times\NC(l)$. We conclude by nothing that the block of $\V$ which is union of two cycles of $\sigma$ is such that one of the cycles lies in $\hat{A}_2 \subset [p]$ while the other lies in $\hat{B}_2 \subset [p+q+1,p+q+l]$. We proceed analogously for $\mathcal{P}\mathcal{C}$, we let $(\U,\sigma)$ to be the same corresponding way as for $\mathcal{P}\mathcal{B}$, thus $(\U,\sigma)\in PS_{NC}^{(1)}(p,q,l)$, is such that $\sigma \lesssim^{(1)}\pi_{\vec{n}}$, $\sigma^{-1}\pi_{\vec{n}}$ separates $N$ and $\U\vee\pi_{\vec{n}}=\V_{\vec{n}}$. Conversely if we consider $(\U,\sigma)\in PS_{NC}^{(1)}(p,q,l)$ that satisfies all mentioned before, then there are two cycles of $\pi_{\vec{n}}$, say $A_1,B_1$, such that any block of $\sigma\vee\pi_{\vec{n}}$ is a cycle of $\pi_{\vec{n}}$ except by one block which is $0_{A_1\cup B_1}$. Moreover, for any cycle $D$ of $\pi_{\vec{n}}$, distinct of $A_1,B_1$, we have $\sigma|_D\in \NC(D)$ and $\sigma|_{A_1\cup B_1}\in S_{NC}(A_1,B_1)$. Suppose $\sigma\in S_{NC}(p,q)\times \NC(l)$, thus we are forced to either $A_1\subset [p]$ and $B_1\subset [p+1,p+q]$ or the other way around, otherwise there would be no cycle of $\sigma$ that meets $[p]$ and $[p+1,p+q]$. Suppose we are in the former case. On the other hand, there exist a block of $\U$ which is the union of two cycles of $\sigma$, say $a,b$. One of the cycles, say $a$ lies in $[p+q]$ while the other, $b$, lies in $[p+q+1,p+q+l]$. The cycle $b$ of $\sigma$ must be in a block of $\sigma\vee\pi_{\vec{n}}$ which must necessarily be a cycle of $\pi_{\vec{n}}$, we may call this cycle $B_2$. The cycle $a$ of $\sigma$ must be in a block of $\sigma\vee\pi_{\vec{n}}$, if this block is $0_{A_1\cup B_1}$ then $0_{A_1\cup B_1\cup B_2}$ is a block of $\U\vee\pi_{\vec{n}}$ and any other block of $\U\vee\pi_{\vec{n}}$ is a cycle of $\pi_{\vec{n}}$, this is impossible as $\U\vee\pi_{\vec{n}}=\V_{\vec{n}}\vee\pi_{\vec{n}}$ which has exactly two blocks, each one being the union of two cycles of $\pi_{\vec{n}}$ as $(\V_{\vec{n}},\pi_{\vec{n}})\in PS_{NC}^{(2)}(p,q,l)$. Hence, there must be another cycle, $A_2$ of $\pi_{\vec{n}}$ such that $a\subset A_2$. Suppose $A_2 \subset [p]$. We have that any block of $\U\vee \pi_{\vec{n}}$ is a cycle of $\pi_{\vec{n}}$ except by the blocks $0_{A_1 \cup B_1}$ and $0_{A_2\cup B_2}$ each one being the union of two cycles of $\pi_{\vec{n}}$. Any block of $\U$ is a cycle of $\sigma$ except $0_{a\cup b}$ which is the union of two cycles of $\sigma$, with $a \subset A_2$ and $b\in B_2$. It remains to let $\sigma_{0,1}=\sigma|_{A_1\cup B_1}$, $(\V_{0,2}\sigma_{0,2})$ to be $(\U,\sigma)$ restricted to $A_2\cup B_2$ and $\sigma_i=\sigma|_{C_i}$ for any other cycle $C_i$ of $\pi_{\vec{n}}$ to write $\kappa_{(\U,\sigma)}$ as in the sum $\mathcal{P}\mathcal{B}$, of course the elements that looks as in the sum $\mathcal{P}\mathcal{C}$ are also obtained when the assumptions that we made are distinct. This proves Equation \ref{aux2}. To prove Equation \ref{aux3}, in the sum $\mathcal{P}\mathcal{D}$ we let $\sigma$ to be as before and $\U$ to be the partition where any block is a cycle of $\sigma$ except by two blocks, each one being the union of two cycles of $\sigma$ and which are given precisely by the blocks of $\V_{0,1}$ and $\V_{0,2}$ that are the union of two cycles of $\sigma_{0,1}$ and $\sigma_{0,2}$ respectively. Now we have $\sigma \leq \pi_{\vec{n}}$ and then $\sigma\in \NC(p)\times\NC(q)\times\NC(l)$. One of the blocks of $\U$ joins one cycle of $\sigma$ in $\hat{A}_1 \subset [p]$ and a cycle of $\sigma$ in $\hat{B}_1 \subset [p+1,p+q]$. The other block of $\U$ joins one cycle of $\sigma$ in $\hat{A}_2\subset [p]$ and one cycle of $\sigma$ in $\hat{B}_2 \subset [p+q+1,p+q+l]$, thus $(\U,\pi)\in PS_{NC}^{(2)}(p,q,l)$. Moreover the conditions $\sigma^{-1}\pi_{\vec{n}}$ separates $N$ and $\U\vee\pi_{\vec{n}}=\V_{\vec{n}}$ are clearly satisfied. Conversely Let $(\U,\pi)\in PS_{NC}^{(2)}(p,q,l)$ that satisfies all mentioned before. Since $(\V_{\vec{n}},\pi_{\vec{n}})\in PS_{NC}^{(2)}(p,q,l)$, there exist 4 cycles of $\pi_{\vec{n}}$, which we may call $A_1,B_1,A_2,B_2$ such that any block of $\V_{\vec{n}}$ is a cycle of $\pi_{\vec{n}}$ except by the blocks $0_{A_1\cup B_1}$ and $0_{A_2\cup B_2}$, suppose $A_1,A_2 \subset [p]$, $B_1\subset [p+1,p+q]$ and $B_2 \subset [p+q+1,p+q+l]$. Each block of $\U$ is a cycle of $\sigma$ except by two blocks which is each the union of cycles of $\sigma$, let $a,b,c,d$ be these cycles of $\sigma$ so that $0_{a\cup b}$ and $0_{c\cup d}$ are blocks of $\U$. Since $\sigma\leq \pi_{\vec{n}}$ and $\U\vee\pi_{\vec{n}}=\V_{\vec{n}}$ then each one of $a,b,c,d$ must lie in one of $A_1,B_1,A_2,B_2$, suppose $a\in A_1$, $b\in B_2$, $c\in A_2$, $d\in B_2$. It remains to define $(\V_{0,1},\sigma_{0,1})$ to be $(\U,\sigma)$ restricted to $A_1\cup B_1$, $(\V_{0,2},\sigma_{0,2})$ to be $(\U,\sigma)$ restricted to $A_2\cup B_2$ and $\sigma_i$ to be $\sigma|_{C_i}$ for any other cycle, $C_i$ of $\pi_{\vec{n}}$ distinct from $A_1,B_1,A_2,B_2$, to write $\kappa_{(\U,\sigma)}$ as in sum $\mathcal{P}\mathcal{D}$, this proves Equation \ref{aux3}. 
\end{proof}

\begin{lemma}\label{Main theorem: Inductive step}
If Equation \ref{me} is satisfied for any $r^{\prime}\leq r$, $s^{\prime}\leq s$ and $t^{\prime}\leq t$ with $(r^{\prime},s^{\prime},t^{\prime})\neq (r,s,t)$, then for any $(\V,\pi) \in \PS_{NC}^{(3)}(r,s,t)\setminus \{(1_{r+s+t},\gamma_{r,s,t})\}$,
\begin{multline*}
\kappa_{(\V,\pi)}(A_1,\dots,A_{r+s+t})= \\
\sum_{\substack{\sigma\in S_{NC}(p,q,l) \\ \sigma \lesssim^{(2)} \pi_{\vec{n}} \\ \sigma^{-1}\pi_{\vec{n}}\text{ separates }N \\ \sigma \vee \pi_{\vec{n}}=\V_{\vec{n}}}} \kappa_{\sigma}(\vec{a}) + \sum_{\substack{(\U,\sigma)\in \PS_{NC}^{(1)}(p,q,l) \\ \sigma \lesssim^{(1)} \pi_{\vec{n}} \\ \sigma^{-1}\pi_{\vec{n}}\text{ separates }N \\ \U\vee \pi_{\vec{n}}=\V_{\vec{n}}}} \kappa_{(\U,\sigma)}(\vec{a}) \\
+ \sum_{\substack{(\U,\sigma)\in \PS_{NC}^{(2)}(p,q,l) \\ \sigma \leq \pi_{\vec{n}} \\ \sigma^{-1}\pi_{\vec{n}}\text{ separates }N \\ \U\vee \pi_{\vec{n}}=\V_{\vec{n}}}} \kappa_{(\U,\sigma)}(\vec{a}) + \sum_{\substack{(\U,\sigma)\in \PS_{NC}^{(3)}(p,q,l) \\ \sigma \leq \pi_{\vec{n}} \\ \sigma^{-1}\pi_{\vec{n}}\text{ separates }N \\ \U\vee \pi_{\vec{n}}=\V_{\vec{n}}}} \kappa_{(\U,\sigma)}(\vec{a}).
\end{multline*}
\end{lemma}
\begin{proof}
Let $C_1\cdots C_wABC$ be the cycle decomposition of $\pi$ so that the blocks of $\V$ are all cycles of $\pi$ corresponding to $0_{C_1},\dots,0_{C_w}$ except by one cycle which is the union of three cycles of $\pi$ corresponding to $0_{A\cup B\cup C}$. Assume that $A\subset [r]$, $B\subset [r+1,r+s]$ and $C\subset [r+s+1,r+s+t]$. We let $\tilde{C}_1\cdots \tilde{C}_w\tilde{A}\tilde{B}\tilde{C}$ be the corresponding cycle decomposition of $\pi_{\vec{n}}$. Thus
$$\kappa_{(\V,\pi)}(\vec{A})=\kappa_{|C_1|}(\vec{A})\cdots \kappa_{|C_w|}(\vec{A})\kappa_{|A|,|B|,|C|}(\vec{A}).$$
By \cite[Theorem 2.2]{KS},
$$\kappa_{|C_i|}(\vec{A})=\sum_{\substack{\sigma_i \in \NC(\hat{C}_i) \\ \sigma_i^{-1} \hat{C}_i \text{ separates }N\cap \tilde{C}_i}}\kappa_{\sigma_i}(\vec{a}).$$
By hypothesis since $(|A|,|B|,|C|)\neq (r,s,t)$,
$$\kappa_{|A|,|B|,|C|}(\vec{A})=\sum_{\substack{(\V_0,\sigma_0)\in PS_{NC}(\tilde{A},\tilde{B},\tilde{C}) \\ \sigma_0^{-1}\tilde{A}\tilde{B}\tilde{C}\text{ separates }N\cap\tilde{A}\cap\tilde{B}\cap\tilde{C}}}\kappa_{(\V_0,\sigma_0)}(\vec{a}).$$
We write the set $PS_{NC}(\tilde{A},\tilde{B},\tilde{C})$ as the disjoint union of the 4 sets, $S_{NC}(\tilde{A},\tilde{B},\tilde{C})$, $PS_{NC}^{(1)}(\tilde{A},\tilde{B},\tilde{C})$, $PS_{NC}^{(2)}(\tilde{A},\tilde{B},\tilde{C})$ and $PS_{NC}^{(3)}(\tilde{A},\tilde{B},\tilde{C})$ as in Lemma \ref{Solutions to third order partitioned permutations}, so that,
\begin{multline*}
\kappa_{|A|,|B|,|C|}(\vec{A})=\sum_{\substack{\sigma_0\in S_{NC}(\tilde{A},\tilde{B},\tilde{C}) \\ \sigma_0^{-1}\tilde{A}\tilde{B}\tilde{C}\text{ separates }N\cap\tilde{A}\cap\tilde{B}\cap\tilde{C}}}\kappa_{\sigma_0}(\vec{a})+\sum_{\substack{(\V_0,\sigma_0)\in PS_{NC}^{(1)}(\tilde{A},\tilde{B},\tilde{C}) \\ \sigma_0^{-1}\tilde{A}\tilde{B}\tilde{C}\text{ separates }N\cap\tilde{A}\cap\tilde{B}\cap\tilde{C}}}\kappa_{(\V_0,\sigma_0)}(\vec{a}) \\
+\sum_{\substack{(\V_0,\sigma_0)\in PS_{NC}^{(2)}(\tilde{A},\tilde{B},\tilde{C}) \\ \sigma_0^{-1}\tilde{A}\tilde{B}\tilde{C}\text{ separates }N\cap\tilde{A}\cap\tilde{B}\cap\tilde{C}}}\kappa_{(\V_0,\sigma_0)}(\vec{a})+\sum_{\substack{(\V_0,\sigma_0)\in PS_{NC}^{(3)}(\tilde{A},\tilde{B},\tilde{C}) \\ \sigma_0^{-1}\tilde{A}\tilde{B}\tilde{C}\text{ separates }N\cap\tilde{A}\cap\tilde{B}\cap\tilde{C}}}\kappa_{(\V_0,\sigma_0)}(\vec{a}).
\end{multline*}
Thus
$$\kappa_{(\V,\pi)}(\vec{A})=\mathcal{P}\mathcal{A}+\mathcal{P}\mathcal{B}+\mathcal{P}\mathcal{C}+\mathcal{P}\mathcal{D},$$
with $\mathcal{P}$ as in the proof of Lemma \ref{Lemma: Relation A and a, version 2} and,
$$\mathcal{A}=\sum_{\substack{\sigma_0\in S_{NC}(\tilde{A},\tilde{B},\tilde{C}) \\ \sigma_0^{-1}\tilde{A}\tilde{B}\tilde{C}\text{ separates }N\cap\tilde{A}\cap\tilde{B}\cap\tilde{C}}}\kappa_{\sigma_0}(\vec{a}),$$
$$\mathcal{B}=\sum_{\substack{(\V_0,\sigma_0)\in PS_{NC}^{(1)}(\tilde{A},\tilde{B},\tilde{C}) \\ \sigma_0^{-1}\tilde{A}\tilde{B}\tilde{C}\text{ separates }N\cap\tilde{A}\cap\tilde{B}\cap\tilde{C}}}\kappa_{(\V_0,\sigma_0)}(\vec{a}),$$
$$\mathcal{C}=\sum_{\substack{(\V_0,\sigma_0)\in PS_{NC}^{(2)}(\tilde{A},\tilde{B},\tilde{C}) \\ \sigma_0^{-1}\tilde{A}\tilde{B}\tilde{C}\text{ separates }N\cap\tilde{A}\cap\tilde{B}\cap\tilde{C}}}\kappa_{(\V_0,\sigma_0)}(\vec{a}),$$
and,
$$\mathcal{D}=\sum_{\substack{(\V_0,\sigma_0)\in PS_{NC}^{(3)}(\tilde{A},\tilde{B},\tilde{C}) \\ \sigma_0^{-1}\tilde{A}\tilde{B}\tilde{C}\text{ separates }N\cap\tilde{A}\cap\tilde{B}\cap\tilde{C}}}\kappa_{(\V_0,\sigma_0)}(\vec{a}).$$
Now we will proceed pretty much as in Lemmas \ref{Lemma: Relation A and a, version 2} and \ref{Lemma: Relation A and a, version 3}. In $\mathcal{P}\mathcal{A}$ we let $\sigma=\sigma_0\times\cdots \times\sigma_w$, so $\sigma^{-1}\pi_{\vec{n}}$ separates $N$ and $\sigma\vee\pi_{\vec{n}}=\V_{\vec{n}}$. At each block $\tilde{C}_i$ of $\pi_{\vec{n}}$ we have $\sigma|_{C_i}=\sigma_i \in \NC(\tilde{C}_i)$ while for the block $0_{\tilde{A}\cup\tilde{B}\cup\tilde{C}}$ of $\V_{\vec{n}}$ we have, $\sigma|_{\tilde{A}\cup\tilde{B}\cup\tilde{C}}=\sigma_0\in S_{NC}(\tilde{A},\tilde{B},\tilde{C})$, thus $\sigma \lesssim^{(2)} \pi_{\vec{n}}$. It remains to prove $\sigma\in S_{NC}(p,q,l)$. The condition $\sigma\vee\gamma_{p,q,l}=1_{n}$ is satisfied since $\sigma_0$ meets all three cycles $\tilde{A}\subset [p]$, $\tilde{B}\subset [p_1,p+q]$ and $\tilde{C}\subset [p+q+1,p+q+l]$. We are thus reduced to verify $\#(\sigma)+\#(\sigma^{-1}\gamma_{p,q,l})+\#(\gamma_{p,q,l})\geq n+2.$ By \cite[Equation 2.9]{MN},
$$\#(\pi_{\vec{n}}^{-1}\gamma_{p,q,l})+\#(\sigma^{-1}\pi_{\vec{n}})+\#(\sigma^{-1}\gamma_{p,q,l}) \leq n+2\#(\sigma^{-1}\gamma_{p,q,l}),$$
thus,
\begin{eqnarray*}
\#(\sigma^{-1}\gamma_{p,q,l}) & \geq & \#(\pi_{\vec{n}}^{-1}\gamma_{p,q,l})+\#(\sigma^{-1}\pi_{\vec{n}})-n \\
&=& (n+3-\#(\pi_{\vec{n}}))+(2\#(\sigma\vee\pi_{\vec{n}})-\#(\pi_{\vec{n}})-\#(\sigma)) \\
&=& n+3-\#(\pi_{\vec{n}})+\#(\pi_{\vec{n}})-4-\#(\sigma),
\end{eqnarray*}
where in second equality we use that $\pi_{\vec{n}}\in \NC(p)\times\NC(q)\times\NC(l)$ and Proposition \ref{Proposition: Sufficient and neccesary conditions general version} with $\sigma\lesssim \pi_{\vec{n}}$. We conclude, $\#(\sigma)+\#(\sigma^{-1}\gamma_{p,q,l})+\#(\gamma_{p,q,l})\geq n+2$. Conversely, let $\sigma\in S_{NC}(p,q,l)$ be such that $\sigma \lesssim^{(2)} \pi_{\vec{n}}$, $\sigma^{-1}\pi_{\vec{n}}$ separates $N$ and $\sigma\vee\pi_{\vec{n}}=\V_{\vec{n}}$. There must be $3$ cycles of $\pi_{\vec{n}}$, $A,B,C$ such that any other cycle of $\pi_{\vec{n}}$ is a block of $\V_{\vec{n}}$ and $0_{A\cup B\cup C}$ is a block of $\V_{\vec{n}}$. The conditions $\sigma\vee\pi_{\vec{n}}=\V_{\vec{n}}$ and $\sigma\lesssim^{(2)}\pi_{\vec{n}}$ means that for any cycle $D$ of $\pi_{\vec{n}}$ distinct of $A,B,C$ we have $\sigma|_D\in \NC(D)$, while for $A,B,C$ we have $\sigma|_{A\cup B\cup C}\in S_{NC}(A,B,C)$. Moreover as $\sigma^{-1}\pi_{\vec{n}}$ separates $N$ then $\sigma|_{D}^{-1}\pi_{\vec{n}}|_D$ separates $N\cap D$ for any $D$ block of $\V_{\vec{n}}$. Therefore
$$\mathcal{P}\mathcal{A}=\sum_{\substack{\sigma\in S_{NC}(p,q,l) \\ \sigma \lesssim^{(2)} \pi_{\vec{n}} \\ \sigma^{-1}\pi_{\vec{n}}\text{ separates }N \\ \sigma \vee \pi_{\vec{n}}=\V_{\vec{n}}}} \kappa_{\sigma}(\vec{a}).$$
To finish the proof our aim is to show,
$$\mathcal{P}\mathcal{B}=\sum_{\substack{(\U,\sigma)\in \PS_{NC}^{(1)}(p,q,l) \\ \sigma \lesssim^{(1)} \pi_{\vec{n}} \\ \sigma^{-1}\pi_{\vec{n}}\text{ separates }N \\ \U\vee \pi_{\vec{n}}=\V_{\vec{n}}}} \kappa_{(\U,\sigma)}(\vec{a}),$$
$$\mathcal{P}\mathcal{C}=\sum_{\substack{(\U,\sigma)\in \PS_{NC}^{(2)}(p,q,l) \\ \sigma \leq \pi_{\vec{n}} \\ \sigma^{-1}\pi_{\vec{n}}\text{ separates }N \\ \U\vee \pi_{\vec{n}}=\V_{\vec{n}}}} \kappa_{(\U,\sigma)}(\vec{a}),$$
and,
$$\mathcal{P}\mathcal{D}=\sum_{\substack{(\U,\sigma)\in \PS_{NC}^{(3)}(p,q,l) \\ \sigma \leq \pi_{\vec{n}} \\ \sigma^{-1}\pi_{\vec{n}}\text{ separates }N \\ \U\vee \pi_{\vec{n}}=\V_{\vec{n}}}} \kappa_{(\U,\sigma)}(\vec{a}).$$
Let us prove the last equality which illustrates the best how the proof proceeds and similar proofs follow for the other two equalities. In $\mathcal{P}\mathcal{D}$ we let $\sigma=\sigma_0\times \cdots \times\sigma_w$ and $\U$ be the partition whose blocks are all cycles of $\sigma$ except by one block which is the union of the three cycles of $\sigma_0$ that form a block of $\V_0$. We have $\sigma \leq \pi_{\vec{n}}$, $\sigma^{-1}\pi_{\vec{n}}$ separates $N$ and $\U \vee \pi_{\vec{n}}=\V_{\vec{n}}$. Now we verify $(\U,\sigma)\in PS_{NC}^{(3)}(p,q,l)$. Since $\sigma\leq\pi_{\vec{n}}$ and $\pi \in \NC(p)\times\NC(q)\times\NC(l)$ then $\sigma\in \NC(p)\times\NC(q)\times\NC(l)$, we conclude by observing that the unique block of $\U$ which is the union of three cycles of $\sigma$ is such that it joins one cycle in $\tilde{A}\subset [p]$, another cycle in $\tilde{B}\subset [p+1,p+q]$ and one more cycle in $\tilde{C}\subset [p+q+1,p+q+l]$. Conversely, let $(\U,\sigma)\in PS_{NC}^{(3)}(p,q,l)$ be such that $\sigma\leq \pi_{\vec{n}}$, $\sigma^{-1}\pi_{\vec{n}}$ separates $N$ and $\U\vee\pi_{\vec{n}}=\V_{\vec{n}}$. There are three cycles of $\pi_{\vec{n}}$, $A\subset [p]$, $B\subset [p+1,p+q]$ and $C\subset [p+q+1,p+q+l]$ such that $0_{A\cup B\cup C}$ is a block of $\V_{\vec{n}}$ and any other cycle of $\pi_{\vec{n}}$ is a block of $\V_{\vec{n}}$. Similarly, since $(\U,\sigma)\in PS_{NC}^{(3)}(p,q,l)$ there are three cycles of $\sigma$, $a\subset [p]$, $b\subset [p+1,p+q]$ and $c\subset[p+q+1,p+q+l]$ such that $0_{a\cup b\cup c}$ is a block of $\U$ and any other cycle of $\sigma$ is a block of $\U$. Since $\sigma\leq \pi_{\vec{n}}$ and $\U \vee \pi_{\vec{n}}=\V_{\vec{n}}$ then each of $a,b,c$ must lie in one of $A,B,C$, suppose $a\subset A$, $b\subset B$ and $c\subset C$. We finish by letting $(\V_0,\sigma_0)$ to be $(\U,\sigma)$ restricted to $A\cup B\cup C$ and $\sigma_i$ to be $\sigma|_{C_i}$ for any other cycle $C_{i}$ of $\pi_{\vec{n}}$ distinct from $A,B,C$. In this way we can write $\kappa_{(\U,\sigma)}$ as in summation $\mathcal{P}\mathcal{D}$ which proves the equality.
\end{proof}

\section{Proof of main theorem}\label{Section: Main theorem proof}

We will prove Theorem \ref{Main theorem cumulants with products as entries} by induction on $(r,s,t)$. So let us start with the case $r=s=t=1$.

\begin{lemma}
\begin{multline*}
\kappa_{1,1,1}(a_1\cdots a_p,a_{p+1}\cdots a_{p+q},a_{p+q+1}\cdots a_{p+q+l}) =  \sum_{(\V,\pi)\in \PS_{NC}(p,q,l)}\kappa_{(\V,\pi)}(a_1,\dots,a_{p+q+l})
\end{multline*}
where the summation is over those $(\V,\pi)\in \PS_{NC}(p,q,l)$ such that $\pi^{-1}\gamma_{p,q,l}$ separates the points of $N=\{p,p+q,p+q+l\}$.
\end{lemma}

\begin{proof}
Through all this proof, for two permutations $\pi,\gamma \in S_{p+q+l}$, we denote by $\Gamma_{\pi}^{\gamma}$ to the partition $0_{\pi^{-1}\gamma}$ restricted to $N$. We have two expressions for $\varphi_3(A_1,A_2,A_3)$, these are,
\begin{multline*}
\kappa_3(A_1,A_2,A_3)+\kappa_3(A_1,A_3,A_2)+\kappa_{1,2}(A_1,A_2,A_3) \\ +\kappa_{1,2}(A_2,A_1,A_3)+\kappa_{1,2}(A_3,A_1,A_2)+\kappa_{1,1,1}(A_1,A_2,A_3),
\end{multline*}
and
$$\sum_{\pi \in S_{NC}(p,q,l)}\kappa_{\pi}(\vec{a})+\sum_{(\V,\pi)\in \PS_{NC}^{(1)}(p,q,l)}\kappa_{(\V,\pi)}(\vec{a})+\sum_{(\V,\pi)\in \PS_{NC}^{\prime}(p,q,l)}\kappa_{(\V,\pi)}(\vec{a}),$$
with $\PS_{NC}^{\prime}(p,q,l)=\PS_{NC}^{(2)}(p,q,l)\cup \PS_{NC}^{(3)}(p,q,l)$. Thus
\begin{multline}\label{Kappa111}
\kappa_{1,1,1}(A_1,A_2,A_3) = \\
\sum_{\pi \in S_{NC}(p,q,l)}\kappa_{\pi}(\vec{a})+\sum_{(\V,\pi)\in \PS_{NC}^{(1)}(p,q,l)}\kappa_{(\V,\pi)}(\vec{a})+\sum_{(\V,\pi)\in \PS_{NC}^{\prime}(p,q,l)}\kappa_{(\V,\pi)}(\vec{a}) \\
-\kappa_3(A_1,A_2,A_3)-\kappa_3(A_1,A_3,A_2)-\kappa_{1,2}(A_1,A_2,A_3) \\ -\kappa_{1,2}(A_2,A_1,A_3)-\kappa_{1,2}(A_3,A_1,A_2).
\end{multline}
By \cite[Theorem 3]{MST} we have,
$$\kappa_{1,2}(A_1,A_2,A_3)=\sum_{(\V,\pi)\in \PS_{NC}(p,q+l)}\kappa_{(\V,\pi)}(\vec{a}),$$
where the sum is over $(\V,\pi)$ such that $\pi^{-1}\gamma_{p,q+l}$ separates the points of $N$. Recall that $\PS_{NC}(p,q+l)$ is the union of the two sets $S_{NC}(p,q+l)$ and $\PS_{NC}^\prime(p,q+l)$ where $(\V,\pi)\in \PS_{NC}^\prime(p,q+l)$ is such that $\pi=\pi_1\times\pi_2\in \NC(p)\times \NC(q+l)$ and any cycle of $\pi$ is a block of $\V$ except one block which is the union of two cycles of $\pi$, one from each $\pi_i$. Let $\pi\in S_{NC}(p,q+l)$ be such that $\pi^{-1}\gamma_{p,q+l}$ separates the points of $N$. It is clear that $\pi$ must connect $[p]$ to at least one of $[p+1,p+q]$ or $[p+q+1,p+q+l]$. Suppose $\pi$ only connects $[p]$ and $[p+1,p+q]$ then $\Gamma_{\pi}^{\gamma_{p,q,l}}$ has the singleton $\{p+q+l\}$ and therefore $\{p+q,p+q+l\}$ is contained in a block of $\Gamma_{\pi}^{\gamma_{p,q+l}}$ since $\pi^{-1}\gamma_{p,q+l}=\pi^{-1}\gamma_{p,q,l}(p+q,p+q+l)$. The latter is a contradiction and therefore it must be $\pi \vee \gamma_{p,q,l}=1$. Moreover,
$$\#(\pi)+\#(\pi^{-1}\gamma_{p,q,l})=\#(\pi)+\#(\pi^{-1}\gamma_{p,q+l})-1=n-1,$$
hence $\pi\in S_{NC}(p,q,l)$. On the other hand, let $(\V,\pi) \in \PS_{NC}^\prime(p,q+l)$ with $\pi= \pi_1\times\pi_2\in \NC(p)\times\NC(q+l)$. It must be that $\pi_2$ connects $[p+1,p+q]$ and $[p+q+1,p+q+l]$, otherwise $\pi$ acts disjointly on each cycle of $\gamma_{p,q,l}$ and hence $\Gamma_{\pi}^{\gamma_{p,q,l}}=\{p\}\{p+q\}\{p+q+l\}$ which leads the same contradiction as before. Moreover,
$$\#(\pi_2)+\#(\pi_2^{-1}\gamma_{q,l})=\#(\pi_2)+\#(\pi_2^{-1}\gamma_{q+l})-1=n,$$
which proves that $\pi_2\in S_{NC}(q,l)$ and therefore $(\V,\pi)\in \PS_{NC}^{(1)}(p,q,l)$. We conclude
\begin{equation}\label{kappa_1,2First}
\kappa_{1,2}(A_1,A_2,A_3)=\sum_{\pi\in S_{NC}(p,q,l)}\kappa_{\pi}(\vec{a})+\sum_{(\V,\pi)\in \PS_{NC}^{(1)}(p,q,l)}\kappa_{(\V,\pi)}(\vec{a}),
\end{equation}
where bot summations are over $\pi$ such that $\pi^{-1}\gamma_{p,q+l}$ separates the points of $N$, or equivalently $\Gamma_{\pi}^{\gamma_{p,q,l}}=\{p\}\{p+q,p+q+l\}$. To compute $\kappa_{1,2}(A_2,A_1,A_3)$ let $b_i \in \mathcal{A}$ be given by $b_i=a_i$ for $p+q+1 \leq i\leq p+q+l$, $b_i=a_{p+i}$ for $1\leq i\leq q$ and $b_i=a_{q+i}$ for $q+1\leq i\leq p+q$. Proceeding as before we get
$$\kappa_{1,2}(A_2,A_1,A_3)=\sum_{\pi\in S_{NC}(q,p,l)}\kappa_{\pi}(\vec{b})+\sum_{(\V,\pi)\in \PS_{NC}^{(1)}(q,p,l)}\kappa_{(\V,\pi)}(\vec{b}),$$
where bot summations are over $\pi$ such that $\Gamma_{\pi}^{\gamma_{q,p,l}}=\{q\}\{p+q,p+q+l\}$. To get the cumulants back in terms of $\vec{a}$ instead of $\vec{b}$ it is enough to relabel the permutations appropriately, thas is; we relabel $i$ as $i+p$ for any $1\leq i\leq q$ and $i$ by $i-q$ for any $q+1\leq i\leq p+q$. In this way each element $\pi$ in the $q,p,l$-annulus becomes an element in the $p,q,l$-annulus and the condition $\Gamma_{\pi}^{\gamma_{q,p,l}}=\{q\}\{p+q,p+q+l\}$ becomes $\Gamma_{\pi}^{\gamma_{p,q,l}}=\{p+q\}\{p,p+q+l\}$, therefore,
\begin{equation}\label{kappa_1,2Second}
\kappa_{1,2}(A_2,A_1,A_3)=\sum_{\pi\in S_{NC}(p,q,l)}\kappa_{\pi}(\vec{a})+\sum_{(\V,\pi)\in \PS_{NC}^{(1)}(p,q,l)}\kappa_{(\V,\pi)}(\vec{a}),
\end{equation}
where bot summations are over $\pi$ such that $\Gamma_{\pi}^{\gamma_{p,q,l}}=\{p+q\}\{p,p+q+l\}$. Analogously,
\begin{equation}\label{kappa_1,2Third}
\kappa_{1,2}(A_3,A_1,A_2)=\sum_{\pi\in S_{NC}(p,q,l)}\kappa_{\pi}(\vec{a})+\sum_{(\V,\pi)\in \PS_{NC}^{(1)}(p,q,l)}\kappa_{(\V,\pi)}(\vec{a}),
\end{equation}
where bot summations are over $\pi$ such that $\Gamma_{\pi}^{\gamma_{p,q,l}}=\{p+q+l\}\{p,p+q\}$.
Combining Equations (\ref{kappa_1,2First},\ref{kappa_1,2Second},\ref{kappa_1,2Third}) with Equation (\ref{Kappa111}) yields
\begin{eqnarray*}
\kappa_{1,1,1}(A_1,A_2,A_3) =
\sum_{\pi \in S_{NC}(p,q,l) }\kappa_{\pi}(\vec{a})+\sum_{(\V,\pi)\in \PS_{NC}^{(1)}(p,q,l)}\kappa_{(\V,\pi)}(\vec{a}) \\
 + \sum_{(\V,\pi)\in \PS_{NC}^{\prime}(p,q,l)}\kappa_{(\V,\pi)}(\vec{a}) -\kappa_3(A_1,A_2,A_3)-\kappa_3(A_1,A_3,A_2),  
\end{eqnarray*}
where the first two sums are over $\pi$ such that $\Gamma_{\pi}^{\gamma_{p,q,l}}$ is either $\{p,p+q,p+q+l\}$ or $\{p\}\{p+q\}\{p+q+l\}$. It is easy to observe that there is no $(\V,\pi)\in \PS_{NC}^{(1)}(p,q,l)$ such that $\Gamma_{\pi}^{\gamma_{p,q,l}}=\{p,p+q,p+q+l\}$ while in the third sum any $\pi$ satisfies $\Gamma_{\pi}^{\gamma_{p,q,l}}=\{p\}\{p+q\}\{p+q+l\}$. Hence
\begin{eqnarray*}
\kappa_{1,1,1}(A_1,A_2,A_3) =
\sum_{\pi \in S_{NC}(p,q,l) }\kappa_{\pi}(\vec{a})+\sum_{(\V,\pi)\in \PS_{NC}(p,q,l)\setminus S_{NC}(p,q,l)}\kappa_{(\V,\pi)}(\vec{a}) \\
  -\kappa_3(A_1,A_2,A_3)-\kappa_3(A_1,A_3,A_2),  
\end{eqnarray*}
where the first summation is over $\pi$ such that $\Gamma_{\pi}^{\gamma_{p,q,l}}$ is either $\{p,p+q,p+q+l\}$ or $\{p\}\{p+q\}\{p+q+l\}$ and the second sum is over $\pi$ such that $\Gamma_{\pi}^{\gamma_{p,q,l}}=\{p\}\{p+q\}\{p+q+l\}$. So, we are reduce to prove that
\begin{equation}\label{Equation_First_Level}
\kappa_3(A_1,A_2,A_3)+\kappa_3(A_1,A_3,A_2)=\sum_{\pi \in S_{NC}(p,q,l) }\kappa_{\pi}(\vec{a}),
\end{equation}
where the summation is over $\pi$ such that $\Gamma_{\pi}^{\gamma_{p,q,l}}=\{p,p+q,p+q+l\}$.
For $\pi\in S_{NC}(p,q,l)$ such that $\Gamma_{\pi}^{\gamma_{p,q,l}}=\{p,p+q,p+q+l\}$, we have $\Gamma_{\pi}^{\gamma_{p,q,l}(p,p+q)}$ is either $\{p,p+q+l\}\{p+q\}$ or $\{p\}\{p+q,p+q+l\}$, we then write the right hand side of Equation (\ref{Equation_First_Level}) as
$$\sum_{\pi \in \mathcal{D}}\kappa_{\pi}(\vec{a})+\sum_{\pi \in \mathcal{E}}\kappa_{\pi}(\vec{a}),$$
where $\mathcal{D}$ are permutations satisfying the first condition while $\mathcal{E}$ are the ones satisfying the second condition. On the other hand, by \cite[Lema 14]{MST} and \cite[Theorem 11.12]{NS} we know
$$\kappa_3(A_1,A_2,A_3)=\sum_{\pi\in \NC(p+q+l)}\kappa_{\pi}(\vec{a}),$$
where the summation is over $\pi$ such that $\pi^{-1}\gamma_{p+q+l}$ separates the points of $N$. Let $\pi \in \NC(p+q+l)$ be such that $\pi^{-1}\gamma_{p+q+l}$ separates the points of $N$, then $\gamma_{\pi}^{\gamma_{p,q,l}}=\{p,p+q,p+q+l\}$ because $\pi^{-1}\gamma_{p,q,l}=\pi^{-1}\gamma_{p+q+l}(p,p+q)(p+q,p+q+l)$. The latter means $\pi \vee \gamma_{p,q,l}=1$, moreover,
$$\#(\pi)+\#(\pi^{-1}\gamma_{p,q,l})=\#(\pi)+\#(\pi^{-1}\gamma_{p+q+l})-2=n-1,$$
which means $\pi\in S_{NC}(p,q,l)$, and $\Gamma_{\pi}^{\gamma_{p,q,l}(p,p+q)}=\{p\}\{p+q,p+q+l\}$ since $\gamma_{p,q,l}(p,p+q)=\gamma_{p+q+l}(p+q,p+q+l)$, we conclude $\pi\in \mathcal{E}$. Similarly if $\pi\in \mathcal{E}$ then
\begin{eqnarray*}
\#(\pi)+\#(\pi^{-1}\gamma_{p+q+l}) &=& \#(\pi)+\#(\pi^{-1}\gamma_{p,q,l}(p,p+q)(p+q,p+q+l)) \\
&=& \#(\pi)+\#(\pi^{-1}\gamma_{p,q,l}(p,p+q))+1 \\
&=& \#(\pi)+\#(\pi^{-1}\gamma_{p,q,l}+2=n+1, \\
\end{eqnarray*}
and clearly $\pi^{-1}\gamma_{p+q+l}$ separates the points of $N$. Therefore
$$\kappa_3(A_1,A_2,A_3)=\sum_{\pi \in \mathcal{E}}\kappa_{\pi}(\vec{a}).$$
To compute $\kappa_3(A_1,A_3,A_2)$ we let $b_i=a_i$ for $1\leq i\leq p$, $b_i=a_{i+q}$ for $p+1\leq i\leq p+l$ and $b_i=a_{i-l}$ for $p+l+1\leq i \leq p+q+l$. Then by the \cite[Lema 14]{MST} and \cite[Theorem 11.12]{NS},
$$\kappa_3(A_1,A_3,A_2)=\sum_{\pi \in \NC(p+q+l)}\kappa_{\pi}(\vec{b}),$$
where the summation is over $\pi$ such that $\pi^{-1}\gamma_{p,l,q}$ separates the points of $\{p,p+l,p+l+q\}$. Let $\pi$ be as before, then as proved before we know $\pi \in S_{NC}(p,l,q)$. As done before we may relabel the values of $\pi$ in the following way, $i$ is relabeled as $i+q$ for $p+1 \leq i \leq p+l$ and $i$ becomes $i-l$ for $p+l+1 \leq i\leq p+q+l$, in this way $\pi$ becomes a permutation in $S_{NC(p,q,l)}$ and the condition $\pi^{-1}\gamma_{p,l,q}$ separates the points of $\{p,p+l,p+l+q\}$ becomes $\pi^{-1}\hat{\gamma}$ separates the points of $\{p,p+q+l,p+q\}$ where $\hat{\gamma}=(1,\dots,p,p+q+1,\dots,p+q+l,p+1,\dots,p+q)$, moreover with this relabeling we can substitute $\kappa_{\pi}(\vec{b})$ by $\kappa_{\pi}(\vec{a})$. Observe that
$$\hat{\gamma}(p,p+q)(p+q,p+q+l)=\gamma_{p,q,l}.$$
Therefore, if  $\pi^{-1}\hat{\gamma}$ separates the points of $\{p,p+q+l,p+q\}$ then $\pi^{-1}\gamma_{p,q,l}$ join them and since
$$\pi^{-1}\gamma_{p,q,l}(p,p+q)=\pi^{-1}\hat{\gamma}(p,p+q)(p+q,p+q+l)(p,p+q)=\pi^{-1}\hat{\gamma}(p,p+q+l),$$
then $\Gamma_{\pi}^{\gamma_{p,q,l}(p,p+q)}=\{p+q\}\{p,p+q+l\}$, i.e. $\pi\in \mathcal{D}$. Conversely, let $\pi\in \mathcal{D}$, then
\begin{eqnarray*}
\#(\pi)+\#(\pi^{-1}\gamma_{p+q+l}) & = & \#(\pi)+\#(\pi^{-1}\gamma_{p,q,l}(p,p+q)(p,p+q+l)) \\
& = & \#(\pi)+\#(\pi^{-1}\gamma_{p,q,l})+2=n+1,
\end{eqnarray*}
and $\pi^{-1}\hat{\gamma}$ separates the points of $N$ because
$$\pi^{-1}\hat{\gamma}=\pi^{-1}\gamma_{p,q,l}(p,p+q)(p,p+q+l),$$
which proves,
$$\kappa_3(A_1,A_3,A_2)=\sum_{\pi \in \mathcal{D}}\kappa_{\pi}(\vec{a}),$$
as desired.
\end{proof}

Now we are ready to prove our main theorem. The main lemmas used in the proof are posed and proved in Appendix \ref{append} for a fluent reading.

\begin{proof}[Proof of the Main Theorem.]

We proved the case $r=s=t=1$ so we will suppose Equation \ref{me} is true for any $r^{\prime}\leq r$, $s^{\prime}\leq s$ and $t^{\prime}\leq t$ with $(r^{\prime},s^{\prime},t^{\prime})\neq (r,s,t)$. Our goal is to prove Equation \ref{me} for $r,s,t$. We write $$\varphi_3(A_1\cdots A_r,A_{r+1}\cdots A_{r+s},A_{r+s+1}\cdots A_{r+s+t})$$ in two distinct ways,
$$\sum_{\pi\in S_{NC}(p,q,l)}\kappa_{\pi}(\vec{a})+\sum_{(\V,\pi)\in PS_{NC}^{(1)}(p,q,l)}\kappa_{(\V,\pi)}(\vec{a})+\sum_{(\V,\pi)\in PS_{NC}^{(2)}(p,q,l)}\kappa_{(\V,\pi)}(\vec{a})+\sum_{(\V,\pi)\in PS_{NC}^{(3)}(p,q,l)}\kappa_{(\V,\pi)}(\vec{a}),$$
and
$$\sum_{\pi\in S_{NC}(r,s,t)}\kappa_{\pi}(\vec{A})+\sum_{(\V,\pi)\in PS_{NC}^{(1)}(r,s,t)}\kappa_{(\V,\pi)}(\vec{A})+\sum_{(\V,\pi)\in PS_{NC}^{(2)}(r,s,t)}\kappa_{(\V,\pi)}(\vec{A})+\sum_{(\V,\pi)\in PS_{NC}^{(3)}(r,s,t)}\kappa_{(\V,\pi)}(\vec{A}).$$
In the last summation we can take the term $(1_{r+s+t},\gamma_{r,s,t})$ out of the sum so that solving for $\kappa_{r,s,t}(\vec{A})=\kappa_{(1_{r+s+t},\gamma_{r,s,t})}(\vec{A})$ gives
\begin{multline}\label{aux4}
\kappa_{r,s,t}(\vec{A})= \\
\sum_{\pi\in S_{NC}(p,q,l)}\kappa_{\pi}(\vec{a})+\sum_{(\V,\pi)\in PS_{NC}^{(1)}(p,q,l)}\kappa_{(\V,\pi)}(\vec{a})+\sum_{(\V,\pi)\in PS_{NC}^{(2)}(p,q,l)}\kappa_{(\V,\pi)}(\vec{a})+\sum_{(\V,\pi)\in PS_{NC}^{(3)}(p,q,l)}\kappa_{(\V,\pi)}(\vec{a}) \\
-\sum_{\pi\in S_{NC}(r,s,t)}\kappa_{\pi}(\vec{A})-\sum_{(\V,\pi)\in PS_{NC}^{(1)}(r,s,t)}\kappa_{(\V,\pi)}(\vec{A})-\sum_{(\V,\pi)\in PS_{NC}^{(2)}(r,s,t)}\kappa_{(\V,\pi)}(\vec{A})-\sum_{(\V,\pi)\in PS_{NC}^{(3)\prime}(r,s,t)}\kappa_{(\V,\pi)}(\vec{A}),
\end{multline}
with $PS_{NC}^{(3)\prime}(r,s,t)=PS_{NC}^{(3)}(r,s,t)\setminus \{(1_{r+s+t},\gamma_{r,s,t})\}.$ We use Lemmas \ref{Lemma: Relation A and a, version 1}, \ref{Lemma: Relation A and a, version 2}, \ref{Lemma: Relation A and a, version 3} and \ref{Main theorem: Inductive step} combined with Equation \ref{aux4} to write $\kappa_{r,s,t}(\vec{A})$ as the sum of four terms,
$$\kappa_{r,s,t}(\vec{A})=\mathcal{A}+\mathcal{B}+\mathcal{C}+\mathcal{D},$$
where,
\begin{multline*}
\mathcal{A}=\sum_{\pi\in S_{NC}(p,q,l)}\kappa_{\pi}(\vec{a})-\sum_{\pi\in S_{NC}(r,s,t)}\sum_{\substack{\sigma\in S_{NC}(p,q,l) \\ \sigma \leq \pi_{\vec{n}} \\ \sigma^{-1}\pi_{\vec{n}}\text{ separates }N}} \kappa_{\sigma}(\vec{a}) \\
-\sum_{(\V,\pi)\in PS_{NC}^{(1)}(r,s,t)}\sum_{\substack{\sigma\in S_{NC}(p,q,l) \\ \sigma \lesssim^{(1)} \pi_{\vec{n}} \\ \sigma^{-1}\pi_{\vec{n}}\text{ separates }N \\ \sigma \vee \pi_{\vec{n}}=\V_{\vec{n}}}} \kappa_{\sigma}(\vec{a})-\sum_{(\V,\pi)\in PS_{NC}^{(2)}(r,s,t)\cup PS_{NC}^{(3)\prime}(r,s,t)}\sum_{\substack{\sigma\in S_{NC}(p,q,l) \\ \sigma \lesssim^{(2)} \pi_{\vec{n}} \\ \sigma^{-1}\pi_{\vec{n}}\text{ separates }N \\ \sigma \vee \pi_{\vec{n}}=\V_{\vec{n}}}} \kappa_{\sigma}(\vec{a}),
\end{multline*}
\begin{multline*}
\mathcal{B}=\sum_{(\V,\pi)\in PS_{NC}^{(1)}(p,q,l)}\kappa_{(\V,\pi)}(\vec{a})-\sum_{(\V,\pi)\in PS_{NC}^{(1)}(r,s,t)}\sum_{\substack{(\U,\sigma)\in \PS_{NC}^{(1)}(p,q,l) \\ \sigma \leq \pi_{\vec{n}} \\ \sigma^{-1}\pi_{\vec{n}}\text{ separates }N \\ \U\vee\pi_{\vec{n}}=\V_{\vec{n}}}} \kappa_{(\U,\sigma)}(\vec{a}) \\
-\sum_{(\V,\pi)\in PS_{NC}^{(2)}(r,s,t)\cup PS_{NC}^{(3)\prime}(r,s,t)}\sum_{\substack{(\U,\sigma)\in \PS_{NC}^{(1)}(p,q,l) \\ \sigma \lesssim^{(1)} \pi_{\vec{n}} \\ \sigma^{-1}\pi_{\vec{n}}\text{ separates }N \\ \U\vee \pi_{\vec{n}}=\V_{\vec{n}}}} \kappa_{(\U,\sigma)}(\vec{a}),
\end{multline*}
$$\mathcal{C}=\sum_{(\V,\pi)\in PS_{NC}^{(2)}(p,q,l)}\kappa_{(\V,\pi)}(\vec{a})-\sum_{(\V,\pi)\in PS_{NC}^{(2)}(r,s,t)\cup PS_{NC}^{(3)\prime}(r,s,t)}\sum_{\substack{(\U,\sigma)\in \PS_{NC}^{(2)}(p,q,l) \\ \sigma \leq \pi_{\vec{n}} \\ \sigma^{-1}\pi_{\vec{n}}\text{ separates }N \\ \U\vee \pi_{\vec{n}}=\V_{\vec{n}}}} \kappa_{(\U,\sigma)}(\vec{a}),$$
and,
$$\mathcal{D}=\sum_{(\V,\pi)\in PS_{NC}^{(3)}(p,q,l)}\kappa_{(\V,\pi)}(\vec{a})-\sum_{(\V,\pi)\in PS_{NC}^{(3)\prime}(r,s,t)}\sum_{\substack{(\U,\sigma)\in \PS_{NC}^{(3)}(p,q,l) \\ \sigma \leq \pi_{\vec{n}} \\ \sigma^{-1}\pi_{\vec{n}}\text{ separates }N \\ \U\vee \pi_{\vec{n}}=\V_{\vec{n}}}} \kappa_{(\U,\sigma)}(\vec{a}).$$
In the term $\mathcal{A}$, by Lemmas \ref{Proposition: Induction version 1}, \ref{Proposition: Induction version 2} and \ref{Proposition: Induction version 3} we get,
$$\mathcal{A}=\sum_{\substack{\pi\in S_{NC}(p,q,l) \\ \pi^{-1}\gamma_{p,q,l}\text{ separates }N}}\kappa_{\pi}(\vec{a}).$$
In the term $\mathcal{B}$, by Lemmas  \ref{Proposition: Induction version 4} and \ref{Proposition: Induction version 5} we get,
$$\mathcal{B}=\sum_{\substack{(\V,\pi)\in PS_{NC}^{(1)}(p,q,l) \\ \pi^{-1}\gamma_{p,q,l}\text{ separates }N}}\kappa_{(\V,\pi)}(\vec{a}).$$
Finally by Lemmas  \ref{Proposition: Induction version 6} and \ref{Proposition: Induction version 7} we have that,
$$\mathcal{C}=\sum_{\substack{(\V,\pi)\in PS_{NC}^{(2)}(p,q,l) \\ \pi^{-1}\gamma_{p,q,l}\text{ separates }N}}\kappa_{(\V,\pi)}(\vec{a}),$$
and,
$$\mathcal{D}=\sum_{\substack{(\V,\pi)\in PS_{NC}^{(3)}(p,q,l) \\ \pi^{-1}\gamma_{p,q,l}\text{ separates }N}}\kappa_{(\V,\pi)}(\vec{a}),$$
which finishes the proof.
\end{proof}

\section{Applications}\label{Section: Application of the theorem}

In this section we present various examples motivated from important Ensambles of Random Matrices. Our first two applications will focus on computing the fluctuation cumulants of operators related to Gaussian Unitary Ensambles and Gaussian Wishart Matrices.

One of the most important ensambles in Random Matrix Theory is a  
self-adjoint normalized matrix with standard complex Gaussian entries, say $G_N$. In the asymptotic limit, this corresponds to a semicircular operator, $s$. To explore this, subsection \ref{semicircle} focuses on the square $s^2$ of a third order semicircular variable. This leads to results on the third order fluctuation cumulants of \( G_N^2 \). The second order case of this problem was previously analyzed in \cite{AM, MST}. 

Secondly, we consider another important example  in random matrix theory concerning Wishart matrices with Gaussian entries and a given covariance matrix \cite{W,LW}.  In the asymptotic limit, these correspond to the product \(cac^*\), where $a$ is an operator which is third order free with $c$, and $c$ is a third order circular operator. In subsection \ref{cac} we prove that the third order cumulants of $cac^*$ are exaclty the  moments of $a$, as in the first and second order case. The second order case of this problem was previously studied in \cite{AM}.

Next, we consider what we call $R$-diagonal elements of third order.
$R$-diagonal elements were introduced by Nica and Speicher in \cite{NS2} as a unifying concept encompassing both Haar unitary and circular elements. They have also been studied in connection with various topics, including the spectral distribution measure of non-normal elements in finite von Neumann algebras \cite{HS}, the fusion rules of irreducible representations of compact matrix quantum groups \cite{T}, and the Feinberg-Zee “single ring theorem” \cite{GKZ}. In particular, within random matrix theory, $R$-diagonal elements characterize the limiting distribution of matrices with rationally invariant spectra.

Building on these ideas, we now explore the extension to third order $R$-diagonal elements and investigate their potential applications.
In order to extend this concept to higher orders, we begin by defining $R$-diagonal in terms of cumulants. Specifically, we say that an operator is $R$-diagonal if the following holds:

\begin{definition}\label{ad1}
    Let $(\mathcal{A},\varphi)$ be a $*$-non-commutative probability space.  A random variable $a \in \mathcal{A}$ is called \textbf{$R$-diagonal} if for all $n \in \mathbb{N}$ we have $\kappa_n(a_1, \dots, a_n) = 0$ whenever the arguments $a_1, \dots, a_n \in \{a, a^*\}$ are not alternating in $a$ and $a^*$.

\end{definition}

One key property of $R$-diagonal elements is that they are closed by multiplying free elements. For instance, since a circular element $c$ is $R$-diagonal, the product $c_1c_2 \dots c_k$ is also $R$-diagonal whenever the $c_i$ are free. This example is particularly relevant, as it corresponds to the product of independent Ginibre matrices.

 The  extension to second order $R$-diagonal elements was studied in \cite{AM}, leading to the following definition.

\begin{definition}
    Let $(\mathcal{A}, \varphi, \varphi_2)$ be a second order $*$-non-commutative probability space. An element $a \in (\mathcal{A}, \varphi, \varphi_2)$ is called \textbf{second order $R$-diagonal} if it is $R$-diagonal (i.e. as in Definition \ref{ad1}) and the only non-vanishing second order cumulants are of the form
    \[
    \kappa_{2p,2q}(a, a^*, \dots, a, a^*) = \kappa_{2p,2q}(a^*, a, \dots, a^*, a).
    \]

\end{definition}

This definition resulted in significant applications in random matrix theory, which were thoroughly developed in \cite{AM}: They extended the results of Dubach and Peled \cite{DP} on the fluctuation moments of products of Ginibre matrices to general $*$-moments and, in particular, they provided a proof of the conjectured formula of Dartois and Forrester \cite{DF} for the fluctuation moments of the product of two independent complex Wishart matrices and generalized it to an arbitrary number of factors. 

In subsections \ref{Rdiagonal1} and \ref{Rdiagonal2}, we develop the theory of third order \( R \)-diagonal operators.  We compute the third order cumulants of \( aa^* \), where \( a \) is a third order \( R \)-diagonal operator. Additionally, we prove that third order \( R \)-diagonality is preserved under multiplication by a free element.  

Finally,  in subsection \ref{productGinibres}, building on the previous examples of the section,  we consider the important example of products of  third order free circular elements. As we mentioned above, this corresponds to products of Ginibre matrices. We give explicit expressions for its cumulants and its moments up to  third order, thus generalizing the results of \cite{AM} and  \cite{DF}, to the third order level.

\subsection{$s^2$} \label{semicircle}
As a first sample of the use of the main theorem of this article. We consider the following definition of a third order semicircular variable. This definition is motivated by the limiting case of GUE matrices $G_N$, see \cite[Theorem 1.1]{MM}. We will be interested in the third order free  cumulants of $s^2$, i.e., the limiting cumulants of $G_N^2$.

\begin{definition}[Semicircular  operator]

    A self-adjoint random variable $s$ in a third order $*$-non-commutative probability space is called a third order semicircular operator if its first
order cumulants satisfy $\kappa_ n(s,\dots,s) = 0$ for all $n \neq 2$ and $\kappa _2(s,s) = 1$,
and for all $p,q$ and $r$ the second and third order cumulants $\kappa_{p,q}$ and $\kappa_{p,q,r}$ are 0.    
\end{definition}

\begin{example}\label{feo}
From  Theorem \ref{Main theorem cumulants with products as entries} we have that $$\kappa_{p,q,r}(s^2,\dots,s^2)=\sum_{(\V,\pi)\in \PS_{NC}(2p,2q,2r)}\kappa_{(\V,\pi)}(s,\dots ,s),$$
where the summation is over those $\PS_{NC}(2p,2q,2r)$ such that $\pi^{-1}\gamma_{2p,2q,2r}$ separates the points of $N:=\{2,4,\dots,2p+2q+2r\}.$ 
According to the definition of a semicircular operator, there are specific instances where we do know the cumulants vanish. Here, the problem reduces to analyzing the cumulant on $(0_{\pi},\pi)$ with $\pi$ a pairing  in $S_{NC}(2p,2q,2r)$. Thus, given the previous, it follows that

$$\kappa_{p,q,r}(s^2,\dots,s^2)=|\{\pi\in S_{NC}(2p,2q,2r) ~|~ \pi \text{ pairing and }\pi\gamma_{2p,2q,2r} \text{ separates } N\}|.$$

Since \(\pi\) is a pairing, we have \(\#\pi = p + q + r\). Given that, \(\pi \in S_{NC}(2p, 2q, 2r)\), it follows that \(\# \pi \gamma_{2p, 2q, 2r} = p + q + r - 1\). Therefore, it is impossible for a pairing \(\pi \in S_{NC}(2p, 2q, 2r)\) to satisfy the separability condition. Specifically, \(\pi \gamma_{2p, 2q, 2r}\) cannot separate \(N\), as it would require at least \(p + q + r\) cycles. Which allows us to conclude that 
\[
\kappa_{p,q,r}(s^2,\dots,s^2) = 0.
\]
\end{example}

\subsection{$cac^*$} \label{cac}
As a second sample of the use of the main theorem of this article, we will consider the operator $cac^*$ where $a$ and $\{c,c^*\}$ are third order free and $c$ is a third order circular operator; our derivation is similar to the one in \cite{AM}.
Let us give a precise definition of a third order circular operator.

\begin{definition}[Circular operator]
Consider $s_{1}$ and $s_
    {2}$ third order free semicircular operators. We call $c =\frac{s_1+is_2}{\sqrt{2}}$
a third order circular
operator.

\end{definition}

\begin{example}\label{e1}

Given $c$  a third order circular operator such that $\{c,c^*\}$  and $\{a\}$ are third order
free, we are interested in the third order cumulants of $cac^*$. Based on the previous definition, we can prove that for the operator $c$ the only non-vanishing 
 cumulants are $\kappa_{2}(c,c^*)=\kappa_{2}(c^*,c)=1.$   
 From Theorem \ref{Main theorem cumulants with products as entries} we have that

$$\kappa_{r,s,t}(cac^*,cac^*,\dots,cac^*,cac^*)=\sum_{(\V,\pi)\in \PS_{NC}(3r,3s,3t)}\kappa_{(\V,\pi)}(c,a,c^*,\dots ,c,a,c^*),$$
where the summation is over those $\PS_{NC}(3r,3s,3t)$ such that $\gamma_{3r,3s,3t}\pi^{-1}$ separates the points of $O=\{1,4,\dots,3r+3s+3t-2\}$ or equivalently $\pi^{-1}\gamma_{3r,3s,3t}$ separates the points of $N=\{3,6,\dots,3r+3s+3t\}.$ According to the hypothesis, there are specific instances where we do know the cumulants vanish. Thus, the next step is to identify these cases and exclude them from the aforementioned summation. Since \(\{c, c^*\}\) and \(\{a\}\) are free of third order, it follows that the blocks of \(V\) consist either of positions corresponding to \(c\) or \(c^*\) (referred to as \(c\)-blocks, contained in \(V_c\)) or positions corresponding to \(a\) (referred to as \(a\)-blocks, contained in \(V_a\)). As a consequence the same is true for the cycles of $\pi$, which will be called $a$-cycles ($\pi_a$) and $c$-cycles ($\pi_c$), respectively. Moreover, we know that $c$ is circular, then  each $c$-cycle have to be of the form $(3i,3j-2)$ and $V_c=0_{\pi_{c}}$ for $i,j \in \{1,2,\dots,r+s+t\}$. In fact, $\pi(3i)=\gamma_{3r,3s,3t}(3i)$ since $\pi^{-1}\gamma_{3r,3s,3t}(N)= N$ and 
 $\pi^{-1}\gamma_{3r,3s,3t}$ separates the points of $N$. Hence, $\pi=(3,\gamma_{3r,3s,3t}(3))\dots(3r+3s+3t,3r+3s+1)\times \pi_a,$
and $V=\{\{3,\gamma_{3r,3s,3t}(3)\},\dots,\{3r+3s+3t,3r+3s+1\}\} \cup V_a,$ (see Figure \ref{fe1}).  Now, all boils down to analyze the properties of $(\V_a,\pi_a) \in \PS$ inherited from the hypothesis. For that, let´s embed it into $\{1,2,\dots,r+s+t\}$ by means of the  bijection $f(x)=(x+1)/3$ from $\{2,5,\dots,3r+3s+3t-1\}$ to $\{1,2,\dots,r+s+t\},$ i.e, $\pi_{f(a)}(i)=(\pi(3i-1)+1)/3$ and $V_{f(a)}$ joins $i$ and $j$ in a block if and only if $V_{a}$ joins $3i-1$ and $3j-1$ in a block. From the construction and the previous conditions, we can directly infer that
\begin{enumerate}

   \item $\#\pi=\#\pi_{f(a)}+r+s+t.$
    \item $\#V=\#V_{f(a)}+r+s+t$.
    \item $V_{f(a)} \vee \gamma_{r,s,t}=1_{r+s+t}.$
    \item $\kappa_{(\V,\pi)}(c,a,c^*\dots, c,a,c^*)=\kappa_{(V_{f(a)}, \pi_{f(a)})}(a,\dots, a).$
\end{enumerate}
Less evident is the fact that 
$$
(5)~~~~~~~~~~~~~~~~~~\#\gamma_{3r,3s,3t}\pi^{-1}=\#\gamma_{r,s,t}\pi_{f(a)}^{-1}+r+s+t.$$
In order to verify this property, notice that $(3i-2)$ are singletons in the permutation $\gamma_{3r,3s,3t}\pi^{-1}$, since  $\gamma_{3r,3s,3t}\pi^{-1}(O)= O$ and 
 $\gamma_{3r,3s,3t}\pi^{-1}$ separates the points of $O$. Hence, we just need to study the cycles  formed by the remaining elements, i.e, $\{2,3,\dots,3r+3s+3t-1,3r+3s+3t\}$. In connection with this task, we also have that $\gamma_{3r,3s,3t}\pi^{-1}(3i)=\gamma^{2}_{3r,3s,3t}(3i),$ since $\pi(\gamma_{3r,3s,3t}(3i))=3i$  (see Figure \ref{fe1}). That means, we know where goes each $3i$ under $\gamma_{3r,3s,3t}\pi^{-1}$ thus, as we did before, we can express the permutation just codifying the behavior of the remaining positions which are $\{2,5,\dots,3r+3s+3t-1\}$. In this direction, we can check that $$\gamma_{r,s,t}\pi^{-1}_{f(a)}(i)=j ~\text{iff}~ \gamma_{3r,3s,3t}\pi^{-1}(3i-1)=\gamma^{-2}_{3r,3s,3t}(3j-1).$$ 
 Noting that, this last property tells us that the number of $\gamma_{r,s,t}\pi^{-1}_{f(a)}$-cycles matches  with the number of remaining $\gamma_{3r,3s,3t}\pi^{-1}$-cycles; with this, the fifth  property is verified. As a consequence of all the properties we have that $(V_{f(a)},\pi_{f(a)}) \in \PS_{NC}(r,s,t)$, i.e, $V_{f(a)}
\vee \gamma_{r,s,t} = 1_{r+s+t}$ and
$
|(V_{f(a)}, \pi_{f(a)})| + |(0_{\pi_{f(a)}^{-1}\gamma_{r,s,t}}, \pi_{f(a)}^{-1} \gamma_{r,s,t})| =
|(1_{r+s+t}, \gamma_{r,s,t})|$. Thus, under the hypothesis,  we have a bijective mapping from  $(\V,\pi) \in \PS_{NC}(3r,3s,3t)$ such that $\pi^{-1}\gamma_{3r,3s,3t}$ separates the points of $N$ to $(\U,\sigma) \in \PS_{NC}(r,s,t)$. Therefore, as conclusion

$$\kappa_{r,s,t}(cac^*,cac^*,\dots,cac^*,cac^*)=\sum_{(\V,\pi)\in \PS_{NC}(r,s,t)} \kappa_{(\V,\pi)}(a,\dots,a)=\varphi_{3}(a^r,a^s,a^t).$$

\end{example}
At this point, it is important highlight that this generalizes Corollary 3.7  of the foundational paper of Mingo and Speicher \cite{MS}.

\begin{remark}

One important case in the previous example is when $a=1$. In this case, one finds that the third order cumulants of $cc^*$ are zero, namely,   $\kappa_{r,s,t}(cc^*,cc^*,\dots,cc^*,cc^*)=0$. Recall that
from Example \ref{feo} we know that the third order cumulants of $s^2$ are also $0$, from this perspective, may not be so surprising that the same holds for $cc^*$, since the calculation is similar, with an extra alternating condition. However, while it is true that first order cumulants of $s^2$ and $cc^*$ are the same, in contrast, as discussed in \cite{AM}, the second-order cumulants do not coincide. This observation is consistent with the fact that higher-order cumulants in non-commutative probability spaces are more delicate.
\end{remark}

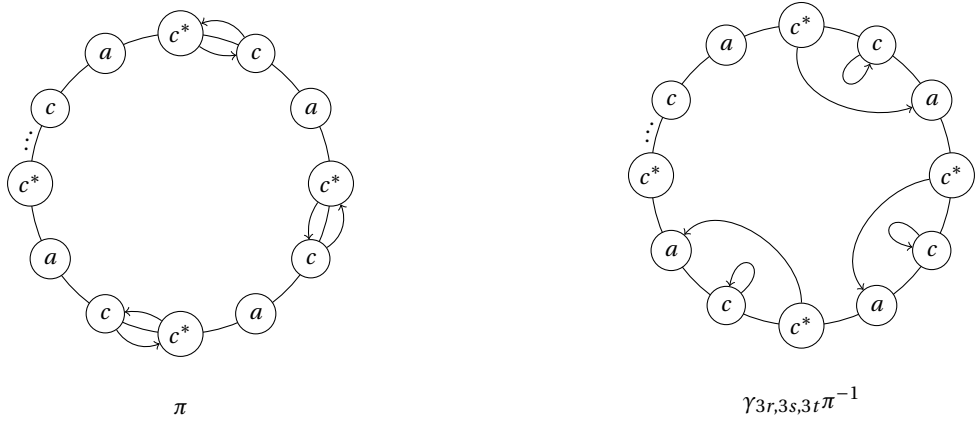
\begin{figure}
\centering
\begin{subfigure}{0.5\textwidth}
\centering
\begin{tikzpicture}
\begin{scriptsize}
  \begin{scope}[shift={(0,0)}]
    \draw[] (0,0) circle (2);

    \foreach \i in {1,4,7,10} {
      \node[fill=white, draw, circle, inner sep=3pt] (n\i) at ({360/12 * (6-\i )}:2) {$c$};
    }
  
    \foreach \i in {2,5,8,11} {
      \node[fill=white, draw, circle, inner sep=3pt] (n\i) at ({360/12 * (6-\i )}:2) {$a$};
    }

    \foreach \i in {3,6,9,12} {
      \node[fill=white, draw, circle, inner sep=2pt] (n\i) at ({360/12 * (6-\i )}:2) {$c^*$};
    }


    \node at (-2.05,0.43) {$\cdot$};
    \node at (-1.99-0.05,0.55) {$\cdot$};
    \node at (-1.95-0.04,0.65) {$\cdot$};
      \node at (0,-3) {$\pi$};

    \path[->] (n3) edge[bend left=-20] (n4);
    \path[->] (n4) edge[bend left=-40] (n3);
    \path[->] (n6) edge[bend left=-20] (n7);
    \path[->] (n7) edge[bend left=-40] (n6);
    \path[->] (n9) edge[bend left=-20] (n10);
    \path[->] (n10) edge[bend left=-40] (n9);
  \end{scope}
\end{scriptsize}
\end{tikzpicture}

\end{subfigure}%
\begin{subfigure}{0.5\textwidth}
\centering
\begin{tikzpicture}
\begin{scriptsize}
  \begin{scope}[shift={(0,0)}]
    \draw[] (0,0) circle (2);
   \foreach \i in {1,4,7,10} {
      \node[fill=white, draw, circle, inner sep=3pt] (n\i) at ({360/12 * (6-\i )}:2) {$c$};
    }
  
    \foreach \i in {2,5,8,11} {
      \node[fill=white, draw, circle, inner sep=3pt] (n\i) at ({360/12 * (6-\i )}:2) {$a$};
    }

    \foreach \i in {3,6,9,12} {
      \node[fill=white, draw, circle, inner sep=2pt] (n\i) at ({360/12 * (6-\i )}:2) {$c^*$};
    }

    \node at (-2.05,0.43) {$\cdot$};
    \node at (-1.99-0.05,0.55) {$\cdot$};
    \node at (-1.95-0.04,0.65) {$\cdot$};
    \node at (0,-3) {$\gamma_{3r,3s,3t}\pi^{-1}$};

    \path[->]  (n3) edge [out=260,in=200,looseness=1]node[above]{}(n5);
    \path[->]  (n4) edge [out=210,in=250,looseness=8] node[above]{}(n4);
    
      \path[->]  (n6) edge [out=190,in=130,looseness=1] node[above]{}(n8);
    \path[->]  (n7) edge [out=130,in=170,looseness=8] node[above]{}(n7);
    
      \path[->]  (n9) edge [out=90,in=50,looseness=1] node[above]{}(n11);
    \path[->]  (n10) edge [out=40,in=80,looseness=8] node[above]{}(n10);

  \end{scope}
\end{scriptsize}
\end{tikzpicture}

\end{subfigure}

\caption{Representation of the permutations used in Example \ref{e1}.}
\label{fe1}
\end{figure}

\subsection{Product with an R-diagonal} \label{Rdiagonal1}

The objective of this current section is to demonstrate that, similar to the situation involving first and second order cases, the preservation of $R$-diagonality (see definition below) persists when multiplying by a free element.
\begin{definition}[$R$-diagonal]
 An element $a$ in  $\left(\mathcal{A}, \varphi, \varphi_2,\varphi_3\right)$, a  $*$-non-commutative probability space, is called third order $R$-diagonal if the only first, second, and third-order cumulants that can possibly be non-zero take the form 
$$
\kappa_{2 r}\left(a, a^*, \ldots, a, a^*\right)=\kappa_{2 r}\left(a^*, a, \ldots, a^*, a\right),
$$
$$
\kappa_{2 r, 2 s}\left(a, a^*, \ldots, a, a^*\right)=\kappa_{2 r, 2 s}\left(a^*, a, \ldots, a^*, a\right),
$$
$$
\kappa_{2 r, 2 s, 2t}\left(a, a^*, \ldots, a, a^*\right)=\kappa_{2 r, 2s ,2t}\left(a^*, a, \ldots, a^*, a\right),
$$
respectively.
\end{definition}

Let us point out that a third order circular element $c$ as defined above provides an example of a third order $R$-diagonal element.

\begin{theorem}\label{t2}
    Let $\left\{a, a^*\right\}$ and $\left\{b, b^*\right\}$ be third order free and suppose that $a$ is third order $R$-diagonal. Then $ab$ is third order $R$-diagonal.
\end{theorem} Let $\epsilon_i\in \{1,-1\}, (ab)^{(1)}=ab$ and  $(ab)^{(-1)}=(ab)^*$, we have to prove that 
 \begin{enumerate}
     \item   $\kappa_{r}((ab)^{(\epsilon_1)},\dots,(ab)^{(\epsilon_r)})= 0$ unless $\epsilon_i = -\epsilon_{\gamma_{r}(i)}$.
\item $\kappa_{r,s}((ab)^{(\epsilon_1)},\dots,(ab)^{(\epsilon_{r+s})})= 0$ unless $\epsilon_i = -\epsilon_{\gamma_{r,s}(i)}$.

\item $\kappa_{r,s,t}((ab)^{(\epsilon_1)},\dots,(ab)^{(\epsilon_{r+s+t})})= 0$ unless $\epsilon_i = -\epsilon_{\gamma_{r,s,t}(i)}$.
 \end{enumerate}
 Since, (1) and (2) are proved in \cite[Lecure 15]{NS} and \cite{AM}, respectively.  We only have to demonstrate (3).

\begin{proof}
   
In accordance with Theorem \ref{Main theorem cumulants with products as entries}, the formula for cumulants involving products as parameters, it follows that
$$\kappa_{r,s,t}((ab)^{(\epsilon_1)},\dots,(ab)^{(\epsilon_{r+s+t})})=\sum_{(\V,\pi)\in \PS_{NC}(2r,2s,2t)}\kappa_{(\V,\pi)}(x_1,x_2,\dots,x_{2(r+s+t)-1},x_{2(r+s+t)}),$$
where the summation is over those $(\V,\pi)\in \PS_{NC}(2r,2s,2t)$ such that $\gamma_{2r,2s,2t}\pi^{-1}$ separates the points of $O= \{1,3,\dots,2(r+s+t)-1\}$ and $$x_{2 i-1}=\begin{cases}a & \epsilon_i=1 \\ b^* & \epsilon_i=-1\end{cases} , x_{2 i}= \begin{cases}a^* & \epsilon_i=-1 \\ b & \epsilon_i=1\end{cases}.$$
Now, since $a$ is third order $R$-diagonal  and $\left\{a, a^*\right\}$ and $\left\{b, b^*\right\}$ are third order free, we have that only those $(\V,\pi) \in \PS_{NC}(2r, 2s, 2t)$ that satisfy the following conditions (possibly) contribute to the sum: $\gamma_{2r, 2s, 2t} \pi^{-1}$ separates the points of $O$, all cycles of $\pi$ either visit only positions corresponding to $a$ or 
$a^*$  (referred to as $a$-cycles) or positions corresponding to $b$ or $b^*$ (referred to as $b$-cycles), $a$-cycles must alternate between $a$ and $a^*$ positions, and $V$ does not join $a$-cycles with $b$-cycles.

Given $j \in \{1,2,\dots,r+s+t
\}$ such that $\epsilon_{j}=1$, due to the  alternating nature of the $a$-cycles, it follows that $ \gamma_{2r,2s,2t}\pi^{-1}(2j-1) \in O$. Besides, $\gamma_{2r,2s,2t}\pi^{-1}$  separates the points of $O$, therefore $\gamma_{2r,2s,2t}\pi^{-1}(2j-1)=2j-1.$ The above relation, along with the  alternating nature of the $a$-cycles, reveals that $\pi^{-1}(2j-1)=\gamma_{2r,2s,2t}^{-1}(2j-1)$ when $\epsilon_{j}=1$,  or equivalently, that $\pi(2j)=\gamma_{2r,2s,2t}(2j)$ when $\epsilon_{j}=-1$ since
$$\pi(2j)=2l-1=\gamma_{2r,2s,2t}\pi^{-1}(2l-1)=\gamma_{2r,2s,2t}(2j).$$ 
Now, let's analyze what the two preceding equalities say in terms of the $\{\epsilon_i\}_i$. Due to the first equality we have that, if $\epsilon_{j}=1$ then $\pi^{-1}(2j-1)=\gamma_{2r,2s,2t}^{-1}(2j-1)$. Thus, $$\epsilon_{\frac{\gamma_{2r,2s,2t}^{-1}(2j-1)}{2}=\gamma_{r,s,t}^{-1}(j)}=-1.$$ Proving that, if $\epsilon_{j}=1$ then $\epsilon_{\gamma_{r,s,t}^{-1}(j)}=-1$. On the other hand,  due to the second equality we have that, if $\epsilon_{j}=-1$ then $\pi(2j)=\gamma_{2r,2s,2t}(2j)$. Thus, $$\epsilon_{\frac{\gamma_{2r,2s,2t}(2j)+1}{2}=\gamma_{r,s,t}(j)}=1.$$ Proving that, if $\epsilon_{j}=-1$ then $\epsilon_{\gamma_{r,s,t}(j)}=1$. Finally, taking into account that $\epsilon_{\gamma_{r,s,t}^{-1}(j)}=-1$ when $\epsilon_{j}=1$ is equivalent to $\epsilon_{j}=-1$ when $\epsilon_{\gamma_{r,s,t}(j)}=1$, the proof is completed.

\end{proof}
Based on the result provided above, the following example has been included to demonstrate its practical application.

\begin{example}\label{e2}

Given $c$,  a third order circular operator, if we consider $\{c,c^*\}$  and $\{a,a^*\}$ third order
free,  by  Theorem \ref{t2}, the element $ca$ is a third order $R$-diagonal operator. Thus, we are interested in the third order cumulants of $ca$. Given $r,s$ and $t$ even, by a process similar to Example \ref{e1}, we have that

$$\kappa_{r,s,t}(a^*c^*,ca,\dots,a^*c^*,ca)=\sum_{(\V,\pi)\in \PS_{NC}(2r,2s,2t)}\kappa_{(\V,\pi)}(a^*,c^*,c,a,\dots,a^*,c^*,c,a),$$
where the summation is over those $\PS_{NC}(2r,2s,2t)$ such that  $V=\{\{2,3\},\dots,\{2r+2s+2t-2,2r+2s+2t-1\}\}\cup V_a$, $\pi=(2,3)\dots (2r+2s+2t-2,2r+2s+2t-1)\pi_{a}$  and   $\pi^{-1}\gamma_{2r,2s,2t}$ separates the points of $\{2,4,\dots,2r+2s+2t\}$ (see Figure \ref{fe2}).  Note that, as in Example \ref{e1}, the separability condition at even positions of the form $4i-2$ is considered in the description of $(\V, \pi)$ when specifying the form of each $c$-cycle. However, unlike Example \ref{e1}, here we must also preserve the separability condition at even positions of the form $4i$, as this property has not yet been included in the description of $(\V, \pi)$. At this point, we set to work with  $\pi_{f(a)}$, which is $\pi_a$ considered under the  bijection $f(x)=(x+1)/2$ if $x=4i-3$ and $f(x)=x/2$ if $x=4i$ from $\{1,4,5,\dots,2r+2s+2t-4,2r+2s+2t-3,2r+2s+2t\}$ to $\{1,2,3,\dots,r+s+t-2,r+s+t-1,r+s+t\}.$ Similarly, in this example we have a version for the first four properties shown in Example \ref{e1} and, in this case, the analogous for the fifth property is the following one: $\#\gamma_{2r,2s,2t}\pi^{-1}=\#\gamma_{r,s,t}\pi^{-1}_{f(a)}+(r+s+t)/2$. In the same way as in the Example \ref{e1}, we can prove that $4i-1$ are singletons in $\gamma_{2r,2s,2t}\pi^{-1}$ and   $\gamma_{2r,2s,2t}\pi^{-1}(4i-2)=\gamma^{2}_{2r,2s,2t}(4i-2)$ (see Figure \ref{fe2}). So, as before, by codifying $\gamma_{2r,2s,2t}\pi^{-1}$  in terms of the elements in $\{1,4,5,\dots,2r+2s+2t-4,2r+2s+2t-3,2r+2s+2t\}$, we obtain that

$$\gamma_{r,s,t}\pi^{-1}_{f(a)}(i)=j ~\text{iff}~ \gamma_{2r,2s,2t}\pi^{-1}(f^{-1}(i))=x_j,$$ 
where $x_j=\gamma^{-2}_{2r,2s,2t}(2j)$ if $j$ is even and  $x_j=2j-1$ if $j$ is odd. The above relation allows us to count the remaining cycles of 
$\gamma_{2r,2s,2t}\pi^{-1}$ and also shows that 
$\gamma_{r,s,t}\pi^{-1}_{f(a)}$ separates 
$\{1, 3, \dots, r+s+t-1\}$, since
$\gamma_{2r,2s,2t}\pi^{-1}$ separates 
$\{1, 5, \dots, 2r+2s+2t-3\}$. Therefore, we have that: $(V_{f(a)},\pi_{f(a)}) \in \PS_{NC}(r,s,t)$, $\gamma_{r,s,t}\pi^{-1}_{f(a)}$ separates the points of $\{1,3,\dots,r+s+t-1\}$ and $\kappa_{(\V,\pi)}(a^*,c^*,\dots,c, a)=\kappa_{(V_{f(a)}, \pi_{f(a)})}(a^*,a,\dots,a^*, a)$.  Thus, under the hypothesis, we can set a bijective mapping from  $(\V,\pi) \in \PS_{NC}(2r,2s,2t)$ such that $\gamma_{2r,2s,2t}\pi^{-1}$ separates the points of $\{1,3,5,\dots,2r+2s+2t-1\}$ to $(\U,\sigma) \in \PS_{NC}(r,s,t)$ such that $\gamma_{r,s,t}\sigma^{-1}$ separates the points of $\{1,3,\dots,r+s+t-1\}$. As a consequence,

$$\kappa_{r,s,t}(a^*c^*,ca,\dots,a^*c^*,ca)=\sum_{(\V,\pi)\in \PS_{NC}(r,s,t)}\kappa_{(\V,\pi)}(a^*,a,\dots,a^*, a),$$
where the summation is over those $\PS_{NC}(r,s,t)$ such that $\gamma_{r,s,t}\pi^{-1}$ separates the points of $\{1,3,\dots,r+s+t-1\}$. Therefore, applying again Theorem \ref{Main theorem cumulants with products as entries}, we finally have that

$$\kappa_{r,s,t}(a^*c^*,ca,\dots,a^*c^*,ca)=\kappa_{r/2,s/2,t/2}(a^*a,\dots,a^*a).$$
\end{example}

\begin{figure}
\centering
\begin{subfigure}{0.5\textwidth}
\centering
\begin{tikzpicture}
\begin{scriptsize}
  \begin{scope}[shift={(0,0)}]
    \draw[] (0,0) circle (2);

    \foreach \i in {1,5,9} {
      \node[fill=white, draw, circle, inner sep=1pt] (n\i) at ({360/12 * (6-\i )}:2) {$a^*$};
    }
  
    \foreach \i in {2,6,10} {
      \node[fill=white, draw, circle, inner sep=1pt] (n\i) at ({360/12 * (6-\i )}:2) {$c^*$};
    }

    \foreach \i in {3,7,11} {
      \node[fill=white, draw, circle, inner sep=2.5pt] (n\i) at ({360/12 * (6-\i )}:2) {$c$};
    }
    
    \foreach \i in {4,8,12} {
      \node[fill=white, draw, circle, inner sep=2.5pt] (n\i) at ({360/12 * (6-\i )}:2) {$a$};
    }


    \node at (-2.05,0.43) {$\cdot$};
    \node at (-1.99-0.05,0.55) {$\cdot$};
    \node at (-1.95-0.04,0.65) {$\cdot$};
      \node at (0,-3) {$\pi$};

    \path[->] (n2) edge[bend left=-20] (n3);
    \path[->] (n3) edge[bend left=-40] (n2);
    \path[->] (n6) edge[bend left=-20] (n7);
    \path[->] (n7) edge[bend left=-40] (n6);
    \path[->] (n10) edge[bend left=-20] (n11);
    \path[->] (n11) edge[bend left=-40] (n10);
  \end{scope}
\end{scriptsize}
\end{tikzpicture}

\end{subfigure}%
\begin{subfigure}{0.5\textwidth}
\centering
\begin{tikzpicture}
\begin{scriptsize}
  \begin{scope}[shift={(0,0)}]
    \draw[] (0,0) circle (2);

    \foreach \i in {1,5,9} {
      \node[fill=white, draw, circle, inner sep=1pt] (n\i) at ({360/12 * (6-\i )}:2) {$a^*$};
    }
  
    \foreach \i in {2,6,10} {
      \node[fill=white, draw, circle, inner sep=1pt] (n\i) at ({360/12 * (6-\i )}:2) {$c^*$};
    }

    \foreach \i in {3,7,11} {
      \node[fill=white, draw, circle, inner sep=2.5pt] (n\i) at ({360/12 * (6-\i )}:2) {$c$};
    }
    
    \foreach \i in {4,8,12} {
      \node[fill=white, draw, circle, inner sep=2.5pt] (n\i) at ({360/12 * (6-\i )}:2) {$a$};
    }


    \node at (-2.05,0.43) {$\cdot$};
    \node at (-1.99-0.05,0.55) {$\cdot$};
    \node at (-1.95-0.04,0.65) {$\cdot$};
    \node at (0,-3) {$\gamma_{2r,2s,2t}\pi^{-1}$};

    \path[->]  (n3) edge [out=250,in=290,looseness=8] node[above]{}(n3);
    \path[->]  (n7) edge [out=120,in=160,looseness=8] node[above]{}(n7);
    \path[->]  (n11) edge [out=20,in=60,looseness=8] node[above]{}(n11);
    \path[->]  (n2) edge [out=270,in=270,looseness=1] node[above]{}(n4);
    \path[->]  (n6) edge [out=150,in=150,looseness=1] node[above]{}(n8);
    \path[->]  (n10) edge [out=30,in=30,looseness=1] node[above]{}(n12);
  \end{scope}
\end{scriptsize}
\end{tikzpicture}

\end{subfigure}

\caption{Representation of the permutations used in Example \ref{e2}.}
\label{fe2}
\end{figure}
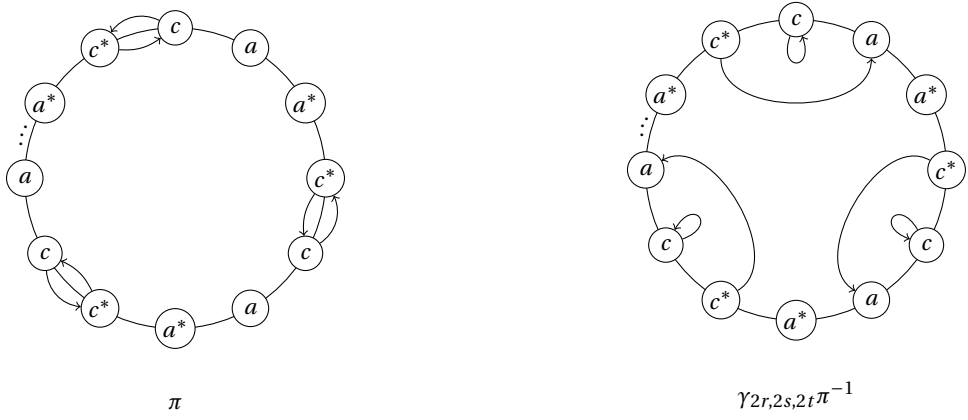

 \subsection{ Cumulants of $aa^*$
for an $R$-diagonal Operator.} \label{Rdiagonal2}
Building on the ideas used in the proof of Theorem \ref{t2} and in the previous examples, this section focuses on computing the final expression obtained in Example \ref{e2} when $a$ exhibits 
$R$-diagonal properties. In plain words, the main result on this section  extends for the third order the formula for the cumulants of $aa^*$, when $a$ is $R$-diagonal which was shown in \cite{NS} and \cite{AM} for the first and second  order, respectively.

\begin{notation}
\begin{enumerate}
    \renewcommand{\labelenumi}{(\arabic{enumi})}
   \item Let $a$ be third order $R$-diagonal. Define $$\beta_r^{(a)}:=\kappa_{2 r}\left(a, a^*, \ldots, a,a^*\right),$$
    $$\beta_{r,s}^{(a)}:=
    \kappa_{2r,2s}\left(a, a^*, \ldots, a,a^*\right),$$
    and  
    $$\beta_{r,s,t}^{(a)}:=\kappa_{2r,2s,2t}\left(a, a^*, \ldots, a,a^*\right).$$
    The sequences  $\left(\beta_r^{(a)}\right)_{r \geq 1}$, $\left(\beta_{r, s}^{(a)}\right)_{r, s \geq 1}$ and $\left(\beta_{r, s,t}^{(a)}\right)_{r, s,t \geq 1}$ are called the (first, second and third
order) determining sequences of $a$. Finally, given a partitioned permutation \((V, \pi)\), the quantity \(\beta_{(V,\pi)}^{(a)}\) is defined as the higher-order cumulants.

    \item We say that a permutation $\pi$ is
parity reversing if for all $k$, $\pi(k)$ and $k$ have the opposite parity. We
denote the elements of $S_{n}$ that are parity reversing by $S^{-}_{n}$ and similarly we denote the elements of  $\PS_{NC}(r_1,\dots,r_m)$ such that $\pi$ is parity reversing by $\PS^{-}_{NC}(r_1,\dots,r_m)$.

\end{enumerate}

\end{notation}

Let us state the main result of this section, the proof
can be seen at the end of the section:    Let $a$ be a third order $R$-diagonal operator with determining sequences $\left(\beta_r^{(a)}\right)_{r \geq 1}$, $\left(\beta_{r, s}^{(a)}\right)_{r, s \geq 1}$ and $\left(\beta_{r, s,t}^{(a)}\right)_{r, s,t \geq 1}$
then 

   \begin{equation*} 
   \kappa_{r, s,t}\left(a a^*, \ldots, a a^*\right)=\sum_{(\mathcal{V}, \pi) \in\PS_{NC}(r,s,t)}\beta^{(a)}_{(\mathcal{V}, \pi) }.   
\end{equation*}

Before proceeding with the proof, we establish some preliminary results.

 \begin{lemma} \label{l1}  Given $\pi \in S^{-}_{2r_1+\dots+2r_m}$, we have that
 $\gamma_{2r_1,\dots,2r_m}\pi^{-1}$ separates the points of
$O=\{1,3,\dots\\,\sum_{i=1}^{m}2r_i-1\}$ if and only if for all $k$, $\pi(2k)=\gamma_{2r_1,\dots,2r_m}(2k)$.

\end{lemma}
\begin{proof}

We know that,  $\gamma_{2r_1,\dots,2r_m}\pi^{-1}$ separates the points of
$O=\{1,3,\dots,\sum_{i=1}^{m}2r_i-1\}$ if and only if $\pi^{-1}\gamma_{2r_1,\dots,2r_m}$ separates the points of
$N=\{2,4,\dots,\sum_{i=1}^{m}2r_i\}$. Now, since $\pi \in S^{-}_{2r_1+\dots+2r_m}$, it follows that $$\pi^{-1}\gamma_{2r_1,\dots,2r_m}(N)=N.$$
This completes the proof of the first implication. The second implication follows directly.

\end{proof}

Based on Lemma \ref{l1}, the behavior of even elements is understood, our attention can be directed solely towards odd elements. The subsequent definition encapsulates this idea.

\begin{definition}\label{d1}
    Suppose $\pi \in S_{2r_1+\dots+2r_m}^-$ is such that  $\gamma_{2r_1,\dots,2r_m}\pi^{-1}$ separates the points of
$O=\{1,3,\dots,\sum_{i=1}^{m}2r_i-1\}$. Let $ \Check{\pi} \in S_{r_1+\dots+r_m}$ defined by $ \Check{\pi}(k)=\pi(\pi(2k))/2.$ We call $\Check{\pi}$ the half of $\pi$.

\end{definition}

In light of the definition provided above, the coming lemmas are formulated under the assumption of $\pi \in S_{2r_1+\dots+2r_m}^-$ is such that  $\gamma_{2r_1,\dots,2r_m}\pi^{-1}$ separates the points of
$O=\{1,3,\dots,\sum_{i=1}^{m}2r_i-1\}$, with the aim of comprehending how $\pi$ and $ \Check{\pi}$ are related. By the way, $\check{\gamma}_{2r_1,\dots,2r_m}=\gamma_{r_1,\dots,r_m}.$ 

\begin{lemma}\label{l3}
   It is validated that, 
    \begin{enumerate}
        \item $\#\pi=\#\check{\pi}.$
      
        \item $\pi(\pi(2k))=2l$ for $2k,2l$ in different $\gamma_{2r_1,\dots,2r_m}$-cycles iff~~ $ \Check{\pi}(k)=l$  for $k,l$ in different $\gamma_{r_1,\dots,r_m}$-cycles.
        \item  $\gamma_{2r_1,\dots,2r_m}\pi^{-1}(2k)=2l$ iff~~   $\gamma_{r_1,\dots,r_m} \Check{\pi}^{-1}(k)=l.$
        \item $\#\gamma_{r_1,\dots,r_m}\check{\pi}^{-1}+\sum_{i=1}^{m}r_i=\#\gamma_{2r_1,\dots,2r_m}\pi^{-1}.$
        
    \end{enumerate}
   
\end{lemma}

\begin{proof}

The first two properties are a direct consequence of the definition; therefore, our current objective is to establish the last two. Let's start with (3), for all $1 \leq k,l \leq r_1+\dots+r_m$, by Lemma \ref{l1},
    $\pi(2\gamma^{-1}_{r_1,\dots,r_m}(l))=\gamma_{2r_1,\dots,2r_m}(2\gamma^{-1}_{r_1,\dots,r_m}(l))$. In addition, $\gamma_{2r_1,\dots,2r_m}(2\gamma^{-1}_{r_1,\dots,r_m}(l))=\gamma_{2r_1,\dots,2r_m}^{-1}(2l),$ since $\check{\gamma}_{2r_1,\dots,2r_m}=\gamma_{r_1,\dots,r_m}.$   Consequently,\\
    
\begin{center}
    $\gamma_{r_1,\dots,r_m} \check{\pi}^{-1}(k) =l$ iff $\check{\pi}\left(\gamma_{r_1,\dots,r_m}^{-1}(l)\right)=k$ iff $\pi^2\left(2 \gamma_{r_1,\dots,r_m}^{-1}(l)\right) =2 k$
\end{center}

\begin{center}
   iff  $\pi\left(\gamma_{2r_1,\dots,2r_m}^{-1}(2 l)\right)=2 k$ iff  $\gamma_{2r_1,\dots,2r_m} \pi^{-1}(2 k)=2 l.$ 
\end{center}
Now, let's revisit the last one. Since, $\gamma_{2r_1,\dots,2r_m}\pi^{-1}(2k+1)=2k+1$ holds for all $k$, we just need to determine the number of cycles formed by even elements in $\gamma_{2r_1,\dots,2r_m}\pi^{-1}$. Due to (3) the $\gamma_{2r_1,\dots,2r_m}\pi^{-1}$- cycles formed by even elements are the cycles of $\gamma_{r_1,\dots,r_m} \Check{\pi}^{-1}.$ Therefore, we confirm the statement.

\end{proof}
 
At this point is crucial to point out that Lemma \ref{l3} (2) indicates that: Through cycles are preserved under the half operation; recalling that conforming to Lemma \ref{l1}  a through cycle in $\pi$ contains at least four elements.

\begin{lemma}\label{l4}
The map $\pi \mapsto \check{\pi}$ is a bijection from $$
\left\{ \pi \in S_{2r_1+\dots+2r_m}^{-} \biggm| \gamma_{2r_1,\dots,2r_m} \pi^{-1}  
\text{ separates the points of } O = \{1,3,\dots,\sum_{i=1}^{m} (2r_i) - 1\} \right\}$$ to $S_{r_1+\dots+r_m}$.    
\end{lemma}

\begin{proof}
   
Let's consider $\pi_1$ and $\pi_2$ belonging to $\{ \pi \in S_{2r_1+\dots+2r_m}^{-} \mid$ $\gamma_{2 r_1,\dots,2r_m} \pi^{-1}$ separates the point of $O\}$ and notice that as a straightforward consequence of Lemma \ref{l1}, $\pi_j(2 k)=\gamma_{2 r_1,\dots,2r_m}(2 k)$ for $j=1,2$. Therefore, $\pi_1$ and $\pi_2$ coincide on the even numbers. Then, if you assume that $\check{\pi}_1=\check{\pi}_2$ in order to prove injectivity our task is to demonstrate that $\pi_1$ and $\pi_2$ exhibit agreement on the odd numbers. Notably, we observe that $\pi_1\left(\gamma_{2 r_1,\dots,2r_m}(2 k)\right)=\pi_1^2(2 k)=\pi_2^2(2 k)=\pi_2\left(\gamma_{2 r_1,\dots,2r_m}(2 k)\right)$. Thus, the first statement is proven. Given $\sigma \in S_{r_1+\dots+r_m}$, let $\pi$ be defined in $S_{2r_1+\dots+2r_m}$ such that $\pi(2k)=\gamma_{2r_1,\dots,2r_m}(2k)$ and $\pi(\gamma_{2r_1,\dots,2r_m}(2k))=2\sigma(k)$. It is evident that $\pi$ is a permutation designed to reverse parity.  Thus, by Lemma \ref{l1}, we  can deduce from the construction that $\gamma_{2 r_1,  \dots,2r_m} \pi^{-1}$ separates the point of $O.$ Therefore, $\pi \in \{ \alpha \in S_{2r_1+\dots+2r_m}^{-} \mid$ $\gamma_{2 r_1,\dots,2r_m} \alpha^{-1}$ separates the point of $O\}$ and $\check{\pi}=\sigma.$ Hence, the proof is finished.

\end{proof}

\begin{definition}
    Given  $\pi \in S_{r_1+\dots+r_m} $, we denote the inverse of the half mapping by 
 $\hat{\pi}$, defined as $\hat{\pi}(2k)=\gamma_{2r_1,\dots,2r_m}(2k)$ and $\hat{\pi}(\gamma_{2r_1,\dots,2r_m}(2k))=2\pi(k).$  We call $\hat{\pi}$ the double of $\pi$.
\end{definition}

Ultimately, the next lemma extends  Lemma \ref{l4} to encompass partitions permutations, which are in effect the combinatorial framework for dealing with higher-order cumulants.

\begin{lemma} \label{l5}
It is verified that $\{(V,\pi) \in \PS^{-}_{NC}(2r_1,\dots,2r_m)|\gamma_{2r_1,\dots,2r_m}\pi^{-1} \text{ separates the points of}\\~ O:=\{1,3,\dots,\sum_{i=1}^{m}2r_i-1\}\}\cong \PS_{NC}(r_1,\dots,r_m).$   
\end{lemma}

\begin{proof}
 The proof involves establishing  that the map $(V,\pi) \mapsto (\check{V},\check{\pi})$  is a bijection from $(V,\pi) \in \PS^{-}_{NC}(2r_1,\dots,2r_m)$ such that $\gamma_{2r_1,\dots,2r_m}\pi^{-1}$ separates the points of $O$ to $\PS_{NC}(r_1,\dots,r_m)$,  where $\check{V}$ is formed by joining the corresponding cycles of $\pi$ under the half mapping that $V$ joins. By construction we have that $(\check{V},\check{\pi})$ is a partition-permutation, now to verify that is an element of $\PS_{NC}(r_1,\dots,r_m)$ we have to check that
 $$\mathcal{\check{V}}
\vee \gamma_{r_1,\dots,r_m} = 1_{r_1+\dots+r_m},$$ and
$$
|(\mathcal{\check{V}}, \check{\pi})| + |(0_{\check{\pi}^{-1}\gamma_{r_1,\dots,r_m}}, \check{\pi}^{-1} \gamma_{r_1,\dots,r_m})| =
|(1_{r_1+\dots+r_m}, \gamma_{r_1,\dots,r_m})|.$$
The first property holds due to the construction of \( \check{V} \), Lemma \ref{l3} (2) and the hypothesis  
\[
\mathcal{V} \vee \gamma_{2r_1,\dots,2r_m} = 1_{2r_1+\dots+2r_m},
\] since  $(V,\pi) \in \PS_{NC}(2r_1,\dots,2r_m)$. The validity of the second property is confirmed using Lemma \ref{l3} (1) and (4), along with the observation that, by construction, $\#V = \#\check{V}$, and the fact that

$$
|(\mathcal{V}, \pi)| + |(0_{\pi^{-1}\gamma_{2r_1,\dots,2r_m}}, \pi^{-1} \gamma_{2r_1,\dots,2r_m})| =
|(1_{2r_1+\dots+2r_m}, \gamma_{2r_1,\dots,2r_m})|,$$
since  $(V,\pi) \in \PS_{NC}(2r_1,\dots,2r_m)$. Next, to demonstrate injectivity, we rely on the one already established in Lemma \ref{l4}. This allows us to assert that if $\check{\pi_1}=\check{\pi_2}$, then $\pi_1=\pi_2$. Thus, if $\pi_1=\pi_2$ and $\check{V_1}=\check{V_2}$ by construction we have that $V_1=V_2$. Confirming injectivity.

Moving forward, given $(U, \sigma) \in \PS_{NC}(r_1,\dots,r_m)$, we set $\pi=\hat{\sigma}$ and $V=\hat{U}$, where $\hat{U}$ is formed by joining the corresponding cycles of $\sigma$ under the double mapping that $U$ joins. Consequently, we can deduce in an analogous manner, as was done above, that $(V,\pi) \in \PS^{-}_{NC}(2r_1,\dots,2r_m)$  and $\gamma_{2r_1,\dots,2r_m}\pi^{-1}$ separates the points of $O$. Indeed, we find that by construction  $(\check{V},\check{\pi})=(U,\sigma)$, successfully completing the proof.

\end{proof}

Now we can prove the main theorem of this section  that we state again for the convenience of the reader.

\begin{theorem}\label{aa*}
   
     Let $a$ be a third order $R$-diagonal operator with determining sequences $\left(\beta_r^{(a)}\right)_{r \geq 1}$, $\left(\beta_{r, s}^{(a)}\right)_{r, s \geq 1}$ and $\left(\beta_{r, s,t}^{(a)}\right)_{r, s,t \geq 1}$
then

   \begin{equation*} 
   \kappa_{r, s,t}\left(a a^*, \ldots, a a^*\right)=\sum_{(\mathcal{V}, \pi) \in\PS_{NC}(r,s,t)}\beta^{(a)}_{(\mathcal{V}, \pi) }.   
\end{equation*}

\end{theorem}

\begin{proof}
 According to the formula for cumulants with products
as arguments, i.e, Theorem \ref{Main theorem cumulants with products as entries}
$$\kappa_{r, s,t}\left(a a^*, \ldots, a a^*\right)=\sum_{(\V,\pi)\in \PS_{NC}(2r,2s,2t)}\kappa_{(\V,\pi)}(a,a^*,\dots,a,a^*),$$
where the summation is over those $(\V,\pi)\in \PS_{NC}(2r,2s,2t)$ such that $\gamma_{2r,2s,2t}\pi^{-1}$ separates the points of $O \vcentcolon= \{1,3,\dots,2r+2s+2t-1\}$. Now, since $a$ is a third order $R$-diagonal operator all comes down to working with $\pi \in S^{-}_{2r+2s+2t}$. 
Hence, the sum mentioned earlier is simplified to the task of adding across those $(\V,\pi)\in \PS^{-}_{NC}(2r,2s,2t)$ such that   $\gamma_{2r,2s,2t}\pi^{-1}$ separates the points of $O$ (see Figure \ref{faa^*}). Thus, applying Lemma \ref{l5} and noting that, 
through the process of building $\beta^{(a)}_{(\check{\mathcal{V}}, \check{\pi})}=\kappa_{(\mathcal{V}, \pi)}\left(a, a^*, \ldots, a, a^*\right)$, the proof is done.

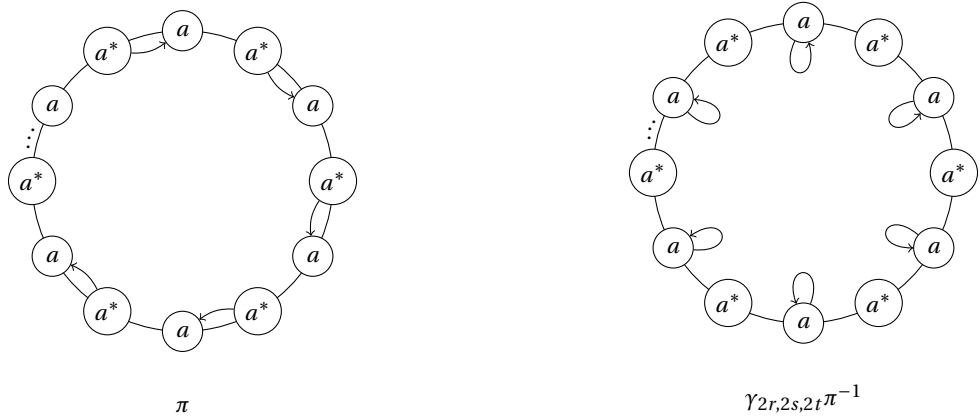
\begin{figure}
\centering
\begin{subfigure}{0.5\textwidth}
\centering
\begin{tikzpicture}
\begin{scriptsize}
  \begin{scope}[shift={(0,0)}]
    \draw[] (0,0) circle (2);

    \foreach \i in {1,3,5,7,9,11} {
      \node[fill=white, draw, circle, inner sep=3pt] (n\i) at ({360/12 * (6-\i )}:2) {$a$};
    }

    \foreach \i in {2,4,6,8,10,12} {
      \node[fill=white, draw, circle, inner sep=2pt] (n\i) at ({360/12 * (6-\i )}:2) {$a^*$};
    }


    \node at (-2.05,0.43) {$\cdot$};
    \node at (-1.99-0.05,0.55) {$\cdot$};
    \node at (-1.95-0.04,0.65) {$\cdot$};
      \node at (0,-3) {$\pi$};

    \path[->] (n2) edge[bend left=-20] (n3);

    \path[->] (n4) edge[bend left=-20] (n5);

    \path[->] (n6) edge[bend left=-20] (n7);
 
    \path[->] (n8) edge[bend left=-20] (n9);

     \path[->] (n10) edge[bend left=-20] (n11);
   
  \end{scope}
\end{scriptsize}
\end{tikzpicture}

\end{subfigure}%
\begin{subfigure}{0.5\textwidth}
\centering
\begin{tikzpicture}
\begin{scriptsize}
  \begin{scope}[shift={(0,0)}]
    \draw[] (0,0) circle (2);
   \foreach \i in {1,3,5,7,9,11} {
      \node[fill=white, draw, circle, inner sep=3pt] (n\i) at ({360/12 * (6-\i )}:2) {$a$};
    }

    \foreach \i in {2,4,6,8,10,12} {
      \node[fill=white, draw, circle, inner sep=2pt] (n\i) at ({360/12 * (6-\i )}:2) {$a^*$};
    }
 
    \node at (-2.05,0.43) {$\cdot$};
    \node at (-1.99-0.05,0.55) {$\cdot$};
    \node at (-1.95-0.04,0.65) {$\cdot$};
    \node at (0,-3) {$\gamma_{2r,2s,2t}\pi^{-1}$};

    \path[->]  (n1) edge [out=315,in=355,looseness=8]node[above]{}(n1);
    \path[->]  (n3) edge [out=245,in=285,looseness=8] node[above]{}(n3);
    
      \path[->]  (n5) edge [out=190,in=230,looseness=8] node[above]{}(n5);
    \path[->]  (n7) edge [out=140,in=180,looseness=8] node[above]{}(n7);
    
      \path[->]  (n9) edge [out=70,in=110,looseness=8] node[above]{}(n9);
    \path[->]  (n11) edge [out=355,in=35,looseness=8] node[above]{}(n11);

  \end{scope}
\end{scriptsize}
\end{tikzpicture}

\end{subfigure}

\caption{Representation of the permutations used in the proof of Theorem \ref{aa*}.}
\label{faa^*}

\end{figure}

\end{proof}

\subsection{Product of circular elements}

\label{productGinibres}
We conclude by using our results above to calculate the third order fluctuation cumulants  and fluctuation moments of $c_1c_2\cdots c_k$,  whenever the $c_i$ are third order free. Our reasoning is similar to that of the second-order case \cite{AM}, so we do not give all the details.

\begin{theorem}\label{thm:free_conjugation_circular}
Let $c_1, \dots, c_k$ be third order circular operators and suppose that
$\{c_1, c_1^*\}$, $\{c_2, c_2^* \}$, \dots, $\{c_k, c_k^*\}$
are third order free. Then
\[
\kappa^{(c_1 c_2 \cdots c_k  c_k^* \cdots\ab c_2^* c_1^*)}=  \zeta^{*k}.
\]
\end{theorem}

\begin{proof}

The proof follows by induction on $k$, by using the Example \ref{e1} above.
\end{proof}

\begin{remark}\label{remark:bousquet_melou}
By the formula of Bousquet-M\'elou and Schaeffer  \cite{BMS} (see also  \cite[\S5.17, p. 38]{CMSS}) we have
\begin{equation*}\label{eq:bousquet_melou_schaeffer}
\zeta^{*l}(1_{p+q+r}, \gamma_{p, q,r})
= l pqr \binom{l p - 1}{p} \binom{l q - 1}{q} \binom{l r - 1}{r}.
\end{equation*}So, we can calculate the fluctuation cumulants of $hh^*$, for $h = c_1 c_2 \cdots c_k $
\begin{equation*}\label{eq:h_cumulants}
\kappa_{p,q,r}(hh^*,\dots,hh^*) 
=kpqr \binom{kp-1}{p} \binom{kq-1}{q} \binom{kr-1}{r}.
\end{equation*} To obtain the third order fluctuation moments just apply the moment cumulant formula:

\begin{eqnarray*}\varphi_3((hh^*)^p, (hh^*)^q,(hh^*)^r) &=& \ 
\sum_{(V,\pi) \in \mathcal{PS}_{NC}(p, q, r)} \ 
\kappa^{(hh^*)}_{(V, \pi)} \\&=& \kappa^{(hh^*)}* \zeta(1_{p+q+r}, \gamma_{p,q,r}) =
\zeta^{*(k+1)}(1_{p+q+r}, \gamma_{p,q,r})\\
&=& (k+1) pqr \binom{(k+1)p-1}{p} \binom{(k+1)q-1}{q} \binom{(k+1)r-1}{r}. 
\end{eqnarray*}Finally, by Theorem \ref{t2}, $h=c_1 \cdots c_k$ is $R$-diagonal and the non-vanishing cumulants  can be derived from  Theorem \ref{aa*}, and are given  by
\begin{eqnarray*}
\kappa_{2p,2q,2r}(c_1 \cdots c_k, c_k^* \cdots c_1^*,  \dots, c_1 \cdots c_k, c_k^* \cdots c_1^*) &=& \beta^{(c_1 \cdots c_k)}_{p,q,r} 
=
\kappa^{(c_1 \cdots c_kc_k^* \cdots c_1^*)} * \mu(1_{p+q+r}, \gamma_{p,q,r})  \\ &=& \zeta^{*(k-1)}(1_{p+q+r}, \gamma_{p,q,r}) \\&=&
(k-1) pqr \binom{(k-1)p-1}{p} \binom{(k-1)q-1}{q} \binom{(k-1)r-1}{r}.
\end{eqnarray*}
\end{remark} To conclude, let us consider two independent complex Ginibre matrices, \( X_1 \) and \( X_2 \), and define \( W_2 = X_1 X_1^\dagger X_2 X_2^\dagger \), where \( A^\dagger \) denotes the complex conjugate of a matrix \( A \). Dartois and Forrester \cite{DF} calculated the second-order fluctuation moments of \( W_2 \) using topological recursion techniques. In \cite{AM}, the authors show how to derive the results of \cite{DF} by computing the limiting cumulants of \( \tilde{W}_2 =X_1^\dagger X_2  X_2^\dagger X_1 \). Following exactly the same arguments as in \cite{AM}, but applied to the third-order fluctuations of \( W_2 \), we obtain that the limiting fluctuation moments are given by  

\begin{align*}\varphi_3(w_2^p, w_2^q, w_2^r)
=
3pqr \binom{3p-1}{p} \binom{3q-1}{q} \binom{3r-1}{r}.
\end{align*}

\section*{Acknowledgements}

The authors thanks Jamie Mingo for useful discussions during the preparation of the paper. D. Muñoz  
was supported  by Hong Kong GRF 16304724 and NSFC 12222121. O. Arizmendi and S. Sigarreta were supported by CONACYT Grant CB-2017-2018-A1-S-9764. This work was initiated during  the event \emph{Seminario Interinstitucional de Matrices Aleatorias,
SIMA 2023} held at Universidad Autónoma de Sinaloa, during October 2023.

\appendix

\section{Combinatorial lemmas} \label{append}

In this appendix, we prove the crucial combinatorial lemmas that are needed in the proof of the main theorem.

\begin{notation}
Let $N_1 \vcentcolon = \{n_1,\dots,n_1+\cdots+n_r\}$, $N_2 \vcentcolon = \{n_1+\cdots+n_{r+1},\dots,n_1+\cdots+n_{r+s}\}$ and $N_3 \vcentcolon = \{n_1+\cdots+n_{r+s+1},\dots,n_1+\cdots+n_{r+s+t}\}$. For $\sigma \in S_{NC}(p,q,l)$ we let $\mathbf{P}_{\sigma}$ be the partition of $[3]$ where $i$ and $j$ are in the same block if $\sigma$ has a cycle that meets $N_i$ and $N_j$.
\end{notation}

\begin{lemma}\label{Proposition: Induction version 1}
$$\sum_{\pi\in S_{NC}(r,s,t)}\sum_{\substack{\sigma\in S_{NC}(p,q,l) \\ \sigma \leq \pi_{\vec{n}} \\ \sigma^{-1}\pi_{\vec{n}}\text{ separates }N}} \kappa_{\sigma}(\vec{a})=\sum_{\substack{\sigma\in S_{NC}(p,q,l) \\ \sigma^{-1}\gamma_{p,q,l}\text{ doesn't sep }N \\ \mathbf{P}_{\sigma}=\{1,2,3\}}}\kappa_{\sigma}(\vec{a}).$$
\end{lemma}

\begin{proof}
Let us first point out some facts what will be widely used in all following proofs. \textbf{Fact 1.} $\sigma^{-1}\gamma_{p,q,l}$ separates $N$ if and only if $\pi=\gamma_{r,s,t}$ provided $\sigma^{-1}\pi_{\vec{n}}$ separates $N$. Indeed, for $i\notin N$, $\pi_{\vec{n}}^{-1}\gamma_{p,q,l}(i)=i$, hence by \cite[Lemma 9]{MST},
$$\sigma^{-1}\gamma_{p,q,l}|_N=\sigma^{-1}\pi_{\vec{n}}|_N\pi_{\vec{n}}^{-1}\gamma_{p,q,l}|_N=\pi_{\vec{n}}^{-1}\gamma_{p,q,l}|_N.$$
Lemma \ref{Lemma: Some properties of pi_n} says
$$\pi_{\vec{n}}^{-1}\gamma_{p,q,l}(n_1+\cdots+n_i)=n_1+\cdots+n_{\pi^{-1}\gamma_{r,s,t}(i)}.$$
Therefore $\sigma^{-1}\gamma_{p,q,l}$ separates $N$ if and only if,
\begin{eqnarray*}
id &=& \sigma^{-1}\gamma_{p,q,l}|_N = \pi_{\vec{n}}^{-1}\gamma_{p,q,l}|_N,
\end{eqnarray*}
which happens if and only if for any $1\leq i\leq r+s+t$,
\begin{eqnarray*}
n_1+\cdots +n_i=\pi_{\vec{n}}^{-1}\gamma_{p,q,l}(n_1+\cdots +n_i)=n_1+\cdots + n_{\pi^{-1}\gamma_{r,s,t}(i)},
\end{eqnarray*}
which holds if and only if $\pi=\gamma_{r,s,t}$. \\
\textbf{Fact 2.} If $\sigma^{-1}\pi_{\vec{n}}$ separates $N$ then,
$$(b_1,\dots,b_w),$$
is a cycle of $\pi^{-1}\gamma_{r,s,t}$ if and only if,
$$(n_1+\cdots + n_{b_1},\dots,n_1+\cdots n_{b_w}),$$
is a cycle of $\sigma^{-1}\gamma_{p,q,l}|_N$. Indeed, as pointed out in Fact 1,
$$\sigma^{-1}\gamma_{p,q,l}|_N(n_1+\cdots +n_i)=\pi_{\vec{n}}^{-1}\gamma_{p,q,l}|_N(n_1+\cdots +n_i)=\pi_{\vec{n}}^{-1}\gamma_{p,q,l}(n_1+\cdots +n_i)=n_1+\cdots +n_{\pi^{-1}\gamma_{r,s,t}(i)},$$
which proves the desired. \\
\textbf{Fact 3.} If $\delta \in S_{r+s+t}$ is given by $\delta(i)=j$ whenever $\sigma^{-1}\gamma_{p,q,l}|_N(n_1+\cdots+n_i)=n_1+\cdots+n_j$, or equivalently $\delta = \psi^{-1}\sigma^{-1}\gamma_{p,q,l}|_N\psi$, with $\psi$ being defined as in Lemma \ref{Lemma: Some properties of pi_n}, and $\pi=\gamma_{r,s,t}\delta^{-1}$ then $\sigma^{-1}\pi_{\vec{n}}$ separates $N$. Indeed,
\begin{eqnarray}\label{aux5}
\sigma^{-1}\gamma_{p,q,l}|_N\psi=\psi\pi^{-1}\gamma_{r,s,t}=\pi_{\vec{n}}^{-1}\gamma_{p,q,l}\psi,
\end{eqnarray}
where last equality follows from Lemma \ref{Lemma: Some properties of pi_n}. This proves $\sigma^{-1}\gamma_{p,q,l}|_N=\pi_{\vec{n}}^{-1}\gamma_{p,q,l}|_N$, however as pointed out in Fact 1, $\sigma^{-1}\gamma_{p,q,l}|_N=\sigma^{-1}\pi_{\vec{n}}|_N\pi_{\vec{n}}^{-1}\gamma_{p,q,l}|_N,$ which forces $\sigma^{-1}\pi_{\vec{n}}|_N$ to be the identity or equivalently $\sigma^{-1}\pi_{\vec{n}}$ separates $N$. \\
\textbf{Fact 4.} $\pi_{\vec{n}}^{-1}\gamma_{p,q,l} \leq \sigma^{-1}\gamma_{p,q,l}$ provided $\sigma^{-1}\pi_{\vec{n}}$ separates $N$. Indeed, as showed in Fact 1, $\sigma^{-1}\gamma_{p,q,l}|_N = \pi_{\vec{n}}^{-1}\gamma_{p,q,l}|_N$ and $\pi_{\vec{n}}^{-1}\gamma_{p,q,l}(i)=i$ for all $i\notin N$. It follows immediately $\pi_{\vec{n}}^{-1}\gamma_{p,q,l} \leq \sigma^{-1}\gamma_{p,q,l}$.\\
\textbf{Fact 5} There exist a unique $\pi\in S_{r+s+t}$ such that $\sigma^{-1}\pi_{\vec{n}}$ separates $N$. Indeed, if $\pi$ is such that $\sigma^{-1}\pi_{\vec{n}}$ separates $N$ then as proved in Fact 1, $\sigma^{-1}\gamma_{p,q,l}|_N=\pi_{\vec{n}}^{-1}\gamma_{p,q,l}|_N.$ Therefore,
$$\sigma^{-1}\gamma_{p,q,l}|_N\psi =\pi_{\vec{n}}^{-1}\gamma_{p,q,l}\psi = \psi\pi^{-1}\gamma_{r,s,t},$$
where the last equality follows from Lemma \ref{Lemma: Some properties of pi_n}. The latter equation determines $\pi$ uniquely.\\
Once stated and proved the previous facts let us follow with the proof. First let $\pi\in S_{NC}(r,s,t)$ and $\sigma\in S_{NC}(p,q,l)$ be such that $\sigma\leq\pi_{\vec{n}}$ and $\sigma^{-1}\pi_{\vec{n}}$ separates $N$. By Fact 1 we have $\sigma^{-1}\gamma_{p,q,l}$ doesn't separate $N$. On the other hand since $\pi\in S_{NC}(r,s,t)$ then so is $\pi^{-1}\gamma_{r,s,t} \in S_{NC}(r,s,t)$ and therefore by Fact 2 as $\pi^{-1}\gamma_{r,s,t}\vee \gamma_{r,s,t}=1$ then $\mathbf{P}_{\sigma}=\{1,2,3\}$. Conversely, let $\sigma\in S_{NC}(p,q,l)$ be such that $\sigma^{-1}\gamma_{p,q,l}$ doesn't separate $N$ and $\mathbf{P}_{\sigma}=\{1,2,3\}$. We aim to show there exist a unique $\pi\in S_{NC}(r,s,t)$ such that $\sigma\leq \pi_{\vec{n}}$ and $\sigma^{-1}\pi_{\vec{n}}$ separates $N$. We first prove existence, let $\delta\in S_{r+s+t}$ be given by $\delta(i)=j$ whenever $\sigma^{-1}\gamma_{p,q,l}|_N(n_1+\cdots+n_i)=n_1+\cdots+n_j$ and let $\pi=\gamma_{r,s,t}\delta^{-1}$. Since $\mathbf{P}_{\sigma}=\{1,2,3\}$ then $\sigma^{-1}\gamma_{p,q,l}|_N \vee \gamma_{p,q,l}|_N =1$, thus it follows by Lemma \ref{Lemma: Restriction to non-crossing is still non-crossing if connectivity is preserved} that $\sigma^{-1}\gamma_{p,q,l}|_N\in S_{NC}(N_1,N_2,N_3)$, where by $S_{NC}(N_1,N_2,N_3)$ we mean $S_{NC}(\gamma_{p,q,l}|_N)$. Thus $\delta \in S_{NC}(r,s,t)$ and then so is $\pi\in S_{NC}(r,s,t)$. By Fact 3 we know $\sigma^{-1}\pi_{\vec{n}}$ separates $N$ and hence by Fact 4, $\pi_{\vec{n}}^{-1}\gamma_{p,q,l}\leq \sigma^{-1}\gamma_{p,q,l}$. By Lemma \ref{Lemma: Less or Equal transitivity} it follows $\sigma \leq \pi_{\vec{n}}$. Finally uniqueness follows from Fact 5.

\end{proof}

\begin{lemma}\label{Proposition: Induction version 2}
$$\sum_{(\V,\pi)\in PS_{NC}^{(1)}(r,s,t)}\sum_{\substack{\sigma\in S_{NC}(p,q,l) \\ \sigma \lesssim^{(1)} \pi_{\vec{n}} \\ \sigma^{-1}\pi_{\vec{n}}\text{ separates }N \\ \sigma \vee \pi_{\vec{n}}=\V_{\vec{n}}}} \kappa_{\sigma}(\vec{a})=\sum_{\substack{\sigma\in S_{NC}(p,q,l) \\ \sigma^{-1}\gamma_{p,q,l}\text{ doesn't sep }N \\ \mathbf{P}_{\sigma}\text{ has two blocks}}}\kappa_{\sigma}(\vec{a}).$$
\end{lemma}

\begin{proof}

Let $(\V,\pi)\in PS_{NC}^{(1)}(r,s,t)$ and $\sigma\in S_{NC}(p,q,l)$ be such that $\sigma\lesssim ^{(1)} \pi_{\vec{n}}$, $\sigma^{-1}\pi_{\vec{n}}$ separates $N$ and $\sigma\vee \pi_{\vec{n}}=\V_{\vec{n}}$. By Fact 1 we have $\sigma^{-1}\gamma_{p,q,l}$ doesn't separate $N$. Since $(\V,\pi)\in PS_{NC}^{(1)}(r,s,t)$ then $\pi$ connects at most two cycles of $\gamma_{r,s,t}$, therefore it does $\pi^{-1}\gamma_{r,s,t}$ and then by Fact 2 it follows $\mathbf{P}_{\sigma}$ has two blocks. Conversely let $\sigma\in S_{NC}(p,q,l)$ be such that $\sigma^{-1}\gamma_{p,q,l}$ doesn't separate $N$ and $\mathbf{P}_{\sigma}$ has two blocks, assume with out loss of generality that $\mathbf{P}_{\sigma}=\{1\}\{2,3\}$. We aim to show there exist a unique $(\V,\pi)\in PS_{NC}^{(1)}(r,s,t)$ such that $\sigma\lesssim ^{(1)} \pi_{\vec{n}}$, $\sigma^{-1}\pi_{\vec{n}}$ separates $N$ and $\sigma\vee \pi_{\vec{n}}=\V_{\vec{n}}$. Let $\delta$ and $\pi$ be defined as in the proof of Proposition \ref{Proposition: Induction version 1}. Since $\mathbf{P}_{\sigma}=\{1\}\{2,3\}$ then $\sigma^{-1}\gamma_{p,q,l}|_N \vee \gamma_{p,q,l}|_N$ has two blocks, $N_1$ and $N_2\cup N_3$. By Lemma \ref{Lemma: Restriction to non-crossing is still non-crossing if connectivity is preserved generalized} it follows $\sigma^{-1}\gamma_{p,q,l}|_N\in \NC(N_1)\times S_{NC}(N_2,N_3)$ and therefore $\delta\in \NC(r)\times S_{NC}(s,t)$ and $\pi\in \NC(r)\times S_{NC}(s,t)$. Fact 3 shows $\sigma^{-1}\pi_{\vec{n}}$ separates $N$ and Fact 5 shows that $\pi$ is unique. Fact 4 shows $\pi_{\vec{n}}^{-1}\gamma_{p,q,l}\leq \sigma^{-1}\gamma_{p,q,l}$ and hence as $\pi_{\vec{n}}\in \NC(p)\times S_{NC}(q,l)$ by Corollary \ref{Corollary: Less or equal transivitiy general version} we get $\sigma\lesssim^{(1)} \pi_{\vec{n}}$. Let $\V_{\vec{n}}$ be defined by $\V_{\vec{n}}=\sigma\vee\pi_{\vec{n}}$, since $\sigma\lesssim^{(1)} \pi_{\vec{n}}$ then any block of $\V_{\vec{n}}$ is a cycle of $\pi_{\vec{n}}$ except by one block which is the union of two cycles of $\pi_{\vec{n}}$, say $\tilde{A}$ and $\tilde{B}$. One of the cycles say $\tilde{A}$ must lie in $[p]$ while the other $\tilde{B}$ must lie in $[p+1,p+q+l]$, otherwise the condition $\sigma\vee\gamma_{p,q,l}=1$ is not satisfied as $\sigma\leq \sigma\vee\pi_{\vec{n}}$. The blocks $\tilde{A}$ and $\tilde{B}$ of $\pi_{\vec{n}}$ correspond to blocks $A$ and $B$ of $\pi$ with $A\subset [r]$ and $B\subset [r+1,r+s+t]$. Let $\V$ be the partition of $[r+s+t]$ determined by $\V_{\vec{n}}$, this is: each block of $\V_{\vec{n}}$ that is also a cycle of $\pi_{\vec{n}}$ correspond to a cycle of $\pi$ which we let to be a block of $\V$ and the unique block of $\V$ which is not a cycle of $\pi$ is $A\cup B$. This proves $(\V,\pi)\in PS_{NC}^{(1)}(r,s,t)$. Finally the condition $\V_{\vec{n}}=\sigma\vee\pi_{\vec{n}}$ uniquely determines $\V_{\vec{n}}$ and then does $\V$.
\end{proof}

\begin{lemma}\label{Proposition: Induction version 3}
$$\sum_{(\V,\pi)\in PS_{NC}^{(2)}(r,s,t)\cup PS_{NC}^{(3)\prime}(r,s,t)}\sum_{\substack{\sigma\in S_{NC}(p,q,l) \\ \sigma \lesssim^{(2)} \pi_{\vec{n}} \\ \sigma^{-1}\pi_{\vec{n}}\text{ separates }N \\ \sigma \vee \pi_{\vec{n}}=\V_{\vec{n}}}} \kappa_{\sigma}(\vec{a})=\sum_{\substack{\sigma\in S_{NC}(p,q,l) \\ \sigma^{-1}\gamma_{p,q,l}\text{ doesn't sep }N \\ \mathbf{P}_{\sigma}=\{1\}\{2\}\{3\}}}\kappa_{\sigma}(\vec{a}).$$
\end{lemma}

\begin{proof}
Let $(\V,\pi)\in PS_{NC}^{(2)}(r,s,t)\cup {PS_{NC}^{(3)}}^\prime(r,s,t)$ and $\sigma\in S_{NC}(p,q,l)$ be such that $\sigma\lesssim ^{(2)} \pi_{\vec{n}}$, $\sigma^{-1}\pi_{\vec{n}}$ separates $N$ and $\sigma\vee \pi_{\vec{n}}=\V_{\vec{n}}$. By Fact 1 we have $\sigma^{-1}\gamma_{p,q,l}$ doesn't separate $N$. Since $(\V,\pi)\in PS_{NC}^{(2)}(r,s,t)\cup {PS_{NC}^{(3)}}^\prime(r,s,t)$ then $\pi$ connects no cycles of $\gamma_{r,s,t}$, therefore $\pi^{-1}\gamma_{r,s,t}$ connects no cycles of $\gamma_{r,s,t}$ and then by Fact 2 it follows $\mathbf{P}_{\sigma}=\{1\}\{2\}\{3\}$. Conversely let $\sigma\in S_{NC}(p,q,l)$ be such that $\sigma^{-1}\gamma_{p,q,l}$ doesn't separate $N$ and $\mathbf{P}_{\sigma}$ has three blocks. We aim to show there exist a unique $(\V,\pi)\in PS_{NC}^{(2)}(r,s,t)\cup {PS_{NC}^{(3)}}^\prime(r,s,t)$ such that $\sigma\lesssim ^{(2)} \pi_{\vec{n}}$, $\sigma^{-1}\pi_{\vec{n}}$ separates $N$ and $\sigma\vee \pi_{\vec{n}}=\V_{\vec{n}}$. Let $\delta$ and $\pi$ be defined as in the proof of Proposition \ref{Proposition: Induction version 1}. Since $\mathbf{P}_{\sigma}=\{1\}\{2\}\{3\}$ then $\sigma^{-1}\gamma_{p,q,l}|_N \vee\gamma_{p,q,l}|_N$ has three blocks $N_1,N_2$ and $N_3$ and therefore $\delta\in \NC(N_1)\times \NC(N_2)\times \NC(N_3)$, thus $\delta,\pi\in \NC(r)\times \NC(s)\times \NC(t)$. Fact 3 shows $\sigma^{-1}\pi_{\vec{n}}$ separates $N$ and Fact 5 shows $\pi$ is unique. Fact 4 shows $\pi_{\vec{n}}^{-1}\gamma_{p,q,l}\leq \sigma^{-1}\gamma_{p,q,l}$ and hence as $\pi_{\vec{n}}\in \NC(p)\times \NC(q)\times \NC(l)$ by Corollary \ref{Corollary: Less or equal transivitiy general version} we get $\sigma\lesssim^{(2)} \pi_{\vec{n}}$. Let $\V_{\vec{n}}$ be defined by $\V_{\vec{n}}=\sigma\vee\pi_{\vec{n}}$ which uniquely determines $\V_{\vec{n}}$ and then it does $\V$ with $\V$ being the partition of $[r+s+t]$ determined by $\V_{\vec{n}}$. We are reduced to show $(\V,\pi)\in PS_{NC}^{(2)}(r,s,t)\cup {PS_{NC}^{(3)}}^\prime(r,s,t)$. First note that since $\sigma^{-1}\pi_{\vec{n}}$ separates $N$ and $\sigma^{-1}\gamma_{p,q,l}$ doesn't separate $N$ then by Fact 1 $\pi\neq \gamma_{r,s,t}$. Since $\sigma\lesssim^{(2)} \pi_{\vec{n}}$ then each block of $\V_{\vec{n}}$ is a cycle of $\pi_{\vec{n}}$ except by either one block which is the union of three cycles of $\pi_{\vec{n}}$ or two blocks with each one being the union of two cycles of $\pi_{\vec{n}}$. Assume we are in the former case and let $\tilde{A},\tilde{B}$ and $\tilde{C}$ be the cycles of $\pi_{\vec{n}}$ that are in the same block of $\V_{\vec{n}}$. It must be each of $\tilde{A}$, $\tilde{B}$ and $\tilde{C}$ lie in each of $[p]$, $[p+1,p+q]$ and $[p+q+1,p+q+l]$, otherwise the condition $\sigma\vee\gamma_{p,q,l}=1$ is not satisfied because $\sigma\leq \sigma\vee\pi_{\vec{n}}$. Suppose $\tilde{A}\subset [p],\tilde{B}\subset[p+1,p+q]$ and $\tilde{C}\subset [p+q+1,p+q+l]$. If $A,B$ and $C$ are the cycles of $\pi$ corresponding to $\tilde{A},\tilde{B}$ and $\tilde{C}$ respectively then $A\subset [r],B\subset [r+1,r+s]$ and $C\subset [r+s+1,r+s+t]$ which proves $(\V,\pi)\in {PS_{NC}^{(3)}}^\prime(r,s,t)$. The case where each of the two blocks of $\V_{\vec{n}}$ is the union of two cycles of $\pi_{\vec{n}}$ follows similarly, and in this case we get $(\V,\pi)\in PS_{NC}^{(2)}(r,s,t)$.
\end{proof}

\begin{lemma}\label{Proposition: Induction version 4}
$$\sum_{(\V,\pi)\in PS_{NC}^{(1)}(r,s,t)}\sum_{\substack{(\U,\sigma)\in \PS_{NC}^{(1)}(p,q,l) \\ \sigma \leq \pi_{\vec{n}} \\ \sigma^{-1}\pi_{\vec{n}}\text{ separates }N \\ \U\vee\pi_{\vec{n}}=\V_{\vec{n}}}} \kappa_{(\U,\sigma)}(\vec{a})=\sum_{\substack{(\U,\sigma)\in PS_{NC}^{(1)}(p,q,l) \\ \sigma^{-1}\gamma_{p,q,l}\text{ doesn't sep }N \\ \mathbf{P}_{\sigma}\text{ has two blocks}}}\kappa_{(\U,\sigma)}(\vec{a}).$$
\end{lemma}
\begin{proof} 
Let $(\V,\pi)\in PS_{NC}^{(1)}(r,s,t)$ and $(\U,\sigma)\in PS_{NC}^{(1)}(p,q,l)$ be such that $\sigma\leq \pi_{\vec{n}}$, $\sigma^{-1}\pi_{\vec{n}}$ separates $N$ and $\U\vee\pi_{\vec{n}}=\V_{\vec{n}}$. By Fact 1 we have $\sigma^{-1}\gamma_{p,q,l}$ separates $N$. On the other hand, as $(\V,\pi)\in PS_{NC}^{(1)}(r,s,t)$ then $\pi\vee\gamma_{r,s,t}$ has two blocks, suppose without loss of generality $\pi \in \NC(r)\times S_{NC}(s,t)$ and then so is $\pi^{-1}\gamma_{r,s,t}\in \NC(r)\times S_{NC}(s,t)$, thus by Fact 2 it follows $\mathbf{P}_{\sigma}=\{1\}\{2,3\}$. Conversely, let $(\U,\sigma)\in PS_{NC}^{(1)}(p,q,l)$ be such that $\sigma^{-1}\gamma_{p,q,l}$ doesn't separate $N$ and $\mathbf{P}_{\sigma}$ has two blocks. We aim to show there exist a unique $(\V,\pi)\in PS_{NC}^{(1)}(r,s,t)$ such that $\sigma\leq \pi_{\vec{n}}$, $\sigma^{-1}\pi_{\vec{n}}$ separates $N$ and $\U\vee\pi_{\vec{n}}=\V_{\vec{n}}$. Suppose with out loss of generality that $\sigma\in \NC(p)\times S_{NC}(q,l)$. Let $\delta$ and $\pi$ as in proof of Proposition \ref{Proposition: Induction version 1}. Observe that $\sigma^{-1}\gamma_{p,q,l}\in \NC(p)\times S_{NC}(q,l)$, so we can write $\sigma^{-1}\gamma_{p,q,l}$ as $(\sigma^{-1}\gamma_{p,q,l})^1\times (\sigma^{-1}\gamma_{p,q,l})^2$ with $(\sigma^{-1}\gamma_{p,q,l})^1\in \NC(p)$ and $(\sigma^{-1}\gamma_{p,q,l})^2\in S_{NC}(q,l)$. Therefore when we restrict to $N$ we get that $(\sigma^{-1}\gamma_{p,q,l})^1|_N 
 \in \NC(N_1)$ and $(\sigma^{-1}\gamma_{p,q,l})^2|_N 
 \in S_{NC}(N_2,N_3)$. Observe that the case $(\sigma^{-1}\gamma_{p,q,l})^2|_N 
 \in \NC(N_2)\times \NC(N_3)$ cannot be possible as that would mean $\mathbf{P}_{\sigma}$ has $3$ blocks. We conclude $\sigma^{-1}\gamma_{p,q,l}|_N \in \NC(N_1)\times S_{NC}(N_2,N_3)$, hence $\delta,\pi\in \NC(r)\times S_{NC}(s,t)$. By Fact 3 it follows $\sigma^{-1}\pi_{\vec{n}}$ separates $N$ while by Fact 5 we get $\pi$ is unique. Since both $\sigma,\pi_{\vec{n}}\in \NC(p)\times S_{NC}(q,l)$ it follows by Lemma \ref{Lemma: Less or Equal transitivity general version when having the same type} that $\sigma\leq \pi_{\vec{n}}$. Let $\V_{\vec{n}}=\U\vee\pi_{\vec{n}}$, this uniquely determines $\V_{\vec{n}}$ and $\V$ with $\V$ being the partition on $[r+s+t]$ corresponding to $\V_{\vec{n}}$. We are reduced to show that $(\V,\pi)\in PS_{NC}^{(1)}(r,s,t)$. Any block of $\U$ is a cycle of $\sigma$ except by one block which is the union of two cycles of $\sigma$, say $a\subset [p]$ and $b\subset [p+1,p+q+l]$. Since $\sigma\leq \pi_{\vec{n}}$ then any cycle of $\sigma$ is contained in a cycle of $\pi_{\vec{n}}$. Let $\tilde{A}$ and $\tilde{B}$ be the cycles of $\pi_{\vec{n}}$ that contain $a$ and $b$ respectively. Let us remind that $\pi_{\vec{n}}\in \NC(p)\times S_{NC}(q,l)$, therefore any cycle of $\pi_{\vec{n}}$ is completely contained in either $[p]$ or $[p+1,p+q+l]$, this forces to $\tilde{A}\subset [p]$ and $\tilde{B}\subset [p+1,p+q+l]$. On the other hand, any block of $\U$ that is a cycle of $\sigma$ is contained in a cycle of $\pi_{\vec{n}}$. So, only the block $a\cup b$ of $\U$ is not completely contained in a cycle of $\pi_{\vec{n}}$ but rather $a\cup b\subset \tilde{A}\cup \tilde{B}$, thus, $\V_{\vec{n}}=\U\vee\pi_{\vec{n}}$ is such that any block of $\V_{\vec{n}}$ is a cycle of $\pi_{\vec{n}}$ except by one block which is the union of the two cycles of $\pi_{\vec{n}}$, $\tilde{A}$ and $\tilde{B}$. Let $A$ and $B$ be the cycles of $\pi$ that correspond to the cycles $\tilde{A}$ and $\tilde{B}$ of $\pi_{\vec{n}}$ respectively. Then any block of $\V$ is a cycle of $\pi$ except by one block which is the union of $A$ and $B$. Since $\tilde{A}\subset [p]$ and $\tilde{B}\subset [p+1,p+q+l]$ then $A\subset [r]$ and $B\subset [r+1,r+s+t]$ and therefore $(\V,\pi)\in PS_{NC}^{(1)}(r,s,t)$ as desired.
 \end{proof}

\begin{lemma}\label{Proposition: Induction version 5}
$$\sum_{(\V,\pi)\in PS_{NC}^{(2)}(r,s,t)\cup PS_{NC}^{(3)\prime}(r,s,t)}\sum_{\substack{(\U,\sigma)\in \PS_{NC}^{(1)}(p,q,l) \\ \sigma \lesssim^{(1)} \pi_{\vec{n}} \\ \sigma^{-1}\pi_{\vec{n}}\text{ separates }N \\ \U\vee \pi_{\vec{n}}=\V_{\vec{n}}}} \kappa_{(\U,\sigma)}(\vec{a})=\sum_{\substack{(\U,\sigma)\in PS_{NC}^{(1)}(p,q,l) \\ \sigma^{-1}\gamma_{p,q,l}\text{ doesn't sep }N \\ \mathbf{P}_{\sigma}=\{1\}\{2\}\{3\}}}\kappa_{(\U,\sigma)}(\vec{a}).$$
\end{lemma}
\begin{proof} Let $(\V,\pi)\in PS_{NC}^{(2)}(r,s,t)\cup {PS_{NC}^{(3)}}^\prime(r,s,t)$ and $(\U,\sigma)\in PS_{NC}^{(1)}(p,q,l)$ be such that $\sigma \lesssim^{(1)}\pi_{\vec{n}}$, $\sigma^{-1}\pi_{\vec{n}}$ separates $N$ and $\U\vee\pi_{\vec{n}}=\V_{\vec{n}}$. By Fact 1 we have $\sigma^{-1}\gamma_{p,q,l}$ separates $N$. On the other hand, as $\pi^{-1}\gamma_{r,s,t}\in \NC(r)\times \NC(s)\times \NC(t)$, hence by Fact 2 it follows $\mathbf{P}_{\sigma}=\{1\}\{2\}\{3\}$. Conversely, let $(\U,\sigma)\in PS_{NC}^{(1)}(p,q,l)$ be such that $\sigma^{-1}\gamma_{p,q,l}$ doesn't separate $N$ and $\mathbf{P}_{\sigma}=\{1\}\{2\}\{3\}$. We will show there exist a unique $(\V,\pi)\in PS_{NC}^{(2)}(r,s,t)\cup {PS_{NC}^{(3)}}^\prime(r,s,t)$ such that $\sigma \lesssim^{(1)}\pi_{\vec{n}}$, $\sigma^{-1}\pi_{\vec{n}}$ separates $N$ and $\U\vee\pi_{\vec{n}}=\V_{\vec{n}}$. Suppose without loss of generality $\sigma\in \NC(p)\times S_{NC}(q,l)$. We let $\delta$ and $\pi$ as before. Proceeding as in the proof of Proposition \ref{Proposition: Induction version 4} we get $\sigma^{-1}\gamma_{p,q,l}|_N \in \NC(N_1)\times \NC(N_2)\times \NC(N_3)$ as in this case the permutation $(\sigma^{-1}\gamma_{p,q,l})^2 |_N \in \NC(N_2)\times \NC(N_3)$ because $\mathbf{P}_{\sigma}$ has 3 blocks. Thus $\delta,\pi\in \NC(r)\times \NC(s)\times \NC(t)$. By Fact 3 we get $\sigma^{-1}\pi_{\vec{n}}$ separates $N$ and then by Fact 5 we know $\pi$ is unique. Now let us prove $\sigma \lesssim^{(1)}\pi_{\vec{n}}$. Let $\rho = \pi_{\vec{n}}^{-1}\gamma_{p,q,l}$. Since $\pi\in \NC(r)\times \NC(s)\times \NC(t)$ then $\pi_{\vec{n}}\in \NC(p)\times \NC(q)\times \NC(l)$ and then $\rho\in \NC(p)\times \NC(q)\times \NC(l)$. Let us remind that $\sigma\in \NC(p)\times S_{NC}(q,l)$, thus we can write $\sigma^{-1}\gamma_{p,q,l}$ as $\sigma_1^{-1}\gamma_1 \times \sigma_2^{-1}\gamma_2$ with $\sigma_1$ and $\sigma_2$ being the restriction of $\sigma$ to $[p]$ and $[p+1,p+q+l]$ respectively and $\gamma_1$ and $\gamma_2$ being the permutations $(1,\dots,p)$ and $(p+1,\dots,p+q)(p+q+1,\dots,p+q+l)$ respectively. It is clear $\sigma_1 \in \NC(p)=S_{NC}(\gamma_1)$ and $\sigma_2\in S_{NC}(q,l)=S_{NC}(\gamma_2)$. In the same way let $\rho_1$ and $\rho_2$ be the permutation $\rho$ restricted to $[p]$ and $[p+1,p+q+l]$ respectively, so that $\rho_1\in \NC(p)$ and $\rho_2\in \NC(q)\times \NC(l)$. By Fact 4 we know $\rho \leq \sigma^{-1}\gamma_{p,q,l}$ and therefore $\rho_i \leq \sigma_i^{-1}\gamma_i$ for $i=1,2$. By Lemma \ref{Lemma: Less or Equal transitivity general version when having the same type} we have $\sigma_1 \leq \gamma_1\rho_1^{-1}$, equivalently,
\begin{equation}\label{aux6}
|\sigma_1|+|\sigma_1^{-1}\gamma_1\rho_1^{-1}|=|\gamma_1\rho_1^{-1}| \\
\text{ and }\\
\#(\sigma_1 \vee \gamma_1\rho_1^{-1})=\#(\gamma_1\rho_1^{-1}).
\end{equation}
On the other hand, $\sigma_2$ must have a cycle that meets $[p+1,p+q]$ and $[p+q+1,p+q+l]$, therefore since $\gamma_2\rho_2^{-1}\in \NC(p)\times \NC(l)$ this cycle must meet more than one cycle of $\gamma_2\rho_2^{-1}$, thus $\#(\gamma_2\rho_2^{-1})-\#(\sigma_2\vee\gamma_2\rho_2^{-1})\geq 1$. By Lemma \ref{Lemma: Less or equal transitivity general version 1},
\begin{eqnarray*}
|\gamma_2\rho_2^{(1)}|+2 &\leq& |\gamma_2\rho_2^{(1)}|+2(\#(\gamma_2\rho_2^{-1})-\#(\sigma_2\vee\gamma_2\rho_2^{-1})) \\
&\leq &|\sigma_2|+|\sigma_2^{-1}\gamma_2\rho_2^{-1}| \\
&=&|\gamma_2\rho_2^{-1}|+2(\#(\rho_2\vee\gamma_2)-1)=|
\gamma_2\rho_2^{-1}|+2.
\end{eqnarray*}
The latter means that all above must be equality, this is,
\begin{equation}\label{aux7}
|\gamma_2\rho_2^{-1}|+2=|\sigma_2|+|\sigma_2^{-1}\gamma_2\rho_2^{-1}| \text{ and }\#(\gamma_2\rho_2^{-1})-\#(\sigma_2\vee\gamma_2\rho_2^{-1})= 1.
\end{equation}
Observe that since $\sigma_1,\rho_1$ and $\gamma_1$ act on the set $[p]$ while $\sigma_2,\rho_2$ and $\gamma_2$ act on the set $[p+1,p+q+l]$ then $\#(\sigma\vee\gamma_{p,q,l}\rho^{-1})=\#(\sigma_1\vee\gamma_1\rho_1^{-1})+\#(\sigma_2\vee\gamma_2\rho_2^{-1})$, $\#(\sigma)=\#(\sigma_1)+\#(\sigma_2)$, $\#(\sigma^{-1}\gamma_{p,q,l}\rho^{-1})=\#(\sigma_1^{-1}\gamma_1\rho_1^{-1})+\#(\sigma_2^{-1}\gamma_2\rho_2^{-1})$ and $\#(\gamma_{p,q,l}\rho^{-1})=\#(\gamma_1\rho_1^{-1})+\#(\gamma_2\rho_2^{-1})$. In terms of the length function the last three equations are $|\sigma|=|\sigma_1|+|\sigma_2|$, $|\sigma^{-1}\gamma_{p,q,l}\rho^{-1}|=|\sigma_1^{-1}\gamma_1\rho_1^{-1}|+|\sigma_2^{-1}\gamma_2\rho_2^{-1}|$ and $|\gamma_{p,q,l}\rho^{-1}|=|\gamma_1\rho_1^{-1}|+|\gamma_2\rho_2^{-1}|$. Combining these with Equations \ref{aux6} and \ref{aux7} yields
$$|\sigma|+|\sigma^{-1}\gamma_{p,q,l}\rho^{-1}|=|\gamma_{p,q,l}\rho^{-1}|+2,$$
and $\#(\gamma_{p,q,l}\rho^{-1})-\#(\sigma\vee\gamma_{p,q,l}\rho^{-1})=1$. Let $B_1,\dots,B_w$ be the blocks of $\sigma\vee\gamma_{p,q,l}\rho^{-1}$. \cite[Equation 2.9]{MN} says that for each block,
$$\#(\sigma|_{B_i})+\#(\sigma|_{B_i}^{-1}(\gamma_{p,q,l}\rho^{-1})|_{B_i})+\#((\gamma_{p,q,l}\rho^{-1})|_{B_i}) \leq |B_i|+2,$$
with equality if and only if $\sigma|_{B_i} \in S_{NC}((\gamma_{p,q,l}\rho^{-1})|_{B_i})$. Summing over $i$ yields
$$|\gamma_{p,q,l}\rho^{-1}|+2(\#(\gamma_{p,q,l}\rho^{-1})-\#(\sigma\vee\gamma_{p,q,l}\rho^{-1}))\leq |\sigma|+|\sigma^{-1}\gamma_{p,q,l}\rho^{-1}|,$$
with equality if and only if $\sigma \in \prod_{B\text{ block of }\sigma\vee\gamma_{p,q,l}\rho^{-1}}S_{NC}((\gamma_{p,q,l}\rho^{-1})|_B).$ But we just proved $|\sigma|+|\sigma^{-1}\gamma_{p,q,l}\rho^{-1}|=|\gamma_{p,q,l}\rho^{-1}|+2$ and $\#(\gamma_{p,q,l}\rho^{-1})-\#(\sigma\vee\gamma_{p,q,l}\rho^{-1})=1$, so the latter inequality must be equality meaning $\sigma \in \prod_{B\text{ block of }\sigma\vee\gamma_{p,q,l}\rho^{-1}}S_{NC}((\gamma_{p,q,l}\rho^{-1})|_B),$ which proves $\sigma\lesssim^{(1)}\gamma_{p,q,l}\rho^{-1}=\pi_{\vec{n}}$. Let $\V_{\vec{n}}=\U\vee\pi_{\vec{n}}$, this uniquely defines $\V_{\vec{n}}$ and so it does $\V$ with $\V$ being the partition on $[r+s+t]$ corresponding to $\V_{\vec{n}}$. We are reduced to prove that $(\V,\pi)\in PS_{NC}^{(2)}(r,s,t)\cup {PS_{NC}^{(3)}}^\prime(r,s,t)$. Let us explicitly describe $\U\vee\pi_{\vec{n}}$, we first describe $\sigma\vee\pi_{\vec{n}}$. Since $\sigma \lesssim^{(1)}\pi_{\vec{n}}$ then any block of $\sigma\vee\pi_{\vec{n}}$ is a cycle of $\pi_{\vec{n}}$ except by one block which is the union of two cycles of $\pi_{\vec{n}}$ which we may call $\tilde{A}$ and $\tilde{B}$. Note that as $\sigma\in \NC(p)\times S_{NC}(q,l)$ then it must have a cycle that meets $[p+1,p+q]$ and $[p+q+1,p+q+l]$, but $\sigma\leq \sigma\vee\pi_{\vec{n}}$, which implies that this cycle must lie inside some block of $\sigma\vee\pi_{\vec{n}}$. Any block of $\sigma\vee\pi_{\vec{n}}$ which is also a cycle of $\pi_{\vec{n}}$ is entirely contained in either $[p]$,$[p+1,p+q]$ or $[p+q+1,p+q+l]$, so we are forced to that cycle to be $\tilde{A}\cup \tilde{B}$. It follows that either $\tilde{A}\subset [p+1,p+q]$ and $\tilde{B}\subset [p+q+1,p+q+l]$ or the other way around, suppose without loss of generality we are in the former case. On the other hand, let us remind that any block of $\U$ is a cycle of $\sigma$ except by one block which is the union of two cycles of $\sigma$, say $a \subset [p]$ and $b \subset [p+1,p+q+l]$. $a$ is a cycle of $\sigma$ and hence it is contained in a block of $\sigma\vee\pi_{\vec{n}}$, this block must necessarily be a cycle of $\pi_{\vec{n}}$ as the block $\tilde{A}\cup\tilde{B}$ of $\sigma\vee\pi_{\vec{n}}$ is entirely contained in $[p+1,p+q+l]$. Let $\tilde{C}$ be the cycle of $\pi_{\vec{n}}$ that contains $a$, we clearly have $\tilde{C}\subset [p]$. Similarly, the cycle $b$ of $\sigma$ must be contained in a block of $\sigma\vee\pi_{\vec{n}}$, at this point we have two possible scenarios. In the first scenario this block is precisely $\tilde{A}\cup\tilde{B}$ or in the second scenario this block is another cycle of $\pi_{\vec{n}}$, say $\tilde{D}$ with $\tilde{D}\subset [p+1,p+q+l]$. Let us adress the second scenario first. Remind that any cycle of $\pi_{\vec{n}}$ is entirely contained in either $[p]$,$[p+1,p+q]$ or $[p+q+1,p+q+l]$, so either $\tilde{D}\subset [p+1,p+q]$ or $\tilde{D}\subset [p+q+1,p+q+l]$. Suppose $\tilde{D}\subset [p+1,p+q]$. We said before any block of $\sigma\vee\pi_{\vec{n}}$ is a cycle of $\pi_{\vec{n}}$ except by $\tilde{A}\cup\tilde{B}$ but at the same time we know any block of $\U$ is a cycle of $\sigma$ except by the block $a\cup b$ with $a\subset \tilde{C}$ and $b\subset \tilde{D}$, hence any block of $\U\vee\pi_{\vec{n}}$ is a cycle of $\pi_{\vec{n}}$ except by two blocks which are $\tilde{A}\cup\tilde{B}$ and $\tilde{C}\cup\tilde{D}$. If we let $A,B,C$ and $D$ to be the blocks of $\pi$ corresponding to $\pi_{\vec{n}}$ we know by construction that $A\subset [r+1,r+s], B\subset [r+s+1,r+s+t], C\subset [r]$ and $D\subset [r+1,r+s]$. This proves $(\V,\pi)\in PS_{NC}^{(2)}(r,s,t)$. Observe that if we instead assume $\tilde{D}\subset [p+q+1,p+q+l]$ we then get $\tilde{D}\subset [r+s+1,r+s+t]$ which doesn't change the conclusion. Finally, in the first scenario we get something even simpler. We said before any block of $\sigma\vee\pi_{\vec{n}}$ is a cycle of $\pi_{\vec{n}}$ except by $\tilde{A}\cup\tilde{B}$ but at the same time we know any block of $\U$ is a cycle of $\sigma$ except by the block $a\cup b$ with $a\subset \tilde{C}$ and $b\subset \tilde{A}\cup\tilde{B}$, hence any block of $\U\vee\pi_{\vec{n}}$ is a cycle of $\pi_{\vec{n}}$ except by the block $\tilde{A}\cup\tilde{B}\cup\tilde{C}$. If we let $A,B$ and $C$ to be as before then $A\subset [r+1,r+s],B\subset [r+s+1,r+s+t]$ and $C\subset [r]$ which proves that $(\V,\pi)\in PS_{NC}^{(3)}(r,s,t)$. Moreover we know $\sigma^{-1}\pi_{\vec{n}}$ separates $N$ and $\sigma^{-1}\gamma_{p,q,l}$ doesn't separate $N$, then by Fact 1 it follows $\pi\neq \gamma_{r,s,t}$ which proves $(\V,\pi)\in {PS_{NC}^{(3)}}^\prime(r,s,t)$.
\end{proof}

\begin{lemma}\label{Proposition: Induction version 6}
$$\sum_{(\V,\pi)\in PS_{NC}^{(2)}(r,s,t)\cup PS_{NC}^{(3)\prime}(r,s,t)}\sum_{\substack{(\U,\sigma)\in \PS_{NC}^{(2)}(p,q,l) \\ \sigma \leq \pi_{\vec{n}} \\ \sigma^{-1}\pi_{\vec{n}}\text{ separates }N \\ \U\vee \pi_{\vec{n}}=\V_{\vec{n}}}} \kappa_{(\U,\sigma)}(\vec{a})=\sum_{\substack{(\U,\sigma)\in PS_{NC}^{(2)}(p,q,l) \\ \sigma^{-1}\gamma_{p,q,l}\text{ doesn't sep }N}}\kappa_{(\U,\sigma)}(\vec{a}).$$
\end{lemma}

\begin{proof} Let $(\V,\pi)\in PS_{NC}^{(2)}(r,s,t)\cup {PS_{NC}^{(3)}}^\prime(r,s,t)$ and $(\U,\sigma)\in PS_{NC}^{(2)}(p,q,l)$ be such that $\sigma \leq\pi_{\vec{n}}$, $\sigma^{-1}\pi_{\vec{n}}$ separates $N$ and $\U\vee\pi_{\vec{n}}=\V_{\vec{n}}$. By Fact 1 we have $\sigma^{-1}\gamma_{p,q,l}$ doesn't separate $N$. Conversely, let $(\U,\sigma)\in PS_{NC}^{(2)}(p,q,l)$ be such that $\sigma^{-1}\gamma_{p,q,l}$ doesn't separate $N$. We will show there exist a unique $(\V,\pi)\in PS_{NC}^{(2)}(r,s,t)\cup {PS_{NC}^{(3)}}^\prime(r,s,t)$ such that $\sigma \leq\pi_{\vec{n}}$, $\sigma^{-1}\pi_{\vec{n}}$ separates $N$ and $\U\vee\pi_{\vec{n}}=\V_{\vec{n}}$. Let $\delta$ and $\pi$ as before. Since $\sigma^{-1}\gamma_{p,q,l} \in \NC(p)\times \NC(q)\times \NC(l)$ then $\sigma^{-1}\gamma_{p,q,l}|_N \in \NC(N_1)\times \NC(N_2)\times \NC(N_3)$, and therefore $\delta,\pi\in \NC(N_1)\times \NC(N_2)\times \NC(N_3)$. By Fact 3 we get $\sigma^{-1}\pi_{\vec{n}}$ separates $N$ and then by Fact 5 $\pi$ is unique. Let us prove that $\sigma \leq \pi_{\vec{n}}$. By Fact 4 we know $\pi_{\vec{n}}^{-1}\gamma_{p,q,l} \leq \sigma^{-1}\gamma_{p,q,l}$, moreover $\pi_{\vec{n}}^{-1}\gamma_{p,q,l},\sigma\in \NC(p)\times \NC(q)\times \NC(l)$ then by Lemma \ref{Lemma: Less or Equal transitivity general version when having the same type} it follows $\sigma\leq \pi_{\vec{n}}$. Let $\V_{\vec{n}}=\U\vee\pi_{\vec{n}}$ which uniquely defines both $\V_{\vec{n}}$ and $\V$ with $\V$ as in previous proofs. We are reduced to show $(\V,\pi)\in PS_{NC}^{(2)}(r,s,t)\cup {PS_{NC}^{(3)}}^\prime(r,s,t)$. Since $(\U,\sigma)\in \PS_{NC}^{(2)}(p,q,l)$ any block of $\U$ is a cycle of $\sigma$ except by two blocks which are each the union of two cycles of $\pi$. Let $a,b,c,d$ be the cycles of $\pi$ such that $a\cup b$ and $c\cup d$ are both blocks of $\U$ with $a,c\subset [p]$, $b\subset [p+1,p+q]$ and $d\subset [p+q+1,p+q+l]$. Since $\sigma\leq \pi_{\vec{n}}$ then any cycle of $\sigma$ is contained in a cycle of $\pi_{\vec{n}}$ let $\tilde{A},\tilde{B},\tilde{C},\tilde{D}$ be the cycles of $\pi_{\vec{n}}$ with $a\subset \tilde{A},b\subset \tilde{B},c\subset \tilde{C}$ and $d\subset \tilde{D}$. Remind that any cycle of $\pi_{\vec{n}}$ is entirely contained in either $[p]$,$[p+1,p+q]$ or $[p+q+,1p+q+l]$. That forces $\tilde{A}\subset [p], \tilde{B}\subset [p+1,p+q],\tilde{C}\subset [p]$ and $\tilde{D}\subset [p+q+1,p+q+l]$. The latter means that possibly only $\tilde{A}$ and $\tilde{C}$ might be the same cycle. Suppose first $\tilde{A}\neq=\tilde{C}$. Since $a\cup b$ is a block of $\U$ then it must be contained in a block of $\U\vee\pi_{\vec{n}}$, moreover we know $a\subset \tilde{A}$ and $b\subset \tilde{B}$, thus this block must contain $\tilde{A}\cup\tilde{B}$. Similarly $\tilde{C}\cup \tilde{D}$ must be contained in a block of $\U\vee\pi_{\vec{n}}$. Any other block of $\U$, is also a cycle of $\pi_{\vec{n}}$, so we can conclude that the blocks of $\U\vee\pi_{\vec{n}}$ are the cycles of $\pi_{\vec{n}}$ except by the two blocks $\tilde{A}\cup\tilde{B}$ and $\tilde{C}\cup \tilde{D}$ which is each one the union of two cycles of $\pi_{\vec{n}}$. Let $A,B,C,D$ be the cycles of $\pi$ that correspond to the cycles $\tilde{A},\tilde{B},\tilde{C},\tilde{D}$ of $\pi_{\vec{n}}$. Then $A,C\subset [r], B\subset [r+1,r+s]$ and $D\subset [r+s+1,r+s+t]$ which proves $(\V,\pi)\in PS_{NC}^{(2)}(r,s,t)$. Now let us suppose the other case in which $\tilde{A}=\tilde{C}$. By the same argument, since $a\cup b$ and $c\cup d$ are both blocks of $\U$, must be contained in a block of $\U\vee\pi_{\vec{n}}$, such a block then must contain $\tilde{A}\cup\tilde{B}\cup\tilde{D}$. Any other block of $\U$ is a cycle of $\pi_{\vec{n}}$, thus any block of $\U\vee\pi_{\vec{n}}$ is a cycle of $\pi_{\vec{n}}$ except by $\tilde{A}\cup\tilde{B}\cup\tilde{D}$. If $A,B,D$ are the cycles of $\pi$ defined as before then $(\V,\pi)\in PS_{NC}^{(3)}(r,s,t)$. Finally, we know $\sigma^{-1}\pi_{\vec{n}}$ separates $N$ and $\sigma^{-1}\gamma_{p,q,l}$ doesn't separate $N$, then by Fact 1 it follows $\pi\neq \gamma_{r,s,t}$ which proves $(\V,\pi)\in {PS_{NC}^{(3)}}^\prime(r,s,t)$.
\end{proof}

\begin{lemma}\label{Proposition: Induction version 7}
$$\sum_{(\V,\pi)\in PS_{NC}^{(3)\prime}(r,s,t)}\sum_{\substack{(\U,\sigma)\in \PS_{NC}^{(3)}(p,q,l) \\ \sigma \leq \pi_{\vec{n}} \\ \sigma^{-1}\pi_{\vec{n}}\text{ separates }N \\ \U\vee \pi_{\vec{n}}=\V_{\vec{n}}}} \kappa_{(\U,\sigma)}(\vec{a})=\sum_{\substack{(\U,\sigma)\in PS_{NC}^{(3)}(p,q,l) \\ \sigma^{-1}\gamma_{p,q,l}\text{ doesn't sep }N}}\kappa_{(\U,\sigma)}(\vec{a}).$$
\end{lemma}

\begin{proof} Let $(\V,\pi)\in {PS_{NC}^{(3)}}^\prime(r,s,t)$ and $(\U,\sigma)\in PS_{NC}^{(3)}(p,q,l)$ be such that $\sigma \leq\pi_{\vec{n}}$, $\sigma^{-1}\pi_{\vec{n}}$ separates $N$ and $\U\vee\pi_{\vec{n}}=\V_{\vec{n}}$. By Fact 1 we have $\sigma^{-1}\gamma_{p,q,l}$ doesn't separate $N$. Conversely, let $(\U,\sigma)\in PS_{NC}^{(3)}(p,q,l)$ be such that $\sigma^{-1}\gamma_{p,q,l}$ doesn't separate $N$. We will show there exist a unique $(\V,\pi)\in {PS_{NC}^{(3)}}^\prime(r,s,t)$ such that $\sigma \leq\pi_{\vec{n}}$, $\sigma^{-1}\pi_{\vec{n}}$ separates $N$ and $\U\vee\pi_{\vec{n}}=\V_{\vec{n}}$. Let $\delta$ and $\pi$ as before. Since $\sigma^{-1}\gamma_{p,q,l} \in \NC(p)\times \NC(q)\times \NC(l)$ then $\sigma^{-1}\gamma_{p,q,l}|_N \in \NC(N_1)\times \NC(N_2)\times \NC(N_3)$, and therefore $\delta,\pi\in \NC(N_1)\times \NC(N_2)\times \NC(N_3)$. By Fact 3 we get $\sigma^{-1}\pi_{\vec{n}}$ separates $N$ and then by Fact 5 $\pi$ is unique. Let us prove that $\sigma \leq \pi_{\vec{n}}$. By Fact 4 we know $\pi_{\vec{n}}^{-1}\gamma_{p,q,l} \leq \sigma^{-1}\gamma_{p,q,l}$, moreover $\pi_{\vec{n}}^{-1}\gamma_{p,q,l},\sigma\in \NC(p)\times \NC(q)\times \NC(l)$ then by Lemma \ref{Lemma: Less or Equal transitivity general version when having the same type} it follows $\sigma\leq \pi_{\vec{n}}$. Let $\V_{\vec{n}}=\U\vee\pi_{\vec{n}}$ which uniquely defines both $\V_{\vec{n}}$ and $\V$ with $\V$ as in previous proofs. We are reduced to show $(\V,\pi)\in {PS_{NC}^{(3)}}^\prime(r,s,t)$. Since $(\U,\sigma)\in \PS_{NC}^{(3)}(p,q,l)$ any block of $\U$ is a cycle of $\sigma$ except by one block which is the union of three cycles of $\pi$. Let $a,b,c$ be the cycles of $\pi$ such that $a\cup b\cup c$ is this block of $\U$ with $a\subset [p]$, $b\subset [p+1,p+q]$ and $c\subset [p+q+1,p+q+l]$. Since $\sigma\leq \pi_{\vec{n}}$ then any cycle of $\sigma$ is contained in a cycle of $\pi_{\vec{n}}$ let $\tilde{A},\tilde{B},\tilde{C}$ be the cycles of $\pi_{\vec{n}}$ with $a\subset \tilde{A},b\subset \tilde{B}$ and $c\subset \tilde{C}$. Remind that any cycle of $\pi_{\vec{n}}$ is entirely contained in either $[p]$,$[p+1,p+q]$ or $[p+q+,1p+q+l]$. That forces $\tilde{A}\subset [p], \tilde{B}\subset [p+1,p+q]$ and $\tilde{C}\subset [p+q+1,p+q+l]$. Since $a\cup b\cup c$ is a block of $\U$ then it must be contained in a block of $\U\vee\pi_{\vec{n}}$, moreover we know $a\subset \tilde{A},b\subset \tilde{B}$ and $c\subset \tilde{C}$, thus this block must contain $\tilde{A}\cup\tilde{B}\cup\tilde{C}$. Any other block of $\U$, is also a cycle of $\pi_{\vec{n}}$, so we can conclude that the blocks of $\U\vee\pi_{\vec{n}}$ are the cycles of $\pi_{\vec{n}}$ except by the block $\tilde{A}\cup\tilde{B}\cup\tilde{C}$ which is the union of three cycles of $\pi_{\vec{n}}$. Let $A,B,C$ be the cycles of $\pi$ that correspond to the cycles $\tilde{A},\tilde{B},\tilde{C}$ of $\pi_{\vec{n}}$. Then $A\subset [r], B\subset [r+1,r+s]$ and $C\subset [r+s+1,r+s+t]$ which proves $(\V,\pi)\in PS_{NC}^{(3)}(r,s,t)$. Finally, we know $\sigma^{-1}\pi_{\vec{n}}$ separates $N$ and $\sigma^{-1}\gamma_{p,q,l}$ doesn't separate $N$, then by Fact 1 it follows $\pi\neq \gamma_{r,s,t}$ which proves $(\V,\pi)\in {PS_{NC}^{(3)}}^\prime(r,s,t)$.
\end{proof}

\end{document}